\documentclass[11pt,leqno,twoside]{amsart}

\usepackage{iftex}
\ifPDFTeX
  \usepackage[utf8]{inputenc} 
  \usepackage[T1]{fontenc} 
  \usepackage{lmodern} 
\else
  \usepackage{fontspec} 
\fi

\usepackage[scale=0.8, a4paper]{geometry}

\usepackage{graphicx}     
\usepackage{xcolor}       
\usepackage{enumitem}     
\usepackage{verbatim}     
\usepackage{multirow}

\usepackage{tikz}
\usepackage{pgfplots}
\usetikzlibrary{pgfplots.groupplots}
\usetikzlibrary{matrix}     
\usetikzlibrary{arrows, automata}
\pgfplotsset{compat=newest,compat/show suggested version=false}

\usepackage{amsmath, amsthm,amscd}
\usepackage{amsfonts, amssymb}
\usepackage{cases}         
\usepackage{bm, dsfont}               
\usepackage{hyperref}
\numberwithin{equation}{section}

\theoremstyle{plain}
	\newtheorem{theorem}{Theorem}[section]
	\newtheorem*{theorem*}{Theorem}
	\newtheorem{lemma}{Lemma}[section]
	
	\newtheorem*{proposition*}{Proposition}

	\renewcommand{\paragraph}[1]{\medskip \noindent \textbf{#1}.}

\theoremstyle{remark}
	\newtheorem{remark}{\textbf{Remark}}[]

\theoremstyle{definition}

\newcommand{\RR}{\mathbb{R}}
\newcommand{\NN}{\mathbb{N}}

\newcommand{\ZZ}{\mathbb{Z}}

\newcommand{\R}{{\amsmathbb R}}
\newcommand{\N}{{\amsmathbb N}}
\newcommand{\Z}{{\amsmathbb Z}}

\newcommand{\T}{{\amsmathbb T}}

\newcommand{\Sp}{{\amsmathbb S}}

\newcommand{\D}{{\partial}}

\DeclareSymbolFontAlphabet{\amsmathbb}{AMSb}

\title{On Deterministic Numerical Methods for the quantum Boltzmann-Nordheim Equation. \\ I. Spectrally accurate approximations, Bose-Einstein condensation, Fermi-Dirac saturation}
\author{Alexandre Mouton}
\address{CNRS, Univ. Lille, UMR 8524 – Laboratoire Paul Painlevé F-59000 Lille, France}
\email{alexandre.mouton@univ-lille.fr}

\author{Thomas Rey}
\address{Univ. Lille, CNRS, UMR 8524, Inria – Laboratoire Paul Painlevé F-59000 Lille, France }
\email{thomas.rey@univ-lille.fr}
\thanks{T.R. was partially funded by Labex CEMPI (ANR-11-LABX-0007-01) and ANR Project MoHyCon (ANR-17-CE40-0027-01).}

\begin{document}

\begin{abstract} Spectral methods, thanks to their high accuracy and the possibility to use fast algorithms, represent an effective way to approximate the collisional kinetic equations of Boltzmann type, such as the Boltzmann-Nordheim equation. 
	This equation, modeled on the seminal Boltzmann equation, describes using a statistical physics formalism the time evolution of a gas composed of bosons or fermions. 
	Using the spectral-Galerkin algorithm introduced in [F. Filbet, J. Hu, and S. Jin, \textit{ESAIM: Math. Model. Numer. Anal.}, 2011], together with some novel parallelization techniques, we investigate some of the conjectured properties of the  large time behavior of the solutions to this equation.
	In particular, we are able to observe numerically both Bose-Einstein condensation and Fermi-Dirac relaxation. \\[1em]
\end{abstract}
  \maketitle
  
  
  \section{Introduction}
The {\bf Boltzmann-Nordheim equation} (BNE), also known as the quantum Boltzmann equation, models the time evolution of a phase-space density $f= f(t,\mathbf x, \mathbf k)$, describing the probability to find at time $t \geq 0$ a quantum particle localized at the infinitesimal position $d\mathbf x \, d\mathbf k$, where $\mathbf x \in \Omega \subset \RR^3$ and $\mathbf k \in \RR^3$. The quantity $\mathbf k$ designates the energy level of the particle and corresponds to the usual ``velocity'' variable in a kinetic Boltzmann-like equation. Its total momentum is then given by $\mathbf v = \hbar \mathbf k$, where $\hbar$ is the reduced Planck constant, and we shall keep the $\mathbf v$ notation throughout the paper. We are interested in this work in efficient, deterministic numerical simulation of this equation.

This equation was first formulated by Uehling and Uhlenbeck in the seminal paper \cite{UehlingUhlenbeck:1933}, starting from a classical Boltzmann equation with heuristic arguments.  It can be used to model both bosons and fermions gases, possibly trapped by a confining potential.
Due to its high dimensionality, the study of the full quantum Boltzmann equation is still widely open. The equation is indeed posed in the usual six-dimensional kinetic phase space. 
Moreover, the collision operator is defined by a seven-fold integral itself. This causes a lot of difficulty, both theoretically and numerically.
For example, in the space homogeneous  bosonic case, it is well known that the particle distribution function $f$ can develop finite time blow-up (weak convergence towards Dirac deltas even if the kinetic energy is conserved), the so-called \textbf{Bose-Einstein condensation}.

The study of the main mathematical properties of the collision operator has been done in \cite{dolbeault1994kinetic,EscobedoMischler:2001,Lu2014,BriantEinav2013}, allowing to understand almost completely the Cauchy problem (existence, stability, creation and propagation of moments, convergence towards the equilibrium) in the space homogeneous setting, without confining potential. The recent breakthrough \cite{EscobedoVelazquez2015} proved rigorously the Bose-Einstein condensation  in the bosonic case, under very mild hypotheses on the moments of the initial condition. 
Nevertheless, the precise blowup behavior (localization of the explosion, rates, etc.) is still open.
Finally, the theoretical study of the full space dependent problem is also mostly open, apart from some recent results concerning the anisotropic setting and Bose-Einstein condensates that can be found in \cite{li2017global}.

The numerical study of this equation is also difficult for all these reasons. The first attempt to compute numerically the collision operator, in a simplified setting (taking advantage of some very specific symmetry properties) was done in the work \cite{PareschiMarkowich:2005}. Extensions of this idea were introduced in \cite{HuYi11} by exploiting the convolution-like structure of the collision operator. This work is following along the lines of the so-called classical and fast spectral methods for the Boltzmann equation, that was investigated by many authors in the series of works \cite{PP96,MoPa:2006,FiMoPa:2006,FM11,FilbetPareschiRey:2015,PareschiRey2020,hu2020new,pareschi2021moment}.
The full extent of this convolution idea was used in \cite{FilbetHuJin:2012}, allowing to write a fast (spectral) method able to compute the full collision operator with the ``reasonably low'' numerical cost of $\mathcal{O}(N^5 \log(N))$ operations, where $N$ is the number of unknowns in each velocity dimension, and with spectral accuracy. 
Extension of this method to the full inhomogeneous case, and to the asymptotic preserving setting was then done in the series of works \cite{HuYing:2015,Hu2015,Hu2015a} for the fermionic (electrons) case only. Nevertheless, these works only deal with the simplified 1-dimensional (1d) in $\mathbf x$, 2-dimensional in $\mathbf v$ setting and use an approximation of the collision operator. Even the slightly more realistic 1d in $\mathbf x$, 3d in $\mathbf v$ case was never tackled up to the authors knowledge, not to mention the full 3d in $\mathbf x$, 3d in $\mathbf v$ case.
One should also note the earlier attempts from \cite{JinPareschi2000, Klar1999}, where a diffusive relaxation system was used, as an approximation of the full equation, and the Fokker-Planck like approach recently seen in \cite{CarrilloHopfWolfram2020}.	

The paper is organized as follows: After a small introduction on the BNE equation and its main mathematical features, we detail in Section \ref{sec:stationary_state} its large time behavior, which depends on the type of considered particles: classical ones, Bosons, or Fermions, and on some critical parameters which insure condensation in the latter cases. We then present the fast algorithm from \cite{FilbetHuJin:2012} along with how we decided to implement it in Section \ref{sec:FBNEA}. Some more information on this topic is also available in the Appendix. Section \ref{sec:Rescaling} is then devoted to the presentation of the so-called rescaling-velocity method, which allows to improve the accuracy of the fast spectral algorithm. Section \ref{sec:NumRes} present some numerical results on the method, emphasizing first on the spectral accuracy of the algorithm in Section \ref{sec:SpecAccu}, and then on its long time behavior in Section \ref{sec:NumTrend}, for the both 2D and 3D fermions and bosons cases.

\subsection{The Boltzmann-Nordheim equation}

We are then interested in the following kinetic equation in the $d_{\mathbf{x}} \times d_{\mathbf{v}}+1$-dimensional phase space, for $d_\mathbf{x} \leq d_\mathbf{v} \leq 3$:
\begin{equation} \label{Boltzmann_start}
	\left \{ \begin{aligned}
		& \D_{t}f + \tau \, \mathbf{v} \cdot \nabla_{\mathbf{x}}f = c\,\mathcal{Q}_\alpha(f) \, , \\
		& f(0,\mathbf x, \mathbf v) =  f_0(\mathbf x, \mathbf v),
	\end{aligned} \right.
\end{equation}
where 
\begin{itemize}
\item $\mathbf{x} \in \Omega_{\mathbf{x}} \subset \R^{d_{\mathbf{x}}}$ is the dimensionless position variable,
\item $\mathbf{v} \in \R^{d_{\mathbf{v}}}$ is the dimensionless velocity variable,
\item $t \in \R_{+}$ is the dimensionless time variable,
\item $f:\R_{+} \times \Omega_{\mathbf{x}} \times \R^{d_{\mathbf{v}}} \to \R$ is the distribution function of the particles,
\item $\tau \in \{0,1\}$ is a dimensionless parameter for indicating if the transport part is taken into account or not,
\item $c \geq 0$ is a dimensionless scaling parameter for the collisions (usually taken as the inverse of the characteristic Knudsen number),
\item $\mathcal{Q}_\alpha(f):\R_{+}\times\Omega_{\mathbf{x}}\times\R^{d_{\mathbf{v}}} \to \R$ is the collision operator, and $\alpha$ characterizes the type of particles, to be specified later.
\end{itemize}
In the present case, the collision operator $\mathcal{Q}_\alpha(f)$ is defined as follows:
\begin{multline}
\label{eq:BNE_OP}
\mathcal{Q}_\alpha(f)(t,\mathbf{x},\mathbf{v})
:= \int_{\R^{d_{\mathbf{v}}} \times \Sp^{d_{\mathbf{v}}-1}} B\left(|\mathbf{v}-\mathbf{v}_{*}|,\theta\right) \, \Big[ f(t,\mathbf{x},\mathbf{v}')\,f(t,\mathbf{x},\mathbf{v}_{*}')\,\left(1-\alpha \, f(t,\mathbf{x},\mathbf{v}) - \alpha\, f(t,\mathbf{x},\mathbf{v}_{*})\right) \\
 - f(t,\mathbf{x},\mathbf{v})\, f(t,\mathbf{x},\mathbf{v}_{*}) \, \left(1-\alpha\,f(t,\mathbf{x},\mathbf{v}') - \alpha\,f(t,\mathbf{x},\mathbf{v}_{*}')\right)\Big] \, d\mathbf{v}_{*}\,d\bm{\sigma} \, ,
\end{multline}
with $(\mathbf{v}',\mathbf{v}_{*}',\theta) \in \R^{d_{\mathbf{v}}} \times \R^{d_{\mathbf{v}}} \times [0,\pi]$ linked to $(\mathbf{v},\mathbf{v}_{*},\bm{\sigma}) \in \R^{d_{\mathbf{v}}} \times \R^{d_{\mathbf{v}}} \times \Sp^{d_{\mathbf{v}}-1}$ by the relations
\begin{equation*} 
\label{changevar_sigma}
\left\{
\begin{array}{rcl}
\mathbf{v}' &=& \cfrac{\mathbf{v}+\mathbf{v}_{*}}{2} + \cfrac{|\mathbf{v}-\mathbf{v}_{*}|}{2} \, \bm{\sigma} \, , \\
\mathbf{v}_{*}' &=& \cfrac{\mathbf{v}+\mathbf{v}_{*}}{2} - \cfrac{|\mathbf{v}-\mathbf{v}_{*}|}{2} \, \bm{\sigma} \, ,
\end{array}
\right.
\end{equation*}
where 
\[\cos\theta = \cfrac{\mathbf{v}-\mathbf{v}_{*}}{|\mathbf{v}-\mathbf{v}_{*}|}\cdot \bm{\sigma}, \] 
and the collision kernel $B:\R^{d_{\mathbf{v}}}\times [0,\pi] \to \R$ of variable hard sphere (VHS) type:
\begin{equation*} 
B(r,\theta) = \Phi(r) \, b(\cos\theta) \, , \qquad \Phi(r) = C_{\Phi}\,r^{\gamma} \, ,
\end{equation*}
$C_{\Phi} > 0$, $\gamma \in [0,1]$ being given constants. Finally, $\alpha$ is defined as follows:
\begin{equation*}
\alpha = \left\{
\begin{array}{ll}
+\hbar^{d_{\mathbf{v}}} \, , & \textnormal{for Fermi-Dirac particles,} \\
-\hbar^{d_{\mathbf{v}}} \, , & \textnormal{for Bose-Einstein particles.}
\end{array}
\right.
\end{equation*}

\begin{remark}
	Note that the case $\gamma = 0$ and $d_\mathbf{v} =2$ represents the toy model of maxwellian pseudo-molecules, whereas $\gamma = 1$ with $d_\mathbf{v} = 3$ models the physically relevant case of a hard sphere collision dynamics.
\end{remark}

We also assume that the microscopic dynamics is subject to the so-called Grad's cut-off assumption, namely that the function $b:[-1,1]\to\R_{+}$ is such that $\|b\|_{L^{\infty}(-1,1)} < +\infty$.

\begin{remark}
	\label{rem:Quadri2Tri}
	The Boltzmann-Nordheim collision operator, although being quadrilinear in its full form \eqref{eq:BNE_OP}, is actually only trilinear in the two cases of interests $\alpha \in \{0,1\}$ using cancellations in \eqref{eq:weakForm}	.
\end{remark}	

\paragraph{Weak form and macroscopic properties} 
Note that, if we define $\bm{\omega} = \cfrac{\mathbf{v}-\mathbf{v}'}{|\mathbf{v}-\mathbf{v}'|}$, we have
\begin{equation*}
\left\{
\begin{array}{rcl}
\mathbf{v}' &=& \mathbf{v} - \left( (\mathbf{v}-\mathbf{v}_{*})\cdot\bm{\omega}\right)\, \bm{\omega} \, , \\
\mathbf{v}_{*}' &=& \mathbf{v}_{*} + \left( (\mathbf{v}-\mathbf{v}_{*})\cdot\bm{\omega}\right)\,\bm{\omega} \, ,
\end{array}
\right.
\qquad \left\{
\begin{array}{rcl}
\mathbf{v} &=& \mathbf{v}' - \left( (\mathbf{v}'-\mathbf{v}_{*}')\cdot\bm{\omega}\right)\,\bm{\omega} \, , \\
\mathbf{v}_{*} &=& \mathbf{v}_{*}' + \left( (\mathbf{v}'-\mathbf{v}_{*}')\cdot\bm{\omega}\right)\,\bm{\omega} \, ,
\end{array}
\right.
\end{equation*}
meaning that the transformation $(\mathbf{v},\mathbf{v}_{*}) \mapsto (\mathbf{v}',\mathbf{v}_{*}')$ is involutive. Thanks to this remark, one  has the following weak form for the collision operator: for any test function $\Psi$ such that the left hand side is defined,
\begin{multline}
\label{eq:weakForm}
		\int_{\R^{d_{\mathbf{v}}}} \mathcal{Q}_\alpha(f)(t,\mathbf{x},\mathbf{v})\, \psi(\mathbf{v})\, d\mathbf{v} = \int_{\R^{d_{\mathbf{v}}} \times \R^{d_{\mathbf{v}}} \times \Sp^{d_{\mathbf{v}}-1}} B\left(|\mathbf{v}-\mathbf{v}_{*}|,\theta\right) \\ \left[ f'\,f_{*}'\left(1-\alpha f - \alpha f_{*} \right) - f\, f_{*} \left(1-\alpha f' - \alpha f_{*}' \right)\right]  \left[ \Psi(\mathbf v_*') + \Psi(\mathbf v') - \Psi(\mathbf v) - \Psi(\mathbf v_*) \right] \, d\mathbf{v}_{*}\, d \mathbf{v} \, d\bm{\sigma},
\end{multline}
where  we used the shorthand notations $f = f(t,\mathbf{x},\mathbf{v})$, $f_* = f(t,\mathbf{x},\mathbf{v}_*)$,
$f ^{'} = f(t,\mathbf{x},\mathbf{v}')$, $f_* ^{'} = f(t,\mathbf{x},\mathbf{v}_* ^{'})$.

In particular, one has using \eqref{changevar_sigma} that the collisior operator preserves the global mass, momentum and kinetic energy:
\begin{equation}
\label{eq:conservationQ}
\int_{\R^{d_{\mathbf{v}}}} \mathcal{Q}_\alpha(f)(t,\mathbf{x},\mathbf{v})\, \psi(\mathbf{v})\, d\mathbf{v} = 0 \, ,
\end{equation}
for $\psi(\mathbf{v}) = 1$, $\psi(\mathbf{v}) = \mathbf{v}$, and $\psi(\mathbf{v}) = |\mathbf{v}|^{2}$.

Finally, let us assume that the distribution function $f$ satifies the following properties:
\begin{itemize}
\item $f(t,\mathbf{x},\mathbf{v}) > 0$ for any $t,\mathbf{x},\mathbf{v}$;
\item For the Fermi-Dirac case ($\alpha > 0$), $f(t,\mathbf{x},\mathbf{v}) < \cfrac{1}{\alpha}$ for any $t,\mathbf{x},\mathbf{v}$.
\end{itemize}
Under these assumptions, we can define the local entropy associated to $f$ as
\begin{equation*}
\mathcal{H}[f](t,\mathbf{x}) = -\int_{\R^{d_{\mathbf{v}}}} \left[f(t,\mathbf{x},\mathbf{v}) \, \ln\left(f(t,\mathbf{x},\mathbf{v})\right)
+ \cfrac{1-\alpha f(t,\mathbf{x},\mathbf{v})}{\alpha}\, \ln\left(1-\alpha f(t,\mathbf{x},\mathbf{v})\right)\right]\, d\mathbf{v} \, .
\end{equation*}
Taking $\psi(t,\mathbf{x},\mathbf{v}) = \ln\left(\cfrac{f(t,\mathbf{x},\mathbf{v})}{1-\alpha\,f(t,\mathbf{x},\mathbf{v})}\right)$ in \eqref{eq:weakForm}, we get the following expansion of the entropy dissipation functional:
\begin{align*}
\mathcal{D}[f](t,\mathbf{x}) & := \int_{\R^{d_{\mathbf{v}}}} \mathcal{Q}_\alpha(f)(t,\mathbf{x},\mathbf{v})\, \psi(t,\mathbf{x},\mathbf{v})\, d\mathbf{v} \\
&= -\cfrac{1}{4} \int_{\R^{d_{\mathbf{v}}}}\int_{\R^{d_{\mathbf{v}}}} \int_{\Sp^{d_{\mathbf{v}}-1}} B(|\mathbf{v}-\mathbf{v}_{*}|, \theta) \\
&\qquad \left[ f'\,f_{*}'\left(1-\alpha f - \alpha f_{*} \right) - f\, f_{*} \left(1-\alpha f' - \alpha f_{*}' \right)\right] \log \left(\cfrac{f\, f_{*} \left(1-\alpha f' - \alpha f_{*}' \right)}{f'\,f_{*}'\left(1-\alpha f - \alpha f_{*} \right)}\right) \, d\bm{\sigma}\, d\mathbf{v}_{*}\, d\mathbf{v} \, .
\end{align*}
In addition, we have
\begin{displaymath}
\int_{\R^{d_{\mathbf{v}}}} \D_{t}f(t,\mathbf{x},\mathbf{v}) \, \psi(t,\mathbf{x},\mathbf{v}) \, d\mathbf{v} = -\D_{t}\mathcal{H}[f](t,\mathbf{x},\mathbf{v}) \, ,
\end{displaymath}
and
\begin{displaymath}
\int_{\R^{d_{\mathbf{v}}}} \mathbf{v} \cdot \nabla_{\mathbf{x}} f(t,\mathbf{x},\mathbf{v}) \, \psi(t,\mathbf{x},\mathbf{v}) \, d\mathbf{v} = \nabla_{\mathbf{x}} \cdot \left( \cfrac{1}{\alpha}\, \int_{\R^{d_{\mathbf{v}}}} \mathbf{v} \, \left(\ln\left(1-\alpha\,f(t,\mathbf{x},\mathbf{v})\right)\right) \, d\mathbf{v} \right) \, .
\end{displaymath}
We finally deduce that the global entropy $\displaystyle \int_{\R^{d_{\mathbf{x}}}} \mathcal{H}[f](t,\mathbf{x})\, d\mathbf{x}$ increases as $t \to +\infty$ since
\begin{equation*}
\int_{\Omega_{\mathbf{x}}} \D_{t}\mathcal{H}[f](t,\mathbf{x})\, d\mathbf{x} = -\int_{\Omega_{\mathbf{x}}} \mathcal{D}[f](t,\mathbf{x})\, d\mathbf{x} \geq 0 \, , \qquad \forall \, t \geq 0 \, .
\end{equation*}

\subsection{Convergence towards a steady state}

\label{sec:stationary_state}

In this section, we focus on the homogeneous equation associated to \eqref{Boltzmann_start} that is characterized by $\tau = 0$:
\begin{equation} \label{Boltzmann_homo}
	\left \{ \begin{aligned}
		& \D_{t}f = c\,\mathcal{Q}_\alpha(f) \, , \\
		& f(0, \mathbf v) =  f^0(\mathbf v).
	\end{aligned} \right.
\end{equation}
 In such case, using \eqref{eq:conservationQ} the momenta of order 0, 1, and 2 of $f$ are constant:
\begin{equation}
\label{eq:moments}
\rho(t) = \rho^0 := \int_{\R^{d_{\mathbf{v}}}} f^0 \, d\mathbf{v} \, , \quad
\mathbf{u}(t) = \mathbf{u}^0 := \frac{1}{\rho^0}\, \int_{\R^{d_{\mathbf{v}}}} \mathbf{v}\, f^0\, d\mathbf{v} \, , \quad
e(t) = e^0 := \cfrac{1}{2\rho^0}\, \int_{\R^{d_{\mathbf{v}}}} |\mathbf{v}-\mathbf{u}^0|^{2}\, f^0 \, d\mathbf{v} \, .
\end{equation}
It is well known that the distribution function $f$ converges to a limit state $f^{\infty}:\R^{d_{\mathbf{v}}} \to \R_{+}$ that depends only on these $d_\mathbf{v} + 2$ quantities as $t\to+\infty$, see e.g. \cite{BriantEinav2013,EscobedoVelazquez2015,li2017global}. To characterize this limit state, we proceed as follows:

\paragraph{The classical case $\alpha = 0$} The limit state is the following maxwellian function \cite{CIP:94}
\begin{equation*}
f^{\infty}(\mathbf{v}) = \mathcal{M}_{c}(\mathbf{v}) = \cfrac{\rho^0}{(2\pi T)^{d_{\mathbf{v}}/2}}\, \exp\left(-\cfrac{|\mathbf{v}-\mathbf{u}^0|^{2}}{2T}\right) \, ,
\end{equation*}
where the temperature $T$ is obtained from the internal energy $e^0$ thanks to the relation
\begin{equation*}
T = \cfrac{2}{d_{\mathbf{v}}}\, e^0 \, .
\end{equation*}

\paragraph{The quantum case $\alpha \neq 0$} We introduce the so-called indicator function $\chi$
\begin{equation} \label{def_chi}
\chi(\mathbf{v}) = \cfrac{\rho^0}{|D(\mathbf{u}^{0},A)|} \, \mathds{1}_{D(\mathbf{u}^{0},A)}(\mathbf{v}) \, , \qquad A = \sqrt{\cfrac{2\, e^0 \, (d_{\mathbf{v}}+2)}{d_{\mathbf{v}}}} \, ,
\end{equation}
and the following quantum maxwellian function \cite{BriantEinav2013}
\begin{equation} \label{def_Mq_2d}
\mathcal{M}_{q}(\mathbf{v}) = \cfrac{1}{|\alpha|}\, \cfrac{1}{z^{-1} \exp\left (\frac{|\mathbf{v}-\mathbf{u}^0|^{2}}{2T}\right ) + \textnormal{sgn}(\alpha)} \, ,
\end{equation}
where $z$ and $T$ are linked to $\rho^0$ and $e^0$ as follows: since the mass and the internal energy are preserved we integrate $\mathcal{M}_{q}(\mathbf{v})$ and $|\mathbf{v}-\mathbf{u}^{0}|^{2}\,\mathcal{M}_{q}(\mathbf{v})$ to get
\begin{subnumcases}{\label{mom02_qmaxwell}}
\rho^0 = \cfrac{(2\pi T)^{\frac{d_{\mathbf{v}}}{2}}}{|\alpha|}\, K_{\frac{d_{\mathbf{v}}}{2}}(z) \, , & \label{mom0_qmaxwell}\\
e^0 = \cfrac{d_{\mathbf{v}}T}{2}\, \cfrac{K_{\frac{d_{\mathbf{v}}}{2}+1}(z)}{K_{\frac{d_{\mathbf{v}}}{2}}(z)} \, , & \label{mom2_qmaxwell}
\end{subnumcases}
where $K_{\nu}(z)$ is defined as
\begin{equation*}
K_{\nu}(z) = \cfrac{1}{\Gamma(\nu)}\, \int_{0}^{+\infty} \cfrac{x^{\nu-1}}{z^{-1}e^{x}+\textnormal{sgn}(\alpha)}\, dx = \left\{
\begin{array}{ll}
F_{\nu}(z) \, , & \textnormal{if $\alpha > 0$ (Fermi-Dirac),} \\
B_{\nu}(z) \, , & \textnormal{if $\alpha < 0$ (Bose-Einstein),}
\end{array}
\right.
\end{equation*}
and $\Gamma$ is the Gamma function
\begin{equation*}
\Gamma(\nu) = \int_{0}^{+\infty} x^{\nu-1}\, e^{-x}\, dx \, .
\end{equation*}
The quantities $F_{\nu}$ and $B_{\nu}$ are known as complete Fermi-Dirac and Bose-Einstein integrals respectively.

\begin{remark}
\label{rem:Fnu}
Setting $\tilde{F}_{\nu}(\mu) = F_{\nu}(e^{\mu})$ for any $\mu \in \R$ and $\nu > -1$, one can show (see e.g \cite{boersma1991asymptotic}) that
\begin{equation*}
\left \{
\begin{aligned}
& \tilde{F}_{\nu}(\mu) = e^{\mu} + o(1)\, , & \textnormal{as $\mu \to -\infty$,} \\
& \tilde{F}_{\nu}(\mu) = \cfrac{\mu^{\nu}}{\Gamma(\nu+1)} \, \left[ 1+\cfrac{\pi\,\nu\,(\nu-1)}{6\mu^{2}} + \mathcal{O}\left(\cfrac{1}{\mu^{4}}\right)\right] \, , & \textnormal{as $\mu \to +\infty$.}
\end{aligned} \right.
\end{equation*}
One also has for  $\nu > 1$ that 
\begin{equation*}
\left \{
\begin{aligned}
& B_{\nu}(z) = z + o(1) \, , & \textnormal{as $z \to 0^{+}$,} \\
& B_{\nu}(z) = \zeta(\nu) + o(1)\, , & \textnormal{as $z \to 1^{-}$,}
\end{aligned}
\right. \end{equation*}
where $\zeta(\nu) = \displaystyle \sum_{k\,=\,1}^{+\infty}\cfrac{1}{k^{\nu}}$ is Riemann's zeta function. Finally
\begin{equation*}
B_{1}(z) \to +\infty \, , \qquad \textnormal{as $z \to 1^{-}$.}
\end{equation*}
\end{remark}

If we inject \eqref{mom2_qmaxwell} in \eqref{mom0_qmaxwell} in order to get rid of $T$, we obtain the following nonlinear equation in $z$:
\begin{equation} \label{eq_fugacity}
|\alpha|\,\rho^0\,\left(\cfrac{\nu}{2\pi e^0}\right)^{\nu} = \cfrac{K_{\nu}(z)^{\nu+1}}{K_{\nu+1}(z)^{\nu}} \, , \qquad \textnormal{with $\nu = \cfrac{d_{\mathbf{v}}}{2}$.}
\end{equation}
Now, to distinguish between the fermionic and bosonic cases, let us define $\mathcal{F}_{\nu}$ and $\mathcal{B}_{\nu}$ as 
\begin{equation*}
\mathcal{F}_{\nu} = \cfrac{F_{\nu}(z)^{\nu+1}}{F_{\nu+1}(z)^{\nu}} \, , \qquad \mathcal{B}_{\nu}(z) = \cfrac{B_{\nu}(z)^{\nu+1}}{B_{\nu+1}(z)^{\nu}} \, .
\end{equation*}

\begin{figure}
\begin{center}
\begin{tabular}{cc}
\includegraphics[width=0.45\textwidth]{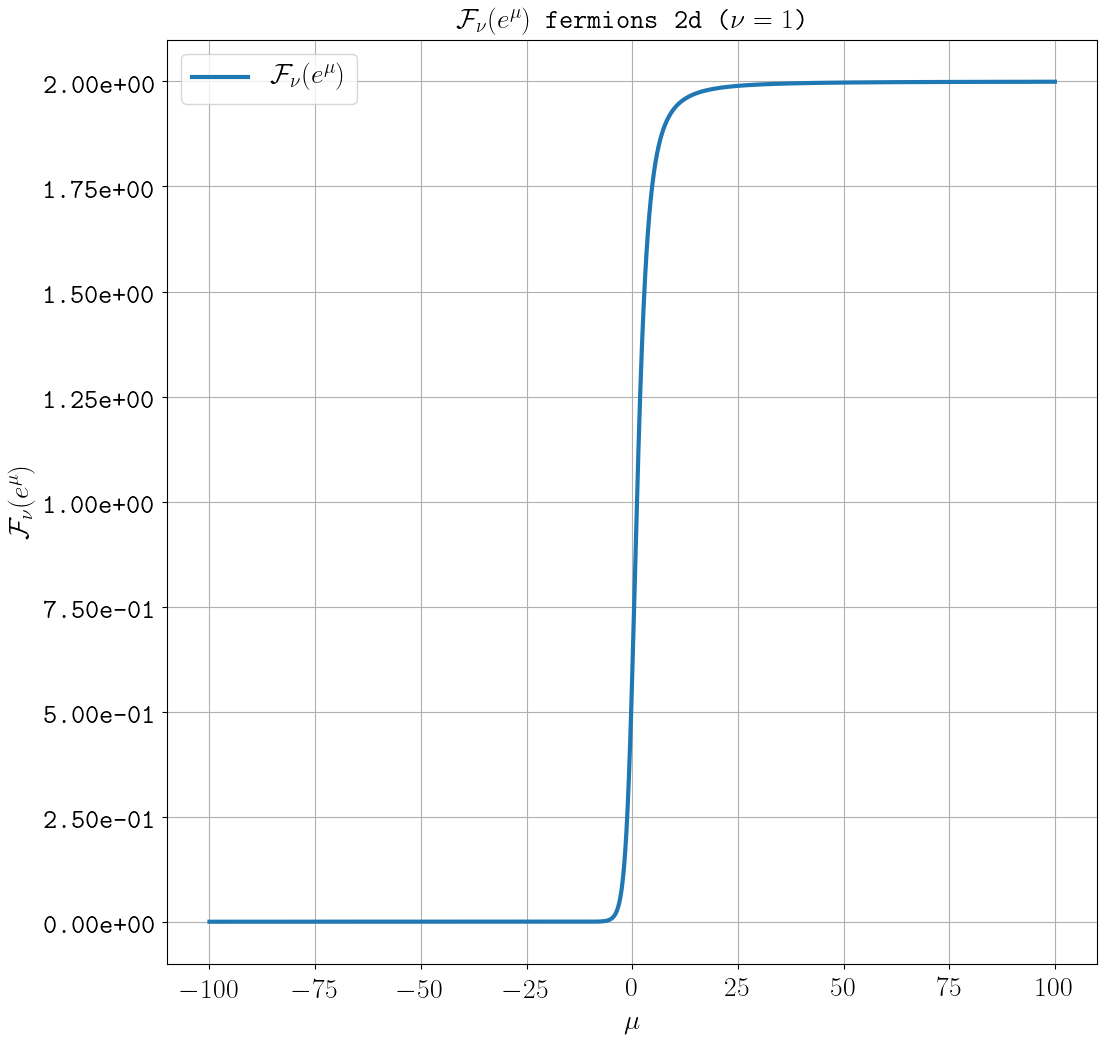} & \includegraphics[width=0.45\textwidth]{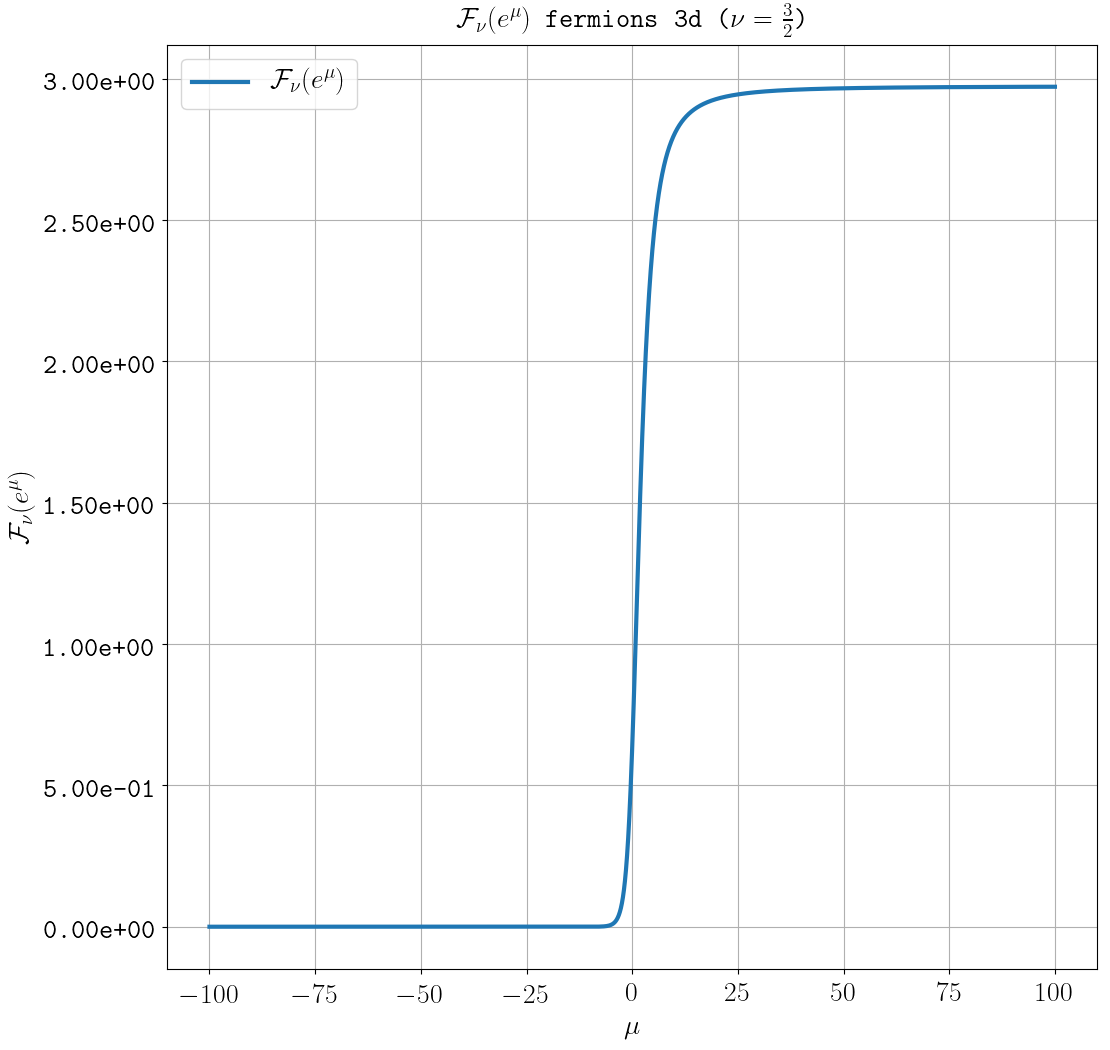} \\
(a) & (b) \\
\includegraphics[width=0.45\textwidth]{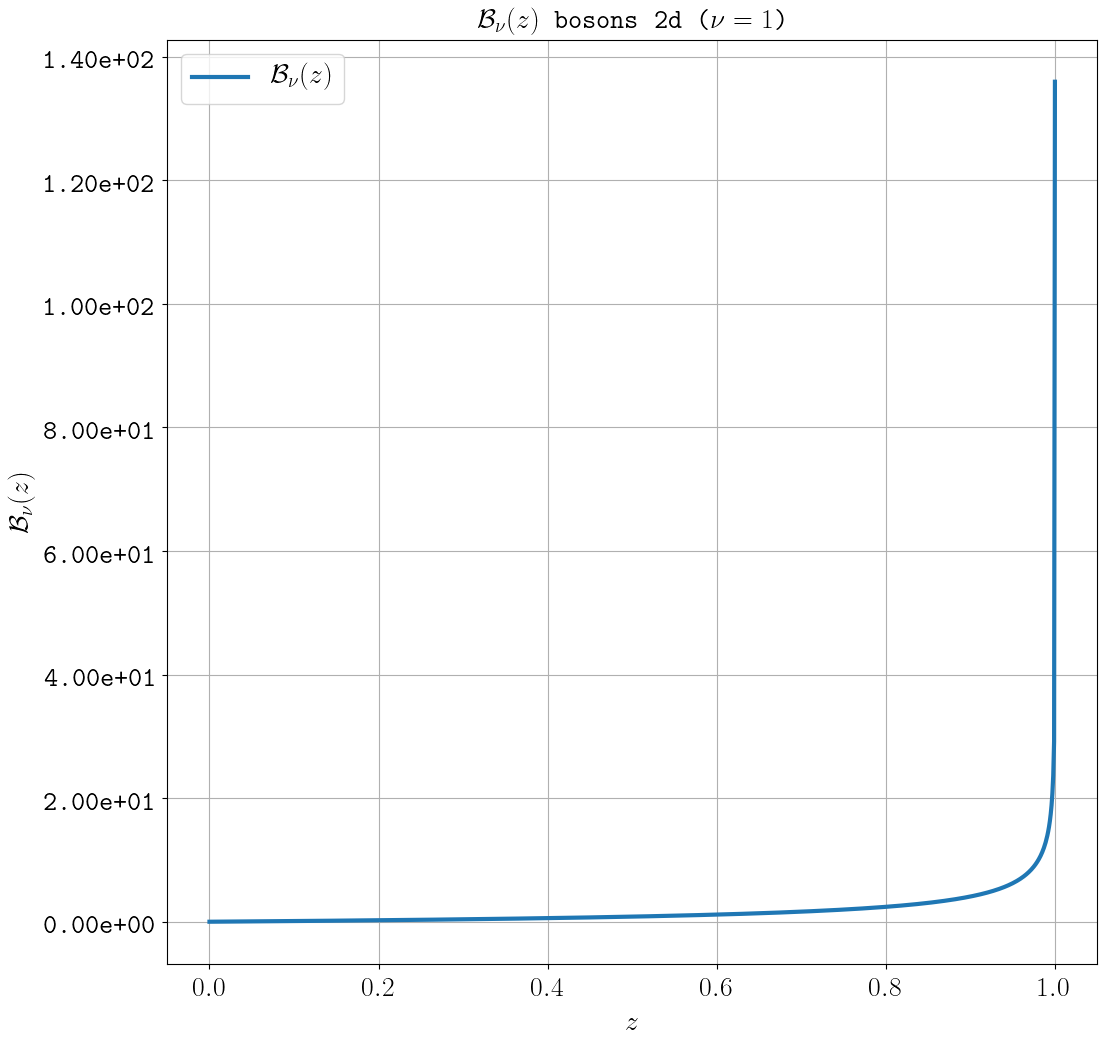} & \includegraphics[width=0.45\textwidth]{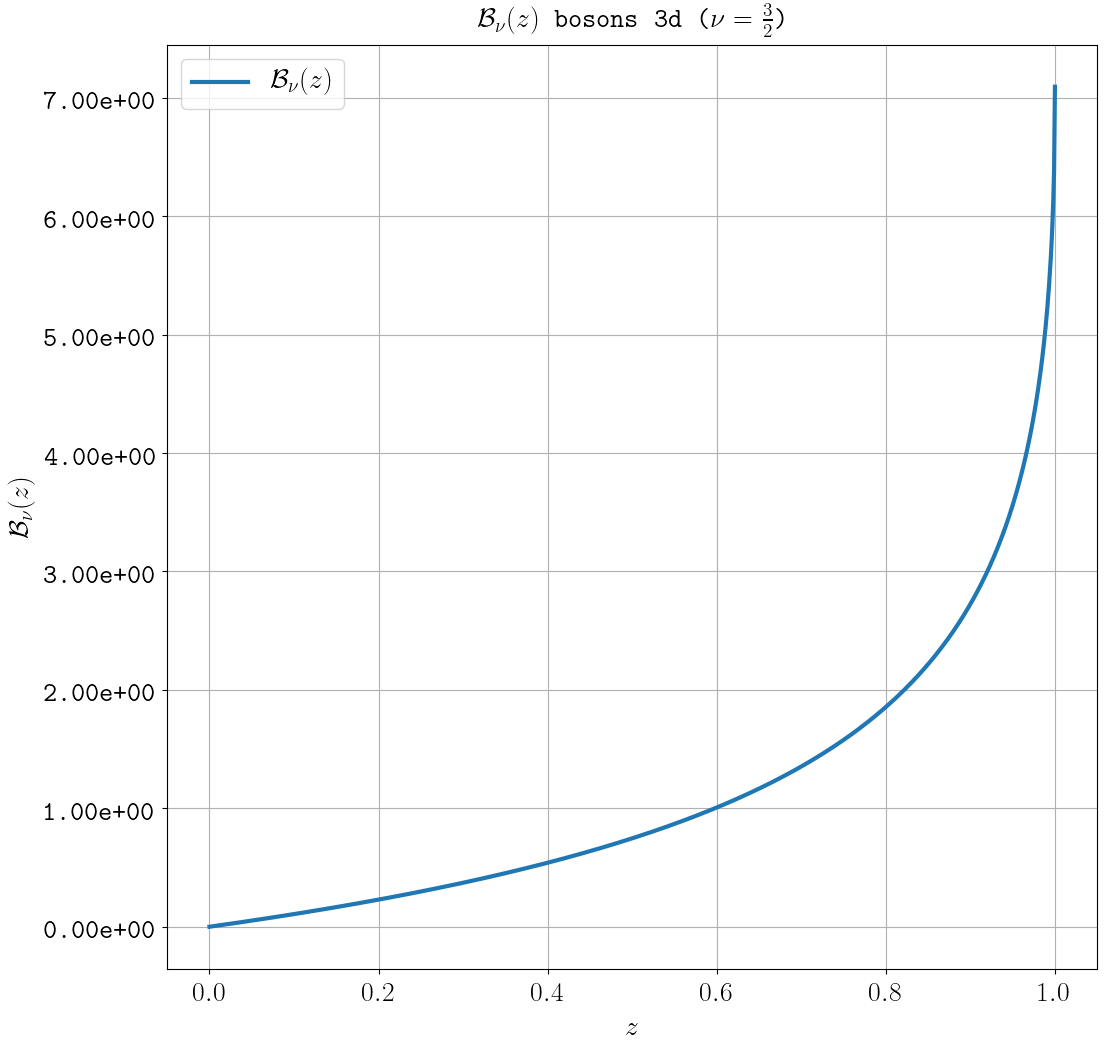} \\
(c) & (d)
\end{tabular}
\caption{Graphs of functions $\mu \longmapsto \mathcal{F}_{\nu}(e^{\mu})$ for $\nu = 1$ (a) and $\nu = \cfrac{3}{2}$ (b), and $z \mapsto \mathcal{B}_{\nu}(z)$ for $\nu = 1$ (c) and $\nu = \cfrac{3}{2}$ (d). We notice that, for each case, $\mathcal{F}_{\nu}$ and $\mathcal{B}_{\nu}$ are valued in the corresponding $I_{\textnormal{eq}}$ set defined in \eqref{def_Ieq}.} \label{FnuBnu}
\end{center}
\end{figure}

We observe numerically that $\mathcal{F}_{\nu}$ and $\mathcal{B}_{\nu}$ are nondecreasing, continuous functions \footnote{Unfortunately, we were not able to prove this assertion.} on their respective definition domain (see Figure \ref{FnuBnu}). In addition, thanks to the properties of $F_{\nu}$ and $B_{\nu}$ in Remark \ref{rem:Fnu}, one can show if $d_{\mathbf{v}} = 2, 3$ that
\begin{equation*}
\left \{\begin{aligned}
	& \mathcal{F}_{\frac{d_{\mathbf{v}}}{2}}(z) = z + o(1) \, ,  && \textnormal{as } z \to 0^{+} , \\
	& \mathcal{F}_{\frac{d_{\mathbf{v}}}{2}}(e^{\mu}) = \cfrac{(\frac{d_{\mathbf{v}}}{2}+1)^{\frac{d_{\mathbf{v}}}{2}}}{\Gamma(\frac{d_{\mathbf{v}}}{2}+1)} \, \cfrac{1+\mathcal{O}(\mu^{-2})}{1+\mathcal{O}(\mu^{-2})} \, ,  && \textnormal{as } \mu \to +\infty,
\end{aligned} \right.
\end{equation*}
and that
\begin{equation*}
\mathcal{B}_{\frac{d_{\mathbf{v}}}{2}}(z) = z + o(1) \, , \qquad \textnormal{as }z \to 0^{+}.
\end{equation*}
More precisely, the following expansions can be obtained:
\begin{equation*}
\left \{ \begin{aligned}
	&\mathcal{B}_{1}(z) = \cfrac{\log(1-z)^{2}}{\textnormal{Li}_{2}(1)} = \cfrac{6\log(1-z)^{2}}{\pi^{2}} \to +\infty \, , && \textnormal{as }z \to 1^{-}, \\
	&\mathcal{B}_{\frac{3}{2}}(z) = \cfrac{\textnormal{Li}_{\frac{3}{2}}(z)^{\frac{5}{2}}}{\textnormal{Li}_{\frac{5}{2}}(z)^{\frac{3}{2}}} \to \cfrac{\zeta(\frac{3}{2})^{\frac{5}{2}}}{\zeta(\frac{5}{2})^{\frac{3}{2}}} \, , && \textnormal{as } z \to 1^{-}.
\end{aligned} \right.
\end{equation*}

Finally, let us define the set $I_{\textnormal{eq}}$ as
\begin{equation} \label{def_Ieq}
I_{\textnormal{eq}} = \left\{
\begin{array}{ll}
]0,2[ \, , \qquad & \textnormal{for 2D Fermi-Dirac particles,} \\ \\
\left]0,\cfrac{5}{3}\sqrt{\cfrac{10}{\pi}}\right[ \, , \qquad & \textnormal{for 3D Fermi-Dirac particles,} \\ \\
]0,+\infty[ \, , \qquad & \textnormal{for 2D Bose-Einstein particles,} \\ \\
\left]0,\cfrac{\zeta(\frac{3}{2})^{\frac{5}{2}}}{\zeta(\frac{5}{2})^{\frac{3}{2}}}\right[ \, , \qquad & \textnormal{for 3D Bose-Einstein particles.}
\end{array}
\right.
\end{equation}
Then, we have shown that there exists a unique solution of \eqref{eq_fugacity} if and only if there holds
\[|\alpha|\,\rho^0\,\left(\cfrac{d_{\mathbf{v}}}{4\pi e^0}\right)^{\frac{d_{\mathbf{v}}}{2}} \in I_{\textnormal{eq}}.
\]
In such case, \eqref{eq_fugacity} can be solved numerically to get $z$, then $T$ thanks to \eqref{mom02_qmaxwell} and identify the limit state $f^{\infty}$ as $\mathcal{M}_{q}$. 
On the other hand, if
\[
	|\alpha|\,\rho^0\,\left(\cfrac{d_{\mathbf{v}}}{4\pi e^0}\right)^{\frac{d_{\mathbf{v}}}{2}} \notin I_{\textnormal{eq}},
\]
 we handle a \textit{quantum degenerative case} and the limit state $f^{\infty}$ cannot be identified as $\mathcal{M}_{q}$. Recent works  (see e.g. \cite{BriantEinav2013,Lu2014,EscobedoVelazquez2015}) indicates that, for the 3D degenerate Bose-Einstein case which describes a quantum condensation, the limit state $f^{\infty}$ is identified as
\begin{equation} \label{def_Mq_3d_bosons}
\widetilde{\mathcal{M}_{q}}(\mathbf{v}) = m_{0}\, \delta_{\mathbf{0}}(\mathbf{v}-\mathbf{u}^0) + \cfrac{1}{|\alpha|}\, \cfrac{1}{\exp\left (\frac{|\mathbf{v}-\mathbf{u}^0|^{2}}{2T}\right ) - 1} \, ,
\end{equation}
where $\delta_{\mathbf{0}}$ is the usual Dirac function, the critical mass $m_{0}$ and the temperature $T$ are linked to the density $\rho^0$ and the internal energy $e^0$ as
\begin{equation} \label{Tcritical_bosons_3d}
T =  \cfrac{2e^{0}\,\zeta(3/2)}{3\zeta(5/2)}\, ,
\end{equation}
and
\begin{equation} \label{m0critical_bosons_3d}
m_{0} = \rho^0 - \cfrac{(2\pi T)^{3/2}}{|\alpha|}\, \zeta(3/2) = \rho^0 - \cfrac{1}{|\alpha|}\,\left(\frac{4\pi e^0}{3}\right)^{3/2} \, \cfrac{\zeta(3/2)^{5/2}}{\zeta(5/2)^{3/2}} \, .
\end{equation}
We can remark here that
\[
|\alpha|\,\rho^0\,\left(\cfrac{d_{\mathbf{v}}}{4\pi e^0}\right)^{\frac{d_{\mathbf{v}}}{2}} \notin I_{\textnormal{eq}} \qquad \Longleftrightarrow \qquad m_{0} > 0 \, .
\]

Concerning Fermi-Dirac degenerative cases, we have to deal with the so-called ``maximum assumption'' 
\[\max f \leq |\alpha|^{-1}\] in order to define the entropy.  It was shown in \cite{escobedo2005entropy,allemand2010incompressible} that, in the critical case where $\rho^0$ and $e^0$ satisfy
\begin{equation} \label{critical_FD}
\cfrac{\rho^0}{|D(\mathbf{0},A)|} = \cfrac{1}{|\alpha|} \, , \qquad A = \sqrt{\cfrac{2\, e^0 \, (d_{\mathbf{v}}+2)}{d_{\mathbf{v}}}} \, ,
\end{equation}
the distribution function $f$ converges to the limit state $\chi$ defined  in \eqref{def_chi}. If the maximum assumption is broken, the entropy cannot be defined anymore as $\mathcal{H}[f]$ so nothing clear can be said about the limit state in this context.

\paragraph{Practical computations} Let us classify the limit states according to the initial parameters:

	 \noindent \textit{Mass $\rho^0$ and internal energy $e^0$ given.} We want to compute the fugacity $z$ and the temperature $T$. For this purpose, we test the assertion \[|\alpha|\,\rho^0\,\left(\cfrac{d_{\mathbf{v}}}{4\pi e^0}\right)^{\frac{d_{\mathbf{v}}}{2}} \in I_{\textnormal{eq}}.\] 
	\begin{itemize}
		\item If it holds, we solve numerically \eqref{eq_fugacity} with Brent's method \cite{brent1973algorithms} to get an approximation of $z$. When Fermi-Dirac particles are considered, we use the scale $z = e^{\mu}$ and run the iterative method to find $\mu$. This  allows to get large values of $z$ when $|\alpha|\,\rho^0\,\left(\frac{d_{\mathbf{v}}}{4\pi e^0}\right)^{\frac{d_{\mathbf{v}}}{2}}$ is close to the upper bound of $I_{\textnormal{eq}}$. The limit distribution is then set\footnote{Note that, the relaxation of $f$ is always entropic for Bose-Einstein particles but can be non-entropic for Fermi-Dirac particles} as $f^{\infty} = \mathcal{M}_{q}$.
		\item If it does not hold,  
		\begin{itemize}
			\item when 3D Bose-Einstein particles are considered, we use \eqref{Tcritical_bosons_3d}-\eqref{m0critical_bosons_3d} to set $T,m_{0}$ and we set $z = 1$ and the limit state is defined again as $f^{\infty} = \widetilde{\mathcal{M}_{q}}$.
			\item For 2D and 3D Fermi-Dirac cases, we test the equality \eqref{critical_FD}.
			\begin{itemize}
				\item If it holds, the limit state is identified as $f^{\infty} = \chi$ and the relaxation is entropic.
				\item If it does not hold, we cannot say anything about the behaviour of $f(t,\cdot)$ as $t \to +\infty$ since it is non-entropic.
			\end{itemize}
		\end{itemize}
	\end{itemize}

	\noindent \textit{Mass $\rho$ and temperature $T$ given.} We want to compute the fugacity $z$ and the internal energy $e^0$.
	\begin{itemize}
		\item If the 3D Bose-Einstein case is considered, we test the inequality \[|\alpha|\, \rho^0 \leq (2\pi T)^{\frac{3}{2}}\, \zeta(\frac{3}{2}).\] 
		\begin{itemize}
			\item If it is true, then we can solve numerically \eqref{eq_fugacity} with Brent's method to get an approximation of $z$. We use it to get $e^0$ from \eqref{mom2_qmaxwell} and set the limit distribution as $f^{\infty} = \mathcal{M}_{q}$.
			\item If it is false, then we use \eqref{Tcritical_bosons_3d}-\eqref{m0critical_bosons_3d} to set $T,m_{0}$ and we set $z = 1$ when Bose-Einstein particles are considered and the limit state is defined as $f^{\infty} = \widetilde{\mathcal{M}_{q}}$.
		\end{itemize}
		\item If 2D Bose-Einstein or Fermi-Dirac cases are considered, we can solve numerically \eqref{eq_fugacity} with Brent's method to get an approximation of $z$, then set the limit distribution as $f^{\infty} = \mathcal{M}_{q}$ and compute $e^{0}$ thanks to \eqref{mom2_qmaxwell}.
	\end{itemize}

\section{Fast deterministic approximations of the Boltzmann-Nordheim collision operator}
\label{sec:FBNEA}

Let us now introduce the numerical method that we shall study and use in this paper. This section is strongly inspired from \cite{FilbetHuJin:2012}. For the sake of simplicity, all the functions described in this section will only depend on the velocity variable, denoted by $\bm{\xi}$.

\subsection{The Carleman representation}

Let us decompose the BNE operator \eqref{eq:BNE_OP}  using the following bilinear and trilinear operators (see also Remark \ref{rem:Quadri2Tri}):
\begin{equation}
\label{eq:subBNE_OP}
\begin{split}
\mathcal{Q}^+_{c}(F,G)(\bm{\xi}) &= \int_{\R^{d_{\mathbf{v}}}}\int_{\Sp^{d_{\mathbf{v}}-1}} B\left(|\bm{\xi}-\bm{\xi}_{*}|,\theta\right)\, F'\,G_{*}' \, d\bm{\sigma}\,d\bm{\xi}_{*} \, , \\
\mathcal{Q}^-_{c}(F,G)(\bm{\xi}) &= F(\bm{\xi}) \int_{\R^{d_{\mathbf{v}}}}\int_{\Sp^{d_{\mathbf{v}}-1}} B\left(|\bm{\xi}-\bm{\xi}_{*}|,\theta\right) G_{*} \, d\bm{\sigma}\,d\bm{\xi}_{*} \, , \\
\mathcal{Q}_{1,q}(F,G,H)(\bm{\xi}) &= \int_{\R^{d_{\mathbf{v}}}}\int_{\Sp^{d_{\mathbf{v}}-1}} B\left(|\bm{\xi}-\bm{\xi}_{*}|,\theta\right)\, F' \,G_{*}' \, H_{*} \, d\bm{\sigma}\,d\bm{\xi}_{*} \, , \\
\mathcal{Q}_{2,q}(F,G,H)(\bm{\xi}) &= F(\bm{\xi}) \int_{\R^{d_{\mathbf{v}}}}\int_{\Sp^{d_{\mathbf{v}}-1}} B\left(|\bm{\xi}-\bm{\xi}_{*}|,\theta\right)\, G'\,H_{*}'  \, d\bm{\sigma}\,d\bm{\xi}_{*} \, , \\
\mathcal{Q}_{3,q}(F,G,H)(\bm{\xi}) &= F(\bm{\xi}) \int_{\R^{d_{\mathbf{v}}}}\int_{\Sp^{d_{\mathbf{v}}-1}}B\left(|\bm{\xi}-\bm{\xi}_{*}|,\theta\right)\, G_{*} \, H' \, d\bm{\sigma}\,d\bm{\xi}_{*} \, , \\
\mathcal{Q}_{4,q}(F,G,H)(\bm{\xi}) &= F(\bm{\xi})  \int_{\R^{d_{\mathbf{v}}}}\int_{\Sp^{d_{\mathbf{v}}-1}} B\left(|\bm{\xi}-\bm{\xi}_{*}|,\theta\right)\, G_{*} \, H_{*}' \, d\bm{\sigma}\,d\bm{\xi}_{*} \, .
\end{split}
\end{equation}
Hence the complete collision operator $\mathcal{Q}(F,F)$ writes as
\begin{equation*}
\mathcal{Q}_{\alpha}(F) = \mathcal{Q}^+_{c}(F,F)-\mathcal{Q}^-_{c}(F,F) - \alpha\,\left[\mathcal{Q}_{1,q}(F,F,F)+\mathcal{Q}_{2,q}(F,F,F) - \mathcal{Q}_{3,q}(F,F,F)-\mathcal{Q}_{4,q}(F,F,F)\right] \ .
\end{equation*}
\begin{remark}
	\label{rem:computComplexity}
	One can notice in \eqref{eq:subBNE_OP} that the only term that consists on a full integral on the three distributions $F$, $G$ and $H$ is $\mathcal{Q}_{1,q}$. This term will then logically be the most computationally expensive.
\end{remark}

By completing the square, one has the following useful lemma from \cite{MoPa:2006}:

\begin{lemma}\label{inject_Dirac_lemma}
For any function $F:\R^{d_{\mathbf{v}}} \to \R^{n}$, we have
\begin{equation*}
\cfrac{1}{2} \, \int_{\Sp^{d_{\mathbf{v}}-1}} F\left(|\mathbf{w}|\,\bm{\sigma} - \mathbf{w}\right) \, d\bm{\sigma} = \cfrac{1}{|\mathbf{w}|^{d_{\mathbf{v}}-2}} \, \int_{\R^{d_{\mathbf{v}}}} \delta\left(2\mathbf{x}\cdot\mathbf{w}+|\mathbf{x}|^{2}\right) \, F(\mathbf{x})\, d\mathbf{x} \, .
\end{equation*}
\end{lemma}

Thanks to this lemma, one has the following Carleman-like representation for the classical collision operator:
\begin{multline*}
\mathcal{Q}^+_{c}(F,G)(\bm{\xi}) - \mathcal{Q}^-_{c}(F,G)(\bm{\xi}) 
= \\ \int_{\R^{2d_{\mathbf{v}}}} \tilde{B}(\mathbf{z},\mathbf{y}) \, \delta(\mathbf{z}\cdot\mathbf{y}) \, \left[ F(\bm{\xi} + \mathbf{z})\, G(\bm{\xi}+\mathbf{y}) - F(\bm{\xi})\,G(\bm{\xi}+\mathbf{z}+\mathbf{y})\right] \, d\mathbf{z}\, d\mathbf{y} \, ,
\end{multline*}
with $\tilde{B}$ defined as
\begin{equation*}
\tilde{B}(\mathbf{z},\mathbf{y}) = 2^{d_{\mathbf{v}}-1}\, C_{\Phi}\,(|\mathbf{z}|^{2}+|\mathbf{y}|^{2})^{\frac{\gamma-d_{\mathbf{v}}+2}{2}} \, b\left(\cfrac{\left|\mathbf{y}\right|}{\sqrt{\left|\mathbf{y}\right|^{2}+\left|\mathbf{z}\right|^{2}}} \, , \, \cfrac{\left|\mathbf{z}\right|}{\sqrt{\left|\mathbf{y}\right|^{2}+\left|\mathbf{z}\right|^{2}}} \right) \, .
\end{equation*}
Similar computations lead to the following Carleman-like representations for the quantum, trilinear collision operators (with the same effect on computational complexity as in Remark \ref{rem:computComplexity}, namely that $\mathcal{Q}_{1,q}$ will be the more computationally intensive operator)
\begin{align*}
\mathcal{Q}_{1,q}(F,G,H)(\bm{\xi}) &= \int_{\R^{2d_{\mathbf{v}}}} \tilde{B}(\mathbf{z},\mathbf{y}) \, \delta(\mathbf{z}\cdot\mathbf{y}) \, F(\bm{\xi} + \mathbf{z})\, G(\bm{\xi}+\mathbf{y})\, H(\bm{\xi}+\mathbf{z}+\mathbf{y})\, d\mathbf{z}\,d\mathbf{y} \, ,
\\
\mathcal{Q}_{2,q}(F,G,H)(\bm{\xi}) &= F(\bm{\xi})\int_{\R^{2d_{\mathbf{v}}}} \tilde{B}(\mathbf{z},\mathbf{y}) \, \delta(\mathbf{z}\cdot\mathbf{y}) \, G(\bm{\xi} + \mathbf{z})\, H(\bm{\xi}+\mathbf{y})\, d\mathbf{z}\,d\mathbf{y} \, ,
\\
\mathcal{Q}_{3,q}(F,G,H)(\bm{\xi}) &= F(\bm{\xi})\int_{\R^{2d_{\mathbf{v}}}} \tilde{B}(\mathbf{z},\mathbf{y}) \, \delta(\mathbf{z}\cdot\mathbf{y}) \, G(\bm{\xi}+\mathbf{z}+\mathbf{y}) \, H(\bm{\xi} + \mathbf{z})\, d\mathbf{z}\,d\mathbf{y} \, ,
\\
\mathcal{Q}_{4,q}(F,G,G)(\bm{\xi}) &= F(\bm{\xi}) \int_{\R^{2d_{\mathbf{v}}}} \tilde{B}(\mathbf{z},\mathbf{y}) \, \delta(\mathbf{z}\cdot\mathbf{y}) \, G(\mathbf{y}) \, H(\bm{\xi} + \mathbf{y})\, d\mathbf{z}\,d\mathbf{y} \, .
\end{align*}

	\subsection{Truncation and periodization of the BNE operator}

In order to develop a fast spectral approximation for the Boltzmann-Nordheim operator, one has first to truncate it, in order to make it deal with function with finite support. We shall follow along the lines of \cite{FilbetHuJin:2012}.

\begin{theorem} \label{truncation}
	Let the function $G$ be compactly supported in the $d_{\mathbf{v}}$-dimensional ball $B(0,S)$. Then, the support of $\bm{\xi} \mapsto \mathcal{Q}_{\alpha}(G)(\bm{\xi})$ is included in $B(0,\sqrt{2}S)$.
\end{theorem}

\begin{proof}
Considering $\mathbf{z},\mathbf{y} \in \R^{d_{\mathbf{v}}}$ such that $\mathbf{z} \cdot \mathbf{y} = 0$, we have
\begin{displaymath}
|\bm{\xi}|^{2} + |\bm{\xi}+\mathbf{z}+\mathbf{y}|^{2} = |\bm{\xi}+\mathbf{z}|^{2} + |\bm{\xi}+\mathbf{y}|^{2} \, .
\end{displaymath}
Hence, if $\bm{\xi},\bm{\xi}+\mathbf{z},\bm{\xi}+\mathbf{y},\bm{\xi}+\mathbf{z}+\mathbf{y} \in B(0,S)$, it is possible to have
\begin{equation*}
G(\bm{\xi}+\mathbf{z})\, G(\bm{\xi}+\mathbf{y}) \, \left[ 1 - \alpha \left( G(\bm{\xi})+G(\bm{\xi}+\mathbf{z}+\mathbf{y})\right) \right] 
- G(\bm{\xi}) \, G(\bm{\xi}+\mathbf{z}+\mathbf{y}) \left[ 1-\alpha \left( G(\bm{\xi}+\mathbf{z}) + G(\bm{\xi}+\mathbf{y}) \right) \right] \neq 0 \, .
\end{equation*}
In this situation, we have
\begin{displaymath}
|\bm{\xi}|^{2} \leq |\bm{\xi}|^{2} + |\bm{\xi}+\mathbf{z}+\mathbf{y}|^{2} \leq 2S^{2} \, ,
\end{displaymath}
so we can conclude that the support of $\mathcal{Q}_{\alpha}(G)$ is included in the ball $B\left(0,\sqrt{2}S\right)$ if the support of $G$ is in $B(0,S)$ in $\bm{\xi}$.
\end{proof}

We now define a \textit{truncated collision operator} $\mathcal{Q}_{\alpha}^{R}$ with $R > 0$ as follows: given a function $G:\R^{d_{\mathbf{v}}}\to \R$ with compact support in $B(0,S)$, we define $\tilde{G}:[-\pi,\pi]^{d_{\mathbf{v}}} \to \R$ such that
\begin{equation*}
\tilde{G}(\bm{\xi}) = \left\{
\begin{array}{ll}
G(\bm{\xi}) \, , & \textnormal{if $\bm{\xi} \in B(0,S)$,} \\
0 \, , & \textnormal{else.}
\end{array}
\right.
\end{equation*}
At this point, we have several candidates for $\mathcal{Q}_{\alpha}^{R}$, each one of them being written under the following generic way
\begin{multline*}
\mathcal{Q}_{\alpha}^{R}(G)(\bm{\xi}) = \int_{\R^{2d_{\mathbf{v}}}} \tilde{B}(\mathbf{x},\mathbf{y}) \, \delta(\mathbf{x}\cdot\mathbf{y}) \, \Big[ \tilde{G}(\bm{\xi} + \mathbf{x})\, \tilde{G}(\bm{\xi}+\mathbf{y})\,\left(1- \alpha \, \left(\tilde{G}(\bm{\xi}) + \tilde{G}(\bm{\xi}+\mathbf{x}+\mathbf{y})\right)\right) \\
 - \tilde{G}(\bm{\xi})\,\tilde{G}(\bm{\xi}+\mathbf{x}+\mathbf{y})\,\left(1-\alpha\, \left(\tilde{G}(\bm{\xi} + \mathbf{x}) + \tilde{G}(\bm{\xi} + \mathbf{y})\right) \right) \Big]  \mathcal{I}(\mathbf{x},\mathbf{y},\bm{\xi})\, d\mathbf{x}\,d\mathbf{y} \, ,
\end{multline*}
where $\mathcal{I}$ is defined as follows:
\begin{itemize}
\item the classical spectral method from \cite{PP96}
\begin{equation*}
\mathcal{I}(\mathbf{x},\mathbf{y},\bm{\xi}) = \mathbf{1}_{B(0,S)}(\bm{\xi}+\mathbf{x}+\mathbf{y})\, \mathbf{1}_{B(0,2S)}(\mathbf{x}-\mathbf{y})\, \mathbf{1}_{B(0,\sqrt{2}S)}(\bm{\xi}) \, ,
\end{equation*}
\item the fast spectral method from \cite{MoPa:2006}
\begin{equation*}
\mathcal{I}(\mathbf{x},\mathbf{y},\bm{\xi}) = \mathbf{1}_{B(0,2R)}(\mathbf{x}-\mathbf{y})\, \mathbf{1}_{B(0,\sqrt{2}S)}(\bm{\xi}) \, , \qquad R \geq S \, , 
\end{equation*}
\item the fast quantum Boltzmann spectral method from \cite{FilbetHuJin:2012}:
\begin{equation}
\label{eq:IndicatorFilbetHuJin}
\mathcal{I}(\mathbf{x},\mathbf{y},\bm{\xi}) = \mathbf{1}_{B(0,R)}(\mathbf{x})\, \mathbf{1}_{B(0,R)}(\mathbf{y})\, \mathbf{1}_{B(0,\sqrt{2}S)}(\bm{\xi}) \, , \qquad R \geq S \, ,
\end{equation}
\end{itemize}
In order to simplify the computations of the following section, we consider the definition \eqref{eq:IndicatorFilbetHuJin} for $\mathcal{I}$. It gives the following expression of $\mathcal{Q}_{\alpha}^{R}(G)$:
\begin{multline*}
\mathcal{Q}_{\alpha}^{R}(G)(\bm{\xi})
= \int_{B(0,R)^{2}} \tilde{B}(\mathbf{x},\mathbf{y}) \, \delta(\mathbf{x}\cdot\mathbf{y}) \, \Big[ \tilde{G}(\bm{\xi} + \mathbf{x})\, \tilde{G}(\bm{\xi}+\mathbf{y})\,\left(1- \alpha \, \left(\tilde{G}(\bm{\xi}) + \tilde{G}(\bm{\xi}+\mathbf{x}+\mathbf{y})\right)\right) \\
 - \tilde{G}(\bm{\xi})\,\tilde{G}(\bm{\xi}+\mathbf{x}+\mathbf{y})\,\left(1-\alpha\, \left(\tilde{G}(\bm{\xi} + \mathbf{x}) + \tilde{G}(\bm{\xi} + \mathbf{y})\right) \right) \Big] \, d\mathbf{x}\,d\mathbf{y} \, .
\end{multline*}

We now define the periodized function $\tilde{\tilde{G}}$ as 
\begin{equation*}
\tilde{\tilde{G}}(\bm{\xi}+2\mathbf{k}\pi) = \tilde{G}(\bm{\xi}) \, , \qquad \forall\,\bm{\xi} \in [-\pi,\pi]^{d_{\mathbf{v}}} \, , \quad \forall\,\mathbf{k} \in \Z^{d_{\mathbf{v}}}\, .
\end{equation*}
Choosing $R > 0$ and $S > 0$ such that $R \geq S$ and $S$ is small enough allows to write
\begin{equation*}
\mathcal{Q}_{\alpha}^{R}(\tilde{\tilde{G}})(\bm{\xi}) = \mathcal{Q}_{\alpha}(\tilde{\tilde{G}})(\bm{\xi}) \, \quad \forall \bm{\xi} \in [-\pi,\pi]^{d_{\mathbf{v}}}.
\end{equation*}

More precisely, since $G$ is supported in $B(0,S)$, $\tilde{\tilde{G}}$ is supported on the balls $B(2\mathbf{k}\pi,S)$ for any $\mathbf{k} \in \Z^{d_{\mathbf{v}}}$. Consequently, if a particle is equipped with a velocity $\bm{\xi} \in B(2\mathbf{k}\pi,S)$, it may interact with any particle equipped with a velocity $\bm{\xi}_{*} \in B\left(2\mathbf{k}\pi,2R+\sqrt{2}S\right)$. Such candidate particle should not be in any $B(2\mathbf{l}\pi,S)$ for all $\mathbf{l} \in \Z^{d_{\mathbf{v}}}$ for avoiding artificial interactions. In order to guarantee this constraint, we must impose $2\pi-S \geq 2R+\sqrt{2}S$, namely
\begin{equation}\label{anti-aliasing}
 S \leq \frac{2 \pi}{2R+(1+\sqrt{2})} \, .
\end{equation}
This condition is generally referred as the \textit{anti-aliasing} condition on $S$ and $R$.

\subsection{Fast spectral approximation of the BNE operator}
 For now on, us assume that $\textnormal{Supp}_{\bm{\xi}}G \subset B(0,S) \subset [-\pi,\pi]^{d_{\mathbf{v}}}$ and consider a truncation radius $R > 0$ satisfying $R \geq 2S$ and \eqref{anti-aliasing}. Hence
\begin{equation*}
\mathcal{Q}_{\alpha}(G)(\mathbf{v}) = \mathcal{Q}_{\alpha}^{R}(G)(\mathbf{v}) \, .
\end{equation*}
Given a function $f\in L^2_{per}[-\pi,\pi]^{d_\mathbf{v}})$,  we define its Fourier series representation as
\begin{equation}
f(\mathbf v) = \sum_{\mathbf k \in\ZZ^{d_\mathbf{v}}}\widehat f_\mathbf{k} e^{i \mathbf{k} \cdot \mathbf{v}}, \qquad \widehat f_\mathbf{k} = \frac{1}{(2\pi)^{d_\mathbf{v}}}\int_{[-\pi,\pi]^{d_\mathbf{v}}} f(\mathbf v)
e^{-i \mathbf k \cdot \mathbf v }\,d\mathbf v,
\label{eq:FS}
\end{equation}
where we used the multi-index notation $\mathbf k=(k_1,\ldots,k_{d_\mathbf{v}})$ to denote the ${d_\mathbf{v}}$-dimensional sums over the indexes $k_j$, $j=1,\ldots,{d_\mathbf{v}}$.

\indent We introduce the truncated terms $\mathcal{Q}_{1,c}^{R}, \mathcal{Q}_{2,c}^{R}, \mathcal{Q}_{1,q}^{R}, \mathcal{Q}_{2,q}^{R}, \mathcal{Q}_{3,q}^{R}, \mathcal{Q}_{4,q}^{R}$ as

\begin{align*}
\mathcal{Q}_{1,c}^{R}&(F,G)(\bm{\xi}) = \int_{B(0,R)^{2}} \tilde{B}(\mathbf{z},\mathbf{y}) \, \delta(\mathbf{z}\cdot\mathbf{y}) \, F(\bm{\xi} + \mathbf{z})\, G(\bm{\xi}+\mathbf{y})\, d\mathbf{z}\,d\mathbf{y} \, , \\
\mathcal{Q}_{2,c}^{R}&(F,G)(\bm{\xi}) = F(\bm{\xi})\int_{B(0,R)^{2}} \tilde{B}(\mathbf{z},\mathbf{y}) \, \delta(\mathbf{z}\cdot\mathbf{y}) \,  G(\bm{\xi}+\mathbf{z}+\mathbf{y})\, d\mathbf{z}\,d\mathbf{y} \, , \\
\mathcal{Q}_{1,q}^{R}&(F,G,H)(\bm{\xi}) = \int_{B(0,R)^{2}} \tilde{B}(\mathbf{z},\mathbf{y}) \, \delta(\mathbf{z}\cdot\mathbf{y}) \, F(\bm{\xi} + \mathbf{z})\, G(\bm{\xi}+\mathbf{y})\, H(\bm{\xi}+\mathbf{z}+\mathbf{y})\, d\mathbf{z}\,d\mathbf{y} \, , \\
\mathcal{Q}_{2,q}^{R}&(F,G,H)(\bm{\xi}) = F(\bm{\xi}) \int_{B(0,R)^{2}} \tilde{B}(\mathbf{z},\mathbf{y}) \, \delta(\mathbf{z}\cdot\mathbf{y}) \, G(\bm{\xi} + \mathbf{z})\, H(\bm{\xi}+\mathbf{y})\,d\mathbf{z}\,d\mathbf{y} \, , \\
\mathcal{Q}_{3,q}^{R}&(F,G,H)(\bm{\xi}) = F(\bm{\xi}) \int_{B(0,R)^{2}} \tilde{B}(\mathbf{z},\mathbf{y}) \, \delta(\mathbf{z}\cdot\mathbf{y}) \,  G(\bm{\xi}+\mathbf{z}+\mathbf{y}) \, H(\bm{\xi} + \mathbf{z})\, d\mathbf{z}\,d\mathbf{y} \, , \\
\mathcal{Q}_{4,q}^{R}&(F,G,H)(\bm{\xi}) = F(\bm{\xi}) \int_{B(0,R)^{2}} \tilde{B}(\mathbf{z},\mathbf{y}) \, \delta(\mathbf{z}\cdot\mathbf{y}) \, G(\mathbf{y}) \, H(\bm{\xi} + \mathbf{y})\, d\mathbf{z}\,d\mathbf{y} \, .
\end{align*}

Using orthogonality properties of orthogonal polynomials, one has the following expansion of the quantum kernel modes.
\begin{lemma}
For any $N \in \NN^{*}$ and any multi-index $\mathbf{n}$, one has the following representation of the Fourier modes of the terms of the BNE collision operator:
\begin{align*}
\widehat{\mathcal{Q}_{1,c}^{R}(F,G)}_{\mathbf{n}} &= \sum_{\substack{\mathbf{k},\mathbf{l} \, \in \, I_{N} \\ \mathbf{k}+\mathbf{l} \,=\, \mathbf{n}}} \beta^{R}(\mathbf{k},\mathbf{l}) \, \widehat{F}_{\mathbf{k}} \, \widehat{G}_{\mathbf{l}} \, , \\
\widehat{\mathcal{Q}_{2,c}^{R}(F, G)}_{\mathbf{n}} &= \sum_{\substack{\mathbf{k},\mathbf{l} \, \in \, I_{N} \\ \mathbf{k}+\mathbf{l} \,=\, \mathbf{n}}}  \beta^{R}(\mathbf{l},\mathbf{l}) \, \widehat{F}_{\mathbf{k}} \, \widehat{G}_{\mathbf{l}} \, , \\
\widehat{\mathcal{Q}_{1,q}^{R}(F,G,H)}_{\mathbf{n}} &= \sum_{\substack{\mathbf{k},\mathbf{l},\mathbf{m} \, \in \, I_{N} \\ \mathbf{k}+\mathbf{l}+\mathbf{m} \,=\, \mathbf{n}}} \beta^{R}(\mathbf{k}+\mathbf{m},\mathbf{l}+\mathbf{m}) \, \widehat{F}_{\mathbf{k}} \, \widehat{G}_{\mathbf{l}}\,\widehat{H}_{\mathbf{m}} \, , \\
\widehat{\mathcal{Q}_{2,q}^{R}(F,G,H)}_{\mathbf{n}} &= \sum_{\substack{\mathbf{k},\mathbf{l},\mathbf{m} \, \in \, I_{N} \\ \mathbf{k}+\mathbf{l}+\mathbf{m} \,=\, \mathbf{n}}} \beta^{R}(\mathbf{k},\mathbf{l}) \, \widehat{F}_{\mathbf{k}} \, \widehat{G}_{\mathbf{l}}\,\widehat{H}_{\mathbf{m}} \, , \\
\widehat{\mathcal{Q}_{3,q}^{R}(F,G,H)}_{\mathbf{n}} &= \sum_{\substack{\mathbf{k},\mathbf{l},\mathbf{m} \, \in \, I_{N} \\ \mathbf{k}+\mathbf{l}+\mathbf{m} \,=\, \mathbf{n}}} \beta^{R}(\mathbf{k}+\mathbf{l},\mathbf{l}) \, \widehat{F}_{\mathbf{k}} \, \widehat{G}_{\mathbf{l}}\,\widehat{H}_{\mathbf{m}} \, , \\
\widehat{\mathcal{Q}_{4,q}^{R}(F,G,H)}_{\mathbf{n}} &= \sum_{\substack{\mathbf{k},\mathbf{l},\mathbf{m} \, \in \, I_{N} \\ \mathbf{k}+\mathbf{l}+\mathbf{m} \,=\, \mathbf{n}}} \beta^{R}(\mathbf{l},\mathbf{k}+\mathbf{l}) \, \widehat{F}_{\mathbf{k}} \, \widehat{G}_{\mathbf{l}}\,\widehat{H}_{\mathbf{m}} \, ,
\end{align*}
with the so-called \emph{kernel modes} $\beta^{R}(\mathbf{k},\mathbf{l})$ being defined as
\begin{equation*}
\beta^{R}(\mathbf{k},\mathbf{l}) = \int_{B(0,R)^{2}} \tilde{B}(\mathbf{z},\mathbf{y}) \, \delta(\mathbf{z}\cdot\mathbf{y})\,e^{i\,(\mathbf{z}\cdot\mathbf{k}+\mathbf{y}\cdot\mathbf{l})} \, d\mathbf{z}\,d\mathbf{y} \, ,
\end{equation*}
and the set $I_N$ given by
\begin{equation*}
I_{{N}} = \prod_{d\,=\,1}^{d_{\mathbf{v}}} \left\{ -\cfrac{N}{2},\ldots,\cfrac{N}{2}-1 \right\} \subset \Z^{d_{\mathbf{v}}} \, .
\end{equation*}
\end{lemma}

Such an expansion of the kernel mode yields an asymptotic complexity for computing the BNE operator of $\mathcal{O}(N^{2d_\mathbf{v}})$ for the classical terms and $\mathcal{O}(N^{3d_\mathbf{v}})$ for the term $\mathcal Q_{1,q}^R$. We shall now expand as in \cite{MoPa:2006} the collision kernel in order to introduce a convolution sum, allowing to improve the overall computational complexity by almost $d_\mathbf{v}$ orders of magnitude. 

\subsection{Expansion of the kernel modes}

Let us assume from now on that $\tilde{B}$ can be expressed as follows:
\begin{equation*} 
\tilde{B}(\mathbf{z},\mathbf{y}) = a\left(|\mathbf{z}|\right) \, b\left(|\mathbf{y}|\right) \, ,
\end{equation*}
with some given functions $a, b:\R_{+} \to \R$. Hence, decomposing polar coordinate $\mathbf{z} = \rho\,\mathbf{e}$ and $\mathbf{y} = \rho'\,\mathbf{e}'$ with $\rho,\rho' \in [0,R]$, $\mathbf{e},\mathbf{e}' \in \Sp^{d_{\mathbf{v}}-1}$, one has
\begin{equation*}
\beta^{R}(\mathbf{k},\mathbf{l}) = \int_{\Sp^{d_{\mathbf{v}}-1}} \int_{\Sp^{d_{\mathbf{v}}-1}} \delta(\mathbf{e}\cdot\mathbf{e}')\, \left[\int_{0}^{R} a\left(|\rho|\right)\,\rho^{d_{\mathbf{v}}-2}\,e^{i\,\rho\,\mathbf{k}\cdot\mathbf{e}}\,d\rho\right]
 \left[ \int_{0}^{R} b\left(|\rho|\right)\,(\rho')^{d_{\mathbf{v}}-2}\,e^{i\,\rho'\,\mathbf{l}\cdot\mathbf{e}'}\,d\rho'\right] \, d\mathbf{e}\,d\mathbf{e}' \, .
\end{equation*}

\subsubsection{2D case}

In such case, defining $\mathbf{e}_{\theta} = (\cos\theta,\sin\theta)$, the expression of $\beta^{R}(\mathbf{k},\mathbf{l})$ is simplified into
\begin{displaymath}
\begin{split}
\beta^{R}(\mathbf{k},\mathbf{l}) = \int_{0}^{2\pi}\int_{0}^{2\pi} \delta(\mathbf{e}_{\theta}\cdot\mathbf{e}_{\theta'}) \, \left[\int_{0}^{R} a\left(|\rho|\right)\,e^{i\,\rho\,\mathbf{k}\cdot\mathbf{e}_{\theta}}\,d\rho\right] \left[ \int_{0}^{R} b\left(|\rho|\right)\,e^{i\,\rho'\,\mathbf{l}\cdot\mathbf{e}_{\theta'}}\,d\rho'\right] \, d\theta'\,d\theta \, .
\end{split}
\end{displaymath}
Taking advantage of symmetries of the kernel, one then has
\begin{equation*}
\beta^{R}(\mathbf{k},\mathbf{l}) 
= \int_{0}^{\pi}\int_{0}^{\pi} \delta(\mathbf{e}_{\theta}\cdot\mathbf{e}_{\theta'})\,\left[\int_{-R}^{R}a\left(|\rho|\right)\,e^{i\,\rho\,\mathbf{k}\cdot\mathbf{e}_{\theta}}\,d\rho\right] \, \left[\int_{-R}^{R}b\left(|\rho|\right)\,e^{i\,\rho'\,\mathbf{l}\cdot\mathbf{e}_{\theta'}}\,d\rho'\right] \, d\theta'\,d\theta
\end{equation*}
Since $\mathbf{e}_{\theta}\cdot\mathbf{e}_{\theta'} = 0$ with $\theta,\theta' \in [0,\pi]$ is equivalent to $\theta' \equiv \theta+\frac{\pi}{2} (\textnormal{mod} \pi)$, we finally get
\begin{equation*}
\beta^{R}(\mathbf{k},\mathbf{l}) = \int_{0}^{\pi} \phi_{R,a}^{2}(\mathbf{k}\cdot\mathbf{e}_{\theta}) \, \phi_{R,b}^{2}(\mathbf{k}\cdot\mathbf{e}_{\theta+\frac{\pi}{2} (\textnormal{mod}\pi)}) \, d\theta \, ,
\end{equation*}
with $\phi_{R,a}^{2}$ and $\phi_{R,b}^{2}$ defined as
\begin{equation*}
\phi_{R,a}^{2}(s) = \int_{-R}^{R}a\left(|\rho|\right)\,e^{i\rho s}\, d\rho \, , \qquad
\phi_{R,b}^{2}(s) = \int_{-R}^{R}b\left(|\rho|\right)\,e^{i\rho s}\, d\rho \, .
\end{equation*}
Since $\phi_{R,a}^{2}$ and $\phi_{R,b}^{2}$ are even, we deduce that $\theta \mapsto \phi_{R,a}^{2}(\mathbf{k}\cdot\mathbf{e}_{\theta}) \, \phi_{R,b}^{2}(\mathbf{k}\cdot\mathbf{e}_{\theta+\frac{\pi}{2} (\textnormal{mod}\pi)})$ is $\pi$-periodic. Hence, the rectangular integration formula is of infinite order and is optimal \cite{shen2011spectral}, so one can write
\begin{equation} \label{beta2D}
\beta^{R}(\mathbf{k},\mathbf{l}) \approx \cfrac{\pi}{M} \,\sum_{p\,=\,0}^{M-1} \alpha_{p}(\mathbf{k})\,\alpha_{p}'(\mathbf{l}) \, ,
\end{equation}
with
\begin{equation}\label{alpha2D}
\alpha_{p}(\mathbf{k}) = \phi_{R,a}^{2}(\mathbf{k}\cdot\mathbf{e}_{\theta_{p}}) \, , \qquad \alpha_{p}'(\mathbf{l}) = \phi_{R,b}^{2}(\mathbf{l}\cdot\mathbf{e}_{\theta_{p}+\frac{\pi}{2} (\textnormal{mod}\pi)}) \, , \qquad \theta_{p} = \cfrac{p\,\pi}{M} \, ,
\end{equation}
and a spectrally small error. 

\begin{remark} In the specific case where $a = b$, the function $\theta \mapsto \phi_{R,a}^{2}(\mathbf{k}\cdot\mathbf{e}_{\theta}) \, \phi_{R,b}^{2}(\mathbf{k}\cdot\mathbf{e}_{\theta+\frac{\pi}{2} (\textnormal{mod}\pi)})$ is $\frac{\pi}{2}$-periodic. Consequently, we have
\begin{equation}
\beta^{R}(\mathbf{k},\mathbf{l}) = 2\,\int_{0}^{\frac{\pi}{2}} \phi_{R,a}^{2}(\mathbf{k}\cdot\mathbf{e}_{\theta}) \, \phi_{R,a}^{2}(\mathbf{k}\cdot\mathbf{e}_{\theta+\frac{\pi}{2} (\textnormal{mod}\pi)}) \, d\theta \, ,
\end{equation}
and the rectangular integration formula gives
\begin{equation} \label{beta2D_aeqb}
\beta^{R}(\mathbf{k},\mathbf{l}) \approx \cfrac{\pi}{M} \,\sum_{p\,=\,0}^{M-1} \alpha_{p}(\mathbf{k})\,\alpha_{p}'(\mathbf{l}) \, ,
\end{equation}
with
\begin{equation}\label{alpha2D_aeqb}
\alpha_{p}(\mathbf{k}) = \phi_{R,a}^{2}(\mathbf{k}\cdot\mathbf{e}_{\theta_{p}}) \, , \qquad \alpha_{p}'(\mathbf{l}) = \phi_{R,a}^{2}(\mathbf{l}\cdot\mathbf{e}_{\theta_{p}+\frac{\pi}{2} (\textnormal{mod}\pi)}) \, , \qquad \theta_{p} = \cfrac{p\,\pi}{2M} \, ,
\end{equation}
 and a spectrally error.
\end{remark}

\subsubsection{3D case}

In the 3d case, tedious (but similar to the previous section) computations yield
\begin{displaymath}
\begin{split}
\beta^{R}(\mathbf{k},\mathbf{l}) &= \cfrac{1}{4}\,\int_{[0,\pi]^{4}} \sin\theta\,\sin\theta' \, \Bigg[\delta(\mathbf{e}_{\theta,\varphi}\cdot\mathbf{e}_{\theta',\varphi'}) \, \phi_{R,a}^{3}(\mathbf{k}\cdot\mathbf{e}_{\theta,\varphi}) \, \phi_{R,b}^{3}(\mathbf{l}\cdot\mathbf{e}_{\theta',\varphi'}) \\
&\quad + \delta(\mathbf{e}_{\theta,\varphi+\pi}\cdot\mathbf{e}_{\theta',\varphi'}) \,   \phi_{R,a}^{3}(\mathbf{k}\cdot\mathbf{e}_{\theta,\varphi+\pi}) \, \phi_{R,b}^{3}(\mathbf{l}\cdot\mathbf{e}_{\theta',\varphi'}) \\
&\quad + \delta(\mathbf{e}_{\theta,\varphi}\cdot\mathbf{e}_{\theta',\varphi'+\pi}) \,   \phi_{R,a}^{3}(\mathbf{k}\cdot\mathbf{e}_{\theta,\varphi}) \, \phi_{R,b}^{3}(\mathbf{l}\cdot\mathbf{e}_{\theta',\varphi'+\pi}) \\
&\quad + \delta(\mathbf{e}_{\theta,\varphi+\pi}\cdot\mathbf{e}_{\theta',\varphi'+\pi}) \, \phi_{R,a}^{3}(\mathbf{k}\cdot\mathbf{e}_{\theta,\varphi+\pi}) \, \phi_{R,b}^{3}(\mathbf{l}\cdot\mathbf{e}_{\theta',\varphi'+\pi}) \Bigg]\, d\varphi'\,d\varphi\,d\theta'\,d\theta \, .
\end{split}
\end{displaymath}
with $\phi_{R,a}^{3}$ and $\phi_{R,b}^{3}$ defined as
\begin{equation*}
\phi_{R,a}^{3}(s) = \int_{-R}^{R} a\left(|\rho|\right)\, |\rho| e^{i\rho\,s} \, d\rho \, , \qquad \phi_{R,b}^{3}(s) = \int_{-R}^{R} b\left(|\rho|\right)\, |\rho| e^{i\rho\,s} \, d\rho \, .
\end{equation*}
We can remark that both $\phi_{R,a}^{3}$ and $\phi_{R,b}^{3}$ functions are even. Hence, one has
\begin{displaymath}
\begin{split}
\beta^{R}(\mathbf{k},\mathbf{l}) &= \cfrac{1}{4}\,\int_{[0,\pi]^{4}} \sin\theta\,\sin\theta' \, \Bigg[\delta(\mathbf{e}_{\theta,\varphi}\cdot\mathbf{e}_{\theta',\varphi'}) \, \phi_{R,a}^{3}(\mathbf{k}\cdot\mathbf{e}_{\theta,\varphi}) \, \phi_{R,b}^{3}(\mathbf{l}\cdot\mathbf{e}_{\theta',\varphi'}) \\
&\quad + \delta(-\mathbf{e}_{\pi-\theta,\varphi}\cdot\mathbf{e}_{\theta',\varphi'}) \,   \phi_{R,a}^{3}(-\mathbf{k}\cdot\mathbf{e}_{\pi-\theta,\varphi}) \, \phi_{R,b}^{3}(\mathbf{l}\cdot\mathbf{e}_{\theta',\varphi'}) \\
&\quad + \delta(-\mathbf{e}_{\theta,\varphi}\cdot\mathbf{e}_{\pi-\theta',\varphi'}) \,   \phi_{R,a}^{3}(\mathbf{k}\cdot\mathbf{e}_{\theta,\varphi}) \, \phi_{R,b}^{3}(-\mathbf{l}\cdot\mathbf{e}_{\pi-\theta',\varphi'}) \\
&\quad + \delta(\mathbf{e}_{\pi-\theta,\varphi}\cdot\mathbf{e}_{\pi-\theta',\varphi'}) \, \phi_{R,a}^{3}(-\mathbf{k}\cdot\mathbf{e}_{\pi-\theta,\varphi}) \, \phi_{R,b}^{3}(-\mathbf{l}\cdot\mathbf{e}_{\pi-\theta',\varphi'}) \Bigg]\, d\varphi'\,d\varphi\,d\theta'\,d\theta \, ,
\end{split}
\end{displaymath}
so we finally get
\begin{equation*}
\begin{split}
\beta^{R}(\mathbf{k},\mathbf{l}) = \int_{[0,\pi]^{4}} \hspace{-0.5cm} \sin\theta\,\sin\theta' \, \delta(\mathbf{e}_{\theta,\varphi}\cdot\mathbf{e}_{\theta',\varphi'}) \, & \phi_{R,a}^{3}(\mathbf{k}\cdot\mathbf{e}_{\theta,\varphi}) \times \phi_{R,b}^{3}(\mathbf{l}\cdot\mathbf{e}_{\theta',\varphi'}) \, d\varphi'\,d\varphi\,d\theta'\,d\theta \, .
\end{split}
\end{equation*}
For any $\mathbf{e} \in \Sp^{2}$, we define $\mathbf{l} \mapsto \Pi_{\mathbf{e}^{\perp}}(\mathbf{l})$ as
\begin{equation*}
\Pi_{\mathbf{e}^{\perp}}(\mathbf{l}) = \mathbf{l} - (\mathbf{l}\cdot\mathbf{e})\,\mathbf{e} \, .
\end{equation*}
Hence, $\beta^{R}(\mathbf{k},\mathbf{l})$ is reformulated as follows:
\begin{displaymath}
\beta^{R}(\mathbf{k},\mathbf{l}) = \int_{[0,\pi]^{2}} \sin\theta \, \phi_{R,a}^{3}(\mathbf{k}\cdot\mathbf{e}_{\theta,\varphi}) \left[ \int_{\substack{(\theta',\varphi')\,\in\,[0,\pi]^{2} \\ \mathbf{e}_{\theta,\varphi}\cdot\mathbf{e}_{\theta',\varphi'} = 0}}\, \sin\theta' \, \phi_{R,b}^{3}(\mathbf{l}\cdot\mathbf{e}_{\theta',\varphi'}) \, d\varphi'\, d\theta' \right] \, d\varphi\,d\theta \, .
\end{displaymath}
Denoting with $\Sp_{+}^{2}$ the half-sphere, we can write
\begin{displaymath}
\int_{\substack{(\theta',\varphi')\,\in\,[0,\pi]^{2} \\ \mathbf{e}_{\theta,\varphi}\cdot\mathbf{e}_{\theta',\varphi'} = 0}}\, \sin\theta' \, \phi_{R,b}^{3}(\mathbf{l}\cdot\mathbf{e}_{\theta',\varphi'}) \, d\varphi'\, d\theta' = \int_{\Sp_{+}^{2} \cap e_{\theta,\varphi}^{\perp}} \phi_{R,b}^{3}(\mathbf{l}\cdot\mathbf{e}')\, d\mathbf{e}' \, .
\end{displaymath}
We define $\mathbf{e}_{\theta_{\mathbf{l}},\varphi_{\mathbf{l}}} = \cfrac{\Pi_{\mathbf{e}_{\theta,\varphi}^{\perp}}(\mathbf{l})}{\left|\Pi_{\mathbf{e}_{\theta,\varphi}^{\perp}}(\mathbf{l})\right|}$. Hence the 2D plan $\mathbf{e}_{\theta,\varphi}^{\perp}$ is provided with the direct orthonormal basis $\{\mathbf{e}_{\theta_{\mathbf{l}},\varphi_{\mathbf{l}}} \, , \, \mathbf{e}_{\theta,\varphi} \times \mathbf{e}_{\theta_{\mathbf{l}},\varphi_{\mathbf{l}}} \}$. In addition, the intersection $\Sp_{+}^{2} \cap \mathbf{e}_{\theta,\varphi}^{\perp}$ is the unit half-circle $\Sp_{+}^{1}$ and
\begin{displaymath}
\mathbf{l}\cdot\mathbf{e}' = \left|\Pi_{\mathbf{e}_{\theta,\varphi}^{\perp}}(\mathbf{l})\right| \, \mathbf{e}_{\theta_{\mathbf{l}},\varphi_{\mathbf{l}}} \cdot \mathbf{e}' \, ,
\end{displaymath}
where $\mathbf{e}_{\theta_{\mathbf{l}},\varphi_{\mathbf{l}}} \cdot \mathbf{e}'$ is no more than the cosinus of the angle between $\mathbf{e}_{\theta_{\mathbf{l}},\varphi_{\mathbf{l}}}$ and $\mathbf{e}' \in \Sp_{+}^{2} \cap e_{\theta,\varphi}^{\perp}$. Consequently, we obtain
\begin{displaymath}
\begin{split}
\int_{\substack{(\theta',\varphi')\,\in\,[0,\pi]^{2} \\ \mathbf{e}_{\theta,\varphi}\cdot\mathbf{e}_{\theta',\varphi'} = 0}}\, \sin\theta' \, \phi_{R,b}^{3}(\mathbf{l}\cdot\mathbf{e}_{\theta',\varphi'}) \, d\varphi'\, d\theta' &= \int_{0}^{\pi} \sin\theta'\, \phi_{R,b}^{3}\left( \left|\Pi_{\mathbf{e}_{\theta,\varphi}^{\perp}}(\mathbf{l})\right| \, \cos\theta'\right) \, d\theta' \\
&=: \psi_{R,b}^{3}(\Pi_{\mathbf{e}_{\theta,\varphi}^{\perp}}(\mathbf{l})) \, ,
\end{split}
\end{displaymath}
and
\begin{equation*}
\beta^{R}(\mathbf{k},\mathbf{l}) = \int_{[0,\pi]^{2}} \sin\theta\, \phi_{R,a}^{3}(\mathbf{k}\cdot\mathbf{e}_{\theta,\varphi}) \psi_{R,b}^{3}(\Pi_{\mathbf{e}_{\theta,\varphi}^{\perp}}(\mathbf{l})) \, d\varphi\,d\theta \, .
\end{equation*}
We can notice that the function $\kappa(\theta,\varphi) = \sin\theta\, \phi_{R,a}^{3}(\mathbf{k}\cdot\mathbf{e}_{\theta,\varphi}) \psi_{R,b}^{3}(\Pi_{\mathbf{e}_{\theta,\varphi}^{\perp}}(\mathbf{l}))$ is $\pi$-periodic in $\theta$ and in $\varphi$. Indeed, reminding that $\mathbf{e}_{\theta,\varphi+\pi} = \mathbf{e}_{\pi-\theta,\varphi}$ firstly gives
\begin{displaymath}
\kappa(\theta,\varphi+\pi) = \kappa(\pi-\theta,\varphi) \, ,
\end{displaymath}
so the $\pi$-periodicity in $\varphi$ will be proved along with the $\pi$-periodicity in $\theta$. This last property can be easily proved since $\sin$ is $\pi$-antiperiodic, $\mathbf{e}_{\theta+\pi,\varphi} = -\mathbf{e}_{\theta,\varphi}$, $\phi_{R,a}^{3}$ is even, and 
\begin{displaymath}
\Pi_{\mathbf{e}_{\theta+\pi,\varphi}^{\perp}}(\mathbf{l}) = \Pi_{\mathbf{e}_{\theta,\varphi}^{\perp}}(\mathbf{l}) \, ,
\end{displaymath}
which makes the basis $\{\mathbf{e}_{\theta_{\mathbf{l}},\varphi_{\mathbf{l}}} \, , \, \mathbf{e}_{\theta+\pi,\varphi} \times \mathbf{e}_{\theta_{\mathbf{l}},\varphi_{\mathbf{l}}} \}$ indirect. Consequently, we have
\begin{displaymath}
\psi_{R,b}^{3}(\Pi_{\mathbf{e}_{\theta+\pi,\varphi}^{\perp}}(\mathbf{l})) = -\psi_{R,b}^{3}(\Pi_{\mathbf{e}_{\theta,\varphi}^{\perp}}(\mathbf{l})) \, .
\end{displaymath}

Since $\kappa$ is $\pi$-periodic in $\theta$ and $\varphi$, the rectangular integration rule on $[0,\pi]^{2}$ is of infinite order and optimal, so we can write
\begin{equation} \label{beta3D}
\beta^{R}(\mathbf{k},\mathbf{l}) \approx \cfrac{\pi^{2}}{M_{1}\,M_{2}} \sum_{p\,=\,0}^{M_{1}-1}\sum_{q\,=\,0}^{M_{2}-1} \sin\theta_{p}\, \phi_{R,a}^{3}(\mathbf{k}\cdot\mathbf{e}_{\theta_{p},\varphi_{q}}) \psi_{R,b}^{3}(\Pi_{\mathbf{e}_{\theta_{p},\varphi_{q}}^{\perp}}(\mathbf{l})) \, ,
\end{equation}
with $\mathbf{e}_{\theta_{p},\varphi_{q}} = (\sin\theta_{p}\,\cos\varphi_{q}, \sin\theta_{p}\,\sin\varphi_{q},\cos\theta_{p})$, $\theta_{p} = \cfrac{p\,\pi}{M_{1}}$, $\varphi_{q} = \cfrac{\pi\,q}{M_{2}}$, and

\begin{equation} \label{alpha3D}
\begin{gathered}
\phi_{R,a}^{3}(s) = \int_{-R}^{R} a\left(|\rho|\right)\, |\rho| e^{i\rho\,s} \, d\rho \, , \quad
\phi_{R,b}^{3}(s) = \int_{-R}^{R} b\left(|\rho|\right)\, |\rho| e^{i\rho\,s} \, d\rho \, , \\
\psi_{R,b}^{3}(\Pi_{\mathbf{e}_{\theta,\varphi}^{\perp}}(\mathbf{l})) = \int_{0}^{\pi} \sin\theta'\, \phi_{R,b}^{3}\left(\left|\Pi_{\mathbf{e}_{\theta,\varphi}^{\perp}}(\mathbf{l})\right| \, \cos\theta'\right) \, d\theta' \, .
\end{gathered}
\end{equation}

\subsection{Fast computation of the collision modes}

It is possible to gather the formula \eqref{beta2D} and \eqref{beta3D} under the common expression
\begin{equation}\label{betadD}
\beta^{R}(\mathbf{k},\mathbf{l}) = C\, \sum_{p\,=\,0}^{P-1} \alpha_{R,p}(\mathbf{k})\,\alpha_{R,p}'(\mathbf{l}) \, ,
\end{equation}
where $C$, $P$, $\alpha_{p}$, $\alpha_{p}'$ are defined as follows:

\paragraph{2D case} Let $C = \cfrac{\pi}{M}$, $P = M$ and, for any $p = 0,\ldots,M-1$,
\begin{equation} \label{alphap}
\left \{\begin{gathered}
\alpha_{R,p}(\mathbf{k}) = \phi_{R,a}^{2}(\mathbf{k}\cdot\mathbf{e}_{\theta_{p}}) \, , \quad
\alpha_{R,p}'(\mathbf{l}) = \phi_{R,b}^{2}(\mathbf{l}\cdot\mathbf{e}_{\theta_{p}+\frac{\pi}{2} (\textnormal{mod}\pi)}) \, , \\
\theta_{p} = \cfrac{p\,\pi}{M} \, , \quad \mathbf{e}_{\theta_{p}} = (\cos\theta_{p},\sin\theta_{p}) \, , \\
\phi_{R,a}^{2}(s) = \int_{-R}^{R} a\left(|\rho|\right)\, e^{i\rho\,s} \, d\rho \, , \quad
\phi_{R,b}^{2}(s) = \int_{-R}^{R} b\left(|\rho|\right)\, e^{i\rho\,s} \, d\rho \, , \\
\end{gathered}\right.
\end{equation}
In the specific case of 2D Maxwellian molecules ($\gamma = 0$, $b\equiv 1$) so we get
\[
\tilde{B}(\mathbf{z},\mathbf{y}) = 2C_{\Phi} \, \quad a(|\mathbf{z}|) = 2C_{\Phi} \, , \quad b(|\mathbf{y}|) = 1\, .
\]
This gives
\begin{displaymath}
\beta^{R}(\mathbf{k},\mathbf{l}) = \cfrac{\pi}{M}\, \sum_{p\,=\,0}^{P-1} \alpha_{R,p}(\mathbf{k})\,\alpha_{R,p}'(\mathbf{l}) \, ,
\end{displaymath}
with $\alpha_{R,p}(\mathbf{k})$ and $\alpha_{R,p}'(\mathbf{l})$ computed as follows:
\[
\begin{gathered}
\alpha_{R,p}(\mathbf{k}) = 2C_{\Phi}\,\phi_{R}^{2}(\mathbf{k}\cdot\mathbf{e}_{\theta_{p}}) \, , \quad
\alpha_{R,p}'(\mathbf{l}) = \phi_{R}^{2}(\mathbf{l}\cdot\mathbf{e}_{\theta_{p}+\frac{\pi}{2} (\textnormal{mod}\pi)}) \, , \\
\phi_{R}^{2}(s) = 2R\,\textnormal{Sinc}(Rs) \, .
\end{gathered}
\]

\paragraph{3D case} Let $C = \cfrac{\pi^{2}}{M_{1}M_{2}}$, $P = M_{1}M_{2}$, and, for any $p = 0,\ldots,M_{1}-1$, $q = 0,\ldots,M_{2}-1$,
\begin{equation} \label{alphapprime}
\left\{\begin{gathered}
\alpha_{R,q+M_{2}p}(\mathbf{k}) = \sin\theta_{p}\, \phi_{R,a}^{3}(\mathbf{k}\cdot\mathbf{e}_{\theta_{p},\varphi_{q}}) \, , \quad
\alpha_{R,q+M_{2}p}'(\mathbf{l}) = \psi_{R,b}^{3}(\Pi_{\mathbf{e}_{\theta_{p},\varphi_{q}}^{\perp}}(\mathbf{l})) \, , \\
\theta_{p} = \cfrac{p\,\pi}{M_{1}} \, , \quad \varphi_{q} = \cfrac{\pi\,q}{M_{2}} \, , \quad
\mathbf{e}_{\theta_{p},\varphi_{q}} = (\sin\theta_{p}\,\cos\varphi_{q}, \sin\theta_{p}\,\sin\varphi_{q},\cos\theta_{p}) \, , \\
\phi_{R,a}^{3}(s) = \int_{-R}^{R} a\left(|\rho|\right)\, |\rho| e^{i\rho\,s} \, d\rho \, , \quad
\phi_{R,b}^{3}(s) = \int_{-R}^{R} b\left(|\rho|\right)\, |\rho| e^{i\rho\,s} \, d\rho \, , \\
\psi_{R,b}^{3}(\Pi_{\mathbf{e}_{\theta,\varphi}^{\perp}}(\mathbf{l})) = \int_{0}^{\pi} \phi_{R,b}^{3}\left(\left|\Pi_{\mathbf{e}_{\theta,\varphi}^{\perp}}(\mathbf{l})\right| \, \cos\theta'\right) \, d\theta' \, .
\end{gathered}\right.
\end{equation}
In the specific case of 3D hard spheres ($\gamma = 1$, $b\equiv 1$), we get
\[
\tilde{B}(\mathbf{z},\mathbf{y}) = 4C_{\Phi} \, , \quad
a(|\mathbf{z}|) = 4C_{\Phi} \, , \qquad b(|\mathbf{y}|) = 1\, .
\]
This gives
\begin{displaymath}
\beta^{R}(\mathbf{k},\mathbf{l}) = \cfrac{\pi^{2}}{M_{1}M_{2}}\, \sum_{p\,=\,0}^{P-1} \alpha_{R,p}(\mathbf{k})\,\alpha_{R,p}'(\mathbf{l}) \, ,
\end{displaymath}
with $M = M_{1}M_{2}$, $\alpha_{R,p}(\mathbf{k})$ and $\alpha_{R,p}'(\mathbf{l})$ computed as follows:
\[
\begin{gathered}
\alpha_{R,q+M_{2}p}(\mathbf{k}) = 4C_{\Phi}\,\sin\theta_{p}\, \phi_{R}^{3}(\mathbf{k}\cdot \mathbf{e}_{\theta_{p},\varphi_{q}}) \, , \quad
\alpha_{R,q+M_{2}p}'(\mathbf{l}) = \psi_{R}^{3}\left(\left|\Pi_{\mathbf{e}_{\theta_{p},\varphi_{q}}^{\perp}}(\mathbf{l})\right|\right) \, , \\
\phi_{R}^{3}(s) = R^{2}\,\left(2\,\textnormal{Sinc}(sR) - \textnormal{Sinc}^{2}\left(\cfrac{sR}{2}\right)\right) \, , \quad
\psi_{R}^{3}(s) = 2R^{2}\,\textnormal{Sinc}^{2}\left(\cfrac{Rs}{2}\right) \, .
\end{gathered}
\]

Both these formula allow for an overall computational cost of $\mathcal{O}(N^{2d_\mathbf{v}} \log_2 N)$. The practical details of the implementation are summarized in the appendix.

\section{A rescaling velocity model of the BNE}
\label{sec:Rescaling}

	We have seen in Section \ref{sec:stationary_state} that the BNE  can exhibits blowup in the quantum degenerative fermionic case,
	but with no a priori knowledge on the rate of explosion. In order to simulate numerically such result we propose to adopt the rescaling velocity framework that was developed in the series of papers \cite{Filbet:2004,FilbetRey:2012,ReyTan:2014}.
	It consists in rescaling dynamically the equation, hence allowing to follow very accurately the spread and concentration of the solution to the equation.
	
	\subsection{A general scaling framework}
We now introduce a given function $\omega:\R_{+} \times \Omega_{\mathbf{x}} \to \R_{+}^{*}$. Such function has to be positive, and represents accurately the support in velocity of the distribution $f(t,\mathbf x, \cdot)$ solution to \eqref{Boltzmann_start}. 
 Candidates for $\omega$ will be presented later. At present time, we focus on the function $G:\R_{+} \times \Omega_{\mathbf{x}} \times \R^{d_{\mathbf{v}}}$ that is introduced as follows: for any $(t,\mathbf{x},\mathbf{v}) \in \R_{+} \times \Omega_{\mathbf{x}} \times \R^{d_{\mathbf{v}}}$, 
\begin{equation} \label{rescaling_def_fg}
f(t,\mathbf{x},\mathbf{v}) = \mu\,\left(\cfrac{\pi\,\omega(t,\mathbf{x})}{L}\right)^{d_{\mathbf{v}}} \, G\left(t,\mathbf{x},\cfrac{\pi\,\omega(t,\mathbf{x})}{L}\left(\mathbf{v}-\lambda\,\mathbf{u}(t,\mathbf{x})\right)\right) \, ,
\end{equation}
with $\mu > 0$, $L > 0$ and $\lambda \in \{0,1\}$ to be precised later. We can remark that
\[
\rho(t,\mathbf{x}) = \int_{\R^{d_{\mathbf{v}}}} f(t,\mathbf{x},\mathbf{v})\, d\mathbf{v} = \mu\, \int_{\R^{d_{\mathbf{v}}}} G(t,\mathbf{x},\bm{\xi})\, d\bm{\xi} \, , \qquad 
\int_{\R^{d_{\mathbf{v}}}} \bm{\xi}\,G(t,\mathbf{x},\bm{\xi}) \, d\bm{\xi}  = \bm{0} \, ,
\]
for all $(t,\mathbf{x}) \in \R_{+} \times \Omega_{\mathbf{x}}$. We shall exhibit the evolution equation that is satisfied by $G$ by injecting \eqref{rescaling_def_fg} in \eqref{Boltzmann_start}. 

 We first define the change of variables
\begin{equation*}
\bm{\xi}(t,\mathbf{x},\mathbf{v}) = \cfrac{\pi\,\omega(t,\mathbf{x})}{L}\left(\mathbf{v}-\lambda\,\mathbf{u}(t,\mathbf{x})\right) \, , \qquad \mathbf{v}(t,\mathbf{x},\bm{\xi}) = \cfrac{L\,\bm{\xi}}{\pi\,\omega(t,\mathbf{x})} + \lambda\,\mathbf{u}(t,\mathbf{x}) \, ,
\end{equation*}
so we have
\begin{equation*}
f(t,\mathbf{x},\mathbf{v}) = \mu\,\left(\cfrac{\pi\,\omega(t,\mathbf{x})}{L}\right)^{d_{\mathbf{v}}}\, G\left(t,\mathbf{x},\bm{\xi}(t,\mathbf{x},\mathbf{v})\right)  \Leftrightarrow
G(t,\mathbf{x},\bm{\xi}) = \mu^{-1} \left(\cfrac{\pi\,\omega(t,\mathbf{x})}{L}\right)^{-d_{\mathbf{v}}}\, f\left(t,\mathbf{x},\mathbf{v}(t,\mathbf{x},\bm{\xi})\right) \, .
\end{equation*}
The time derivative of $f$ is rewritten as
\begin{equation} \label{rescaling_dt}
\begin{split}
\D_{t}f(t,\mathbf{x},\mathbf{v}) &= \mu\, \left(\cfrac{\pi\,\omega(t,\mathbf{x})}{L}\right)^{d_{\mathbf{v}}} \Bigg[ d_{\mathbf{v}}\, \cfrac{\D_{t}\omega(t,\mathbf{x})}{\omega(t,\mathbf{x})}\, G\left(t,\mathbf{x},\bm{\xi}(t,\mathbf{x},\mathbf{v})\right) \\
&\qquad + (\D_{t}G)\left(t,\mathbf{x},\bm{\xi}(t,\mathbf{x},\mathbf{v})\right) \\
&\qquad + \cfrac{\D_{t}\omega(t,\mathbf{x})}{\omega(t,\mathbf{x})} \, \bm{\xi}(t,\mathbf{x},\mathbf{v}) \cdot (\nabla_{\bm{\xi}}G)\left(t,\mathbf{x},\bm{\xi}(t,\mathbf{x},\mathbf{v})\right) \\
&\qquad - \lambda\,\cfrac{\pi\,\omega(t,\mathbf{x})}{L} \, \D_{t}\mathbf{u}(t,\mathbf{x})\cdot (\nabla_{\bm{\xi}}G)\left(t,\mathbf{x},\bm{\xi}(t,\mathbf{x},\mathbf{v})\right) \Bigg] \, .
\end{split}
\end{equation}

\indent The free transport term is rewritten as
\begin{equation} \label{rescaling_transport}
\begin{split}
\mathbf{v} \cdot \nabla_{\mathbf{x}}f(t,\mathbf{x},\mathbf{v}) &=  \left( \cfrac{L}{\pi\,\omega(t,\mathbf{x})}\, \bm{\xi}(t,\mathbf{x},\mathbf{v}) + \lambda\,\mathbf{u}(t,\mathbf{x})\right) \, \mu\, \left(\cfrac{\pi\,\omega(t,\mathbf{x})}{L}\right)^{d_{\mathbf{v}}} \\
& \qquad \cdot \Bigg[ \cfrac{d_{\mathbf{v}}}{\omega(t,\mathbf{x})}\, \nabla_{\mathbf{x}}\omega(t,\mathbf{x})\, G\left(t,\mathbf{x},\bm{\xi}(t,\mathbf{x},\mathbf{v})\right) \\
&\qquad \quad + (\nabla_{\mathbf{x}}G)\left(t,\mathbf{x},\bm{\xi}(t,\mathbf{x},\mathbf{v})\right) \\
&\qquad \quad + \left(\cfrac{\D_{x_{j}}\omega(t,\mathbf{x})}{\omega(t,\mathbf{x})}\, \bm{\xi}(t,\mathbf{x},\mathbf{v}) \cdot (\nabla_{\bm{\xi}}G)\left(t,\mathbf{x},\bm{\xi}(t,\mathbf{x},\mathbf{v})\right)\right)_{j\,=\,1,\dots,d_{\mathbf{v}}} \\
&\qquad \quad - \cfrac{\pi\,\omega(t,\mathbf{x})}{L}\, \lambda\, \left(\D_{x_{j}}\mathbf{u}(t,\mathbf{x}) \cdot (\nabla_{\bm{\xi}}G)\left(t,\mathbf{x},\bm{\xi}(t,\mathbf{x},\mathbf{v})\right)\right)_{j\,=\,1,\dots,d_{\mathbf{v}}} \Bigg] \, .
\end{split}
\end{equation}

Concerning the rewritting of the collision operator $\mathcal{Q}_\alpha(f)$, we first remark that
\begin{displaymath}
B\left(|\mathbf{v}-\mathbf{v}_{*}|, \cos\theta\right) = \left(\cfrac{\pi\,\omega(t,\mathbf{x})}{L}\right) ^{-\gamma}\, B\left(\left| \bm{\xi}(t,\mathbf{x},\mathbf{v}) - \bm{\xi}(t,\mathbf{x},\mathbf{v}_{*})\right|, \cos\theta\right) \, ,
\end{displaymath}
for any $(\mathbf{v},\mathbf{v}_{*},\theta) \in \R^{d_{\mathbf{v}}} \times \R^{d_{\mathbf{v}}} \times \R$. In the same spirit, we can remark that, with the convention \eqref{changevar_sigma}, we have
\begin{displaymath}
\begin{split}
f(t,\mathbf{x},\mathbf{v}') &= \left(\cfrac{\pi\,\omega(t,\mathbf{x})}{L}\right)^{d_{\mathbf{v}}}\, \mu\, G\Bigg(t, \mathbf{x}, \cfrac{\bm{\xi}(t,\mathbf{x},\mathbf{v})+\bm{\xi}(t,\mathbf{x},\mathbf{v}_{*})}{2} + \cfrac{\left| \bm{\xi}(t,\mathbf{x},\mathbf{v})-\bm{\xi}(t,\mathbf{x},\mathbf{v}_{*})\right|}{2}\, \bm{\sigma}\Bigg) \, , \\
f(t,\mathbf{x},\mathbf{v}_{*}') &= \left(\cfrac{\pi\,\omega(t,\mathbf{x})}{L}\right)^{d_{\mathbf{v}}}\, \mu\, G\Bigg(t, \mathbf{x}, \cfrac{\bm{\xi}(t,\mathbf{x},\mathbf{v})+\bm{\xi}(t,\mathbf{x},\mathbf{v}_{*})}{2} - \cfrac{\left| \bm{\xi}(t,\mathbf{x},\mathbf{v})-\bm{\xi}(t,\mathbf{x},\mathbf{v}_{*})\right|}{2}\, \bm{\sigma}\Bigg) \, , \, .
\end{split}
\end{displaymath}
Consequently, the collision operator writes as
\begin{equation} \label{rescaling_collision}
\mathcal{Q}_\alpha(f)(t,\mathbf{x},\mathbf{v}) = \left(\cfrac{\pi\,\omega(t,\mathbf{x})}{L}\right)^{d_{\mathbf{v}}-\gamma} \, \mu^{2}\, \mathcal{Q}_{\alpha\mu(\frac{\pi\omega(t,\mathbf{x})}{L})^{d_{\mathbf{v}}}}(G)\left(t,\mathbf{x},\bm{\xi}(t,\mathbf{x},\mathbf{v})\right) \, .
\end{equation}

Gathering \eqref{rescaling_dt}, \eqref{rescaling_transport} and \eqref{rescaling_collision}, we get the evolution equation for $G$:
\begin{equation} \label{rescaled_model}
\begin{split}
 \D_{t}G & +   \nabla_{\bm{\xi}} \cdot \left ( \left ( \cfrac{\D_{t}\omega}{\omega} \, \bm{\xi} - \lambda\,\cfrac{\pi\,\omega}{L} \, \D_{t}\mathbf{u} \right ) G\right )G \\
&+ \tau\, \left( \cfrac{L}{\pi\,\omega}\, \bm{\xi} + \lambda\,\mathbf{u}\right) \cdot \Bigg[ \cfrac{d_{\mathbf{v}}}{\omega}\, \nabla_{\mathbf{x}}\omega\, G + \nabla_{\mathbf{x}}G + \left(\cfrac{\D_{x_{j}}\omega}{\omega}\, \bm{\xi} \cdot \nabla_{\bm{\xi}}G\right)_{j\,=\,1,\dots,d_{\mathbf{v}}} \\
&\qquad \qquad \qquad \qquad \qquad - \cfrac{\pi\,\omega}{L}\, \lambda\, \left(\D_{x_{j}}\mathbf{u} \cdot \nabla_{\bm{\xi}}G\right)_{j\,=\,1,\dots,d_{\mathbf{v}}} \Bigg] \\
= c \, \mu\, &\left(\cfrac{\pi\,\omega}{L}\right)^{-\gamma} \, \mathcal{Q}_{\alpha\mu(\frac{\pi\omega}{L})^{d_{\mathbf{v}}}}(G)
\end{split}
\end{equation}

	\subsection{Application to the space homogeneous BNE}
\label{spectral_homogeneous}
We now focus on the following homogeneous quantum Boltzmann equation written under divergence form:
\begin{equation*}
\left\{
\begin{array}{l}
\D_{t}G + \nabla_{\bm{\xi}} \cdot \left(\left(\cfrac{\D_{t}\omega}{\omega} \, \bm{\xi} - \lambda\,\cfrac{\pi\,\omega}{L} \, \D_{t}\mathbf{u} \right)\, G \right) = c \, \mu\, \left(\cfrac{\pi\,\omega}{L}\right)^{-\gamma} \, \mathcal{Q}_{\alpha\mu(\frac{\pi\omega}{L})^{d_{\mathbf{v}}}}^{R}(G) \, , \\
G(t=0,\bm{\xi}) = \cfrac{1}{\mu}\,\left(\cfrac{L}{\pi\,\omega^{0}}\right)^{d_{\mathbf{v}}} \, f^{0}\left(\cfrac{L}{\pi\,\omega^{0}}\, \bm{\xi} + \mathbf{u}^{0}\right) \, ,
\end{array}
\right.
\end{equation*}
At this point, the classical method can be distinguished from the velocity rescaling method as follows:

\paragraph{The classical method} It is characterized by the following parameter choice:
\begin{equation*}
\lambda = 0 \, , \qquad \omega \equiv 1 \, , \qquad \mu = \left(\cfrac{L}{\pi}\right)^{d_{\mathbf{v}}} \, ,
\end{equation*}
so the equation to be discretized on $\R_{+} \times [-\pi,\pi]^{d_{\mathbf{v}}}$ is
\begin{equation*}
\left\{
\begin{array}{l}
\D_{t}G = c \, \left(\frac{\pi}{L}\right)^{-d_{\mathbf{v}}-\gamma} \, \mathcal{Q}_{\alpha}^{R}(G) \, , \\
G(t=0,\bm{\xi}) = f^{0}\left(\cfrac{L}{\pi}\,\bm{\xi}\right) \, ,
\end{array}
\right.
\end{equation*}
and $f$ is obtained from $G$ on $\R_{+} \times \Omega_{\mathbf{v}}$ with $\Omega_{\mathbf{v}} = [-L,L]^{d_{\mathbf{v}}}$.

\paragraph{The velocity rescaling method} It is characterized by 
\begin{equation*}
\lambda = 1 \, , \qquad \mu = 1 \, ,
\end{equation*}
and $\omega$ to be identified later, so the equation in $G$ to be solved on $\R_{+}\times[-\pi,\pi]^{d_{\mathbf{v}}}$ is
\begin{equation*}
\left\{
\begin{array}{l}
\D_{t}G + \nabla_{\bm{\xi}} \cdot \left(\left(\cfrac{\D_{t}\omega}{\omega} \, \bm{\xi} - \cfrac{\pi\,\omega}{L} \, \D_{t}\mathbf{u} \right)\, G \right) = c \, \left(\cfrac{\pi\,\omega}{L}\right)^{-\gamma} \, \mathcal{Q}_{\alpha(\frac{\pi\omega}{L})^{d_{\mathbf{v}}}}^{R}(G) \, , \\
G(t=0,\bm{\xi}) = \left(\cfrac{L}{\pi\,\omega^{0}}\right)^{d_{\mathbf{v}}} \, f^{0}\left(\cfrac{L}{\pi\,\omega^{0}}\, \bm{\xi} + \mathbf{u}^{0}\right) \, ,
\end{array}
\right.
\end{equation*}
and $f(t,\cdot)$ is obtained from $G(t,\cdot)$ on $\Omega_{\mathbf{v}}(t) = \mathbf{u}(t)+\left[-\cfrac{L}{\omega(t)},\cfrac{L}{\omega(t)}\right]^{d_{\mathbf{v}}}$ for any $t \in \R_{+}$.

\subsection{Numerical scheme for the velocity rescaled BNE}
We assume from now that the support of $G(t=0,\cdot)$ is included in $B(0,S) \subset [-\pi,\pi]^{d_{\mathbf{v}}}$ with $S > 0$ and $R > 0$ such that 
\begin{displaymath}
R \geq 2S \, , \qquad \pi \geq \cfrac{R+(1+\sqrt{2})S}{2} \, .
\end{displaymath}
We consider a time grid $(t^{n})_{n\,\in\,\N} \subset \R_{+}$ with $\Delta t^{n} = t^{n+1}-t^{n} > 0$, $t^{0} = 0$ and $G^{n}:[-\pi,\pi]^{d_{\mathbf{v}}} \to \R$ being an approximation of $G(t^{n},\cdot)$. Finally, we define $\Delta v^{n}$ and $\Delta \xi$ as
\begin{displaymath}
\Delta v^{n} =  \left(\cfrac{2L}{\omega^{n} \, N}\right)^{d_{\mathbf{v}}} \, , \qquad \Delta \xi = \left ( \cfrac{2\pi}{N} \right ) ^{d_{\mathbf{v}}} ,
\end{displaymath}
and we consider a space discretization of $G^{n}$ with 
\begin{displaymath}
G^{n}(\bm{\xi}_{\mathbf{j}}) \approx G_{\mathbf{j}}^{n} \, , \qquad \xi_{j_{d},d} = \cfrac{2j_{d}\pi}{N} \, , \quad j_{d} = -\cfrac{N}{2},\dots,\cfrac{N}{2}-1\, , \quad d=1,\dots,d_{\mathbf{v}} \, ,
\end{displaymath}
and
\begin{displaymath}
G_{\mathbf{j}}^{0} := G(0,\bm{\xi}_{\mathbf{j}}) = \cfrac{1}{\mu}\,\left(\cfrac{L}{\pi\,\omega^{0}}\right)^{d_{\mathbf{v}}} \, f^{0}\left(\cfrac{L}{\pi\,\omega^{0}}\, \bm{\xi}_{\mathbf{j}} + \lambda\,\mathbf{u}^{0}\right) \, .
\end{displaymath}
Defining $f_{h}^{n} = (f_{\mathbf{j}}^{n})_{\mathbf{j} \, \in \, I_{\mathbf{N}}}$ with
\begin{displaymath}
f_{\mathbf{j}}^{n} := \mu\, \left(\cfrac{\pi\,\omega^{n}}{L}\right)^{d_{\mathbf{v}}} \, G_{\mathbf{j}}^{n} \, ,
\end{displaymath}
and $v_{j_{d},d}^{n} = \lambda\,u_{d}^{n} + \cfrac{2L\,j_{d}}{N\,\omega^{n}}$ for any $j_{d} = -\cfrac{N}{2},\dots,\cfrac{N}{2}-1$, $d = 1,\dots,d_{\mathbf{v}}$ ($u^n_d$ being an approximation of $u_d(t^n)$, to be defined later), we have
\begin{displaymath}
G(t^{n},\bm{\xi}_{\mathbf{j}}) \approx G_{\mathbf{j}}^{n} \quad \Longleftrightarrow \quad f(t^{n},\mathbf{v}_{\mathbf{j}}^{n}) \approx f_{\mathbf{j}}^{n}
\end{displaymath}
for any $\mathbf{j} \in I_{{N}}$ and $n \in \N$. This provides an approximation of $f(t,\cdot)$ on $\Omega_{\mathbf{v}}(t)$. 

\begin{remark}
\textnormal{In most cases, it is possible to identify the limit state $f^{\infty}$ such that $f \to f^{\infty}$ as $t \to +\infty$. For this purpose, we refer to Section \ref{sec:stationary_state} in order to identify the expression of $f^{\infty}$. This function is characterized thanks to the moments of $f^{0}$ (mass $\rho^{0}$, averaged velocity $\mathbf{u}^{0}$, internal energy $e^{0}$...). Assuming that $f^{\infty}$ is analytically known, we build $f_{h}^{\infty,n} = (f_{\mathbf{j}}^{\infty,n})_{\mathbf{j}\,\in\,I_{N}}$ for any $n$ such that}
\begin{displaymath}
f_{\mathbf{j}}^{\infty,n} = f^{\infty}(\mathbf{v}_{\mathbf{j}}^{n}) \, ,
\end{displaymath}
\textnormal{which gives an approximation of $f^{\infty}$ on $\Omega_{\mathbf{v}}^{n}$. We can also introduce $G_{h}^{\infty,n} = (G_{\mathbf{j}}^{\infty,n})_{\mathbf{j} \, \in \, I_{{N}}}$ as}
\begin{displaymath}
G_{\mathbf{j}}^{\infty,n} = \cfrac{1}{\mu}\, \left( \cfrac{L}{\pi\,\omega^{n}} \right)^{d_{\mathbf{v}}} \, f_{\mathbf{j}}^{\infty,n} \, , \qquad \forall\, \mathbf{j} \in I_{{N}} \, .
\end{displaymath}
\end{remark}

We describe now the procedure associated to a forward Euler time-semi-discretization:
\begin{enumerate}
\item We assume that $G_{h}^{n}$ (or equivalently $f_{h}^{n}$) is known along with the mass $\rho^{n}$, the averaged velocity $\mathbf{u}^{n}$, the internal energy $e^{n}$, and the velocity scale $\omega^{n}$ at time step $n$.
\item We compute $\mathcal{Q}_{\alpha\mu(\frac{\pi\omega^{n}}{L})^{d_{\mathbf{v}}}}^{R}(G_{h}^{n})_{\mathbf{j}}$ for any $\mathbf{j}\in I_{{N}}$:
\begin{enumerate}
\item We compute the Fourier modes of $G_{h}^{n}$ and get
\begin{displaymath}
\widehat{G}_{\mathbf{k}}^{n} = \cfrac{1}{N}\, \sum_{\mathbf{j}\,\in\,I_N} G_{\mathbf{j}}^{n}\, e^{-2i\pi \, \frac{\mathbf{k}\cdot\mathbf{j}}{N}}\, , \qquad \forall \mathbf{k} \in I_{{N}} \, .
\end{displaymath}

\item We compute $\mathcal{Q}_{i,c}^{R}(G_{h}^{n},G_{h}^{n})_{\mathbf{j}}$ ($i = 1,2$) and $\mathcal{Q}_{i,q}^{R}(G_{h}^{n},G_{h}^{n},G_{h}^{n})_{\mathbf{j}}$ ($i=1,2,3,4$) for any $\mathbf{j} \in I_N$.

\item We get $\mathcal{Q}_{\alpha\mu(\frac{\pi\omega^{n}}{L})^{d_{\mathbf{v}}}}^{R}(G_{h}^{n})_{\mathbf{j}}$ as
\begin{displaymath}
\begin{split}
\mathcal{Q}_{\alpha\mu(\frac{\pi\omega^{n}}{L})^{d_{\mathbf{v}}}}^{R}(G_{h}^{n})_{\mathbf{j}} = &\mathcal{Q}_{1,c}^{R}(G_{h}^{n},G_{h}^{n})_{\mathbf{j}} - \mathcal{Q}_{2,c}^{R}(G_{h}^{n},G_{h}^{n})_{\mathbf{j}} \\
&- \alpha\mu(\frac{\pi\omega^{n}}{L})^{d_{\mathbf{v}}} \Big[ \mathcal{Q}_{1,q}^{R}(G_{h}^{n},G_{h}^{n},G_{h}^{n})_{\mathbf{j}} + \mathcal{Q}_{2,q}^{R}(G_{h}^{n},G_{h}^{n},G_{h}^{n})_{\mathbf{j}} \\
&\qquad \qquad \qquad \qquad - \mathcal{Q}_{3,q}^{R}(G_{h}^{n},G_{h}^{n},G_{h}^{n})_{\mathbf{j}} - \mathcal{Q}_{4,q}^{R}(G_{h}^{n},G_{h}^{n},G_{h}^{n})_{\mathbf{j}} \Big] \, .
\end{split}
\end{displaymath}
\end{enumerate}

\item \underline{If we do not assume $\omega \equiv 1$} (velocity rescaling case), we do the following computations:
\begin{enumerate}
\item We estimate $\rho^{n+1}$, $\mathbf{u}^{n+1}$ and $e^{n+1}$ with
\begin{equation*}
\begin{split}
\cfrac{\rho^{n+1}-\rho^{n}}{\Delta t^{n}} 
&= c \, \Delta v^{n} \sum_{\mathbf{j} \, \in \, I_N} \mathcal{Q}_{\alpha}^{R}(f_{h}^{n})_{\mathbf{j}} \\
&= c \Delta \xi \, \left(\cfrac{\pi\omega^{n}}{L}\right)^{-\gamma} \, \mu^{2}\,\sum_{\mathbf{j} \, \in \, I_N}  \mathcal{Q}_{\alpha\mu(\frac{\pi\omega^{n}}{L})^{d_{\mathbf{v}}}}^{R}(G_{h}^{n})_{\mathbf{j}} \, ;
\end{split}
\end{equation*}
\begin{equation*}
\begin{split}
\cfrac{\rho^{n+1}\mathbf{u}^{n+1}-\rho^{n}\mathbf{u}^{n}}{\Delta t^{n}} 
& = c \, \Delta v^{n} \sum_{\mathbf{j} \, \in \, I_N} \mathbf{v}_{\mathbf{j}}^{n}\, \mathcal{Q}_{\alpha}^{R}(f_{h}^{n})_{\mathbf{j}}  \\
& = c \, \Delta \xi \, \left(\cfrac{\pi\omega^{n}}{L}\right)^{-\gamma} \, \mu^{2}\, \sum_{\mathbf{j} \, \in \, I_N} \left(\cfrac{L}{\pi\,\omega^{n}}\, \bm{\xi}_{\mathbf{j}} + \lambda\,\mathbf{u}^{n}\right) \, \mathcal{Q}_{\alpha\mu(\frac{\pi\omega^{n}}{L})^{d_{\mathbf{v}}}}^{R}(G_{h}^{n})_{\mathbf{j}} \, ;
\end{split}
\end{equation*}
\begin{equation*}
\begin{split}
\cfrac{\rho^{n+1}e^{n+1}-\rho^{n}e^{n}}{\Delta t^{n}} & = \cfrac{c}{2} \, \Delta v^{n} \sum_{\mathbf{j} \, \in \, I_N} |\mathbf{v}_{\mathbf{j}}^{n}-\mathbf{u}^{n}|^{2}\, \mathcal{Q}_{\alpha}^{R}(f_{h}^{n})_{\mathbf{j}} \\
& = \cfrac{c}{2} \, \Delta \xi \, \left(\cfrac{\pi\omega^{n}}{L}\right)^{-\gamma} \, \mu^{2}\, \sum_{\mathbf{j} \, \in \, I_N} \left|\cfrac{L}{\pi\,\omega^{n}}\, \bm{\xi}_{\mathbf{j}} + (\lambda-1)\,\mathbf{u}^{n}\right|^{2} \,  \mathcal{Q}_{\alpha\mu(\frac{\pi\omega^{n}}{L})^{d_{\mathbf{v}}}}^{R}(G_{h}^{n})_{\mathbf{j}}  \, .
\end{split}
\end{equation*}

\item We then compute the temperature $T^{n+1}$. Since it is trivial for the classical case $\alpha = 0$ ($T^{n+1} = {2}/d_{\mathbf{v}} \, e^{n+1}$), let us focus on the quantum case $\alpha \neq 0$. The temperature $T^{n+1}$ is obtained as follows: test the inequality
		\begin{displaymath}
|\alpha| \, \rho^{n+1}\, \left(\cfrac{d_{\mathbf{v}}}{4\pi e^{n+1}}\right)^{\frac{d_{\mathbf{v}}}{2}} \in I_{\textnormal{eq}} \, ,
		\end{displaymath}
		where $I_{\textnormal{eq}}$ is defined by \eqref{def_Ieq}. If it is true, $T^{n+1}$ is solved along with the fugacity $z^{n+1}$ by solving the following nonlinear system:
\begin{displaymath}
\left\{
\begin{array}{rcl}
\rho^{n+1} &=& \cfrac{(2\pi T^{n+1})^{\frac{d_{\mathbf{v}}}{2}}}{|\alpha|} \, K_{\frac{d_{\mathbf{v}}}{2}}(z^{n+1}) \, , \\
e^{n+1} &=& \cfrac{d_{\mathbf{v}}\,T^{n+1}}{2}\, \cfrac{K_{\frac{d_{\mathbf{v}}}{2}+1}(z^{n+1})}{K_{\frac{d_{\mathbf{v}}}{2}}(z^{n+1})} \, .
\end{array}
\right.
\end{displaymath}
		If it is false, we distinguish 2 cases:
		\begin{itemize}
			\item For 3D Bose-Einstein case, condensation occurs, so the fugacity $z^{n+1}$ is equal to 1 and $T^{n+1}$ is explicitly computed as follows: 
			\begin{displaymath}
T^{n+1} = \left[\cfrac{2\,|\alpha|\,\rho^{n+1} \, e^{n+1}}{3\,\zeta(5/2)\, (2\pi)^{3/2}}\right]^{2/5} \, ,
			\end{displaymath}
			and the critical mass $m_{0}^{n+1}$ is obtained as
			\begin{displaymath}
m_{0}^{n+1} = \rho^{n+1} - \cfrac{(2\pi T^{n+1})^{3/2}}{|\alpha|}\, \zeta(3/2) \, .
			\end{displaymath}
			\item For Fermi-Dirac case (2D or 3D), we test the equality
			\begin{displaymath}
|\alpha| \, \rho^{n+1}\, \left(\cfrac{d_{\mathbf{v}}}{4\pi e^{n+1}}\right)^{\frac{d_{\mathbf{v}}}{2}} = \left\{
\begin{array}{ll}
2 \, , & \textnormal{for 2D Fermi-Dirac case,} \\
\cfrac{5}{3}\,\sqrt{\cfrac{10}{\pi}}\, , &\textnormal{for 3D Fermi-Dirac case.}
\end{array}
\right.
			\end{displaymath}
			If it is true, then $z^{n+1} = +\infty$ and $T^{n+1} = 0$. If it is false, we cannot say anything about the values of $z^{n+1}$ and $T^{n+1}$.
		\end{itemize}

\item We compute an estimate of $\omega^{n+1}$:
\begin{itemize}
\item By using the temperature (quantum case $\alpha \neq 0$):
\begin{displaymath}
\omega^{n+1} = \sqrt{T^{n+1}} \, ,
\end{displaymath}
\item By using the kinetic energy (classical case $\alpha = 0$):
\begin{displaymath}
\omega^{n+1} = \sqrt{T^{n+1}+\cfrac{|\mathbf{u}^{n+1}|^{2}}{d_{\mathbf{v}}}} \, .
\end{displaymath}
\end{itemize}
\end{enumerate}

\item We compute $\left(\cfrac{\omega^{n+1}-\omega^{n}}{\Delta t^{n}\, \omega^{n}} \, \bm{\xi}_{\mathbf{j}} - \lambda\,\cfrac{\pi\,\omega^{n}}{L} \, \cfrac{\mathbf{u}^{n+1}-\mathbf{u}^{n}}{\Delta t^{n}} \right)\, G_{\mathbf{j}}^{n}$ for any $\mathbf{j} \in I_N$, then its Fourier modes, and finally get $\left(\nabla_{\bm{\xi}} \cdot \left(\cfrac{\omega^{n+1}-\omega^{n}}{\Delta t^{n}\, \omega^{n}} \, \bm{\xi} - \lambda\,\cfrac{\pi\,\omega^{n}}{L} \, \cfrac{\mathbf{u}^{n+1}-\mathbf{u}^{n}}{\Delta t^{n}} \right)\, G_{h}^{n}\right)_{\mathbf{j}}$ for any $\mathbf{j} \in I_N$.

\item We finally compute $G_{h}^{n+1}$ by solving
\begin{displaymath}
\begin{split}
\cfrac{G_{\mathbf{j}}^{n+1}-G_{\mathbf{j}}^{n}}{\Delta t^{n}} + \left(\nabla_{\bm{\xi}} \cdot \left(\cfrac{\omega^{n+1}-\omega^{n}}{\Delta t^{n}\, \omega^{n}} \, \bm{\xi} - \lambda\,\cfrac{\pi\,\omega^{n}}{L} \, \cfrac{\mathbf{u}^{n+1}-\mathbf{u}^{n}}{\Delta t^{n}} \right)\, G_{h}^{n}\right)_{\mathbf{j}} &\\
= c \mu\, \left(\cfrac{\pi\omega^{n}}{L}\right)^{-\gamma} \, \mathcal{Q}_{\alpha\mu(\frac{\pi\omega^{n}}{L})^{d_{\mathbf{v}}}}^{R}(G_{h}^{n})_{\mathbf{j}}& \, ,
\end{split}
\end{displaymath}
and  can reconstruct the approximation $f_{h}^{n+1}$ of $f(t^{n+1},\cdot)$ on $\Omega_{\mathbf{v}}^{n+1} = \mathbf{u}^{n+1} + \left[-\cfrac{L}{\omega^{n+1}},\cfrac{L}{\omega^{n+1}}\right]$ thanks to the relation
\[
f_{\mathbf{j}}^{n+1} = \mu \, \left(\cfrac{\pi\omega^{n+1}}{L}\right)^{d_{\mathbf{v}}}\, G_{\mathbf{j}}^{n+1} \, , \qquad \forall\, \mathbf{j} \in I_N \, .
\]

\item \underline{If we assume that $\omega \equiv 1$ and $\lambda = 0$ (velocity rescaling is not considered)}, we compute $\rho^{n+1}$, $\mathbf{u}^{n+1}$ and $e^{n+1}$ as follows:
\begin{displaymath}
\rho^{n+1} = \Delta v \, \sum_{\mathbf{j}\,\in\,I_N} f_{\mathbf{j}}^{n+1} = \Delta \xi \, \mu \, \sum_{\mathbf{j}\,\in\,I_N} G_{\mathbf{j}}^{n+1} \, ,
\end{displaymath}

\begin{displaymath}
\rho^{n+1}\mathbf{u}^{n+1} = \Delta v \, \sum_{\mathbf{j}\,\in\,I_N} \mathbf{v}_{\mathbf{j}}^{n+1} \, f_{\mathbf{j}}^{n+1} = \Delta \xi\, \mu\, \sum_{\mathbf{j}\,\in\,I_N} \cfrac{L}{\pi}\, \bm{\xi}_{\mathbf{j}} \, G_{\mathbf{j}}^{n+1} \, ,
\end{displaymath}

\begin{displaymath}
\rho^{n+1}e^{n+1} = \cfrac{\Delta v}{2} \, \sum_{\mathbf{j} \, \in \, I_N} |\mathbf{v}_{\mathbf{j}}^{n+1}-\mathbf{u}^{n+1}|^{2} \, f_{\mathbf{j}}^{n+1} = \cfrac{\Delta \xi}{2}\, \mu \sum_{\mathbf{j} \, \in \, I_N} \left| \cfrac{L}{\pi}\, \bm{\xi}_{\mathbf{j}} \right|^{2} \, G_{\mathbf{j}}^{n+1} \, .
\end{displaymath}
\end{enumerate}

It is also possible to consider a RK2 time semi-discretization by using cleverly some parts of the explicit Euler time loop above.

\section{Numerical results}
\label{sec:NumRes}

We shall now present some numerical simulations of our fast spectral method for the BNE. For a given discrete solution $f_h^n$ of the usual variables or $G_h^n$ in rescaled variables, we aim to compute the following quantities:

\begin{itemize}
\item The mass $\rho^{n}$, the averaged velocity $\mathbf{u}^{n}$, the kinetic energy $E_{c}^{n}$, and the internal energy $e^n$:
\[
\begin{aligned}
& \rho^{n} := \Delta v^{n} \, \sum_{\mathbf{j} \, \in \, I_{N}} f_{\mathbf{j}}^{n} = \Delta \xi\, \mu \, \sum_{\mathbf{j} \, \in \, I_{N}} G_{\mathbf{j}}^{n} \, ,
\\
&\mathbf{u}^{n} := \cfrac{1}{\rho^{n}} \, \Delta v^{n} \, \sum_{\mathbf{j} \, \in \, I_{N}} \mathbf{v}_{\mathbf{j}}^{n}\, f_{\mathbf{j}}^{n} = \cfrac{1}{\rho^{n}} \, \Delta\xi\,\mu\, \sum_{\mathbf{j} \, \in \, I_{N}}\left(\cfrac{L}{\pi\,\omega^{n}}\, \bm{\xi}_{\mathbf{j}} + \lambda\,\mathbf{u}^{n}\right) \, G_{\mathbf{j}}^{n} \, ,
\\
&E_{c}^{n} := \Delta v^{n} \, \sum_{\mathbf{j} \, \in \, I_{N}} |\mathbf{v}_{\mathbf{j}}^{n}|^{2}\, f_{\mathbf{j}}^{n} = \Delta\xi\,\mu\, \sum_{\mathbf{j} \, \in \, I_{N}}\left|\cfrac{L}{\pi\,\omega^{n}}\, \bm{\xi}_{\mathbf{j}} + \lambda\,\mathbf{u}^{n}\right|^{2} \, G_{\mathbf{j}}^{n} \, ,
\\
&e^{n} := \cfrac{1}{2\rho^{n}}\, \Delta v^{n} \, \sum_{\mathbf{j} \, \in \, I_{N}} |\mathbf{v}_{\mathbf{j}}^{n}-\mathbf{u}^{n}|^{2}\, f_{\mathbf{j}}^{n} = \cfrac{1}{2\rho^{n}}\,  \Delta\xi\,\mu\, \sum_{\mathbf{j} \, \in \, I_{N}}\left|\cfrac{L}{\pi\,\omega^{n}}\, \bm{\xi}_{\mathbf{j}} + (\lambda-1)\,\mathbf{u}^{n}\right|^{2} \, G_{\mathbf{j}}^{n} \, ;
\end{aligned}
\]

\item The stress-energy tensor $\T^{n} = (T_{d,d'}^{n})_{d,d'} \in \mathcal{M}_{d_{\mathbf{v}}}(\R)$:
\[
\begin{split}
T_{d,d'}^{n} &:= \cfrac{1}{\rho^{n}} \, \Delta v^{n} \, \sum_{\mathbf{j} \, \in \, I_{N}} (v_{j_{d},d}^{n}-u_{d}^{n})\, (v_{j_{d'},d'}^{n}-u_{d'}^{n})\, f_{\mathbf{j}}^{n} \\
&= \cfrac{1}{\rho^{n}} \,  \Delta\xi\,\mu\, \sum_{\mathbf{j} \, \in \, I_{N}}\left(\cfrac{L}{\pi\,\omega^{n}}\, \xi_{j_{d},d} + (\lambda-1)\,u_{d}^{n}\right) \, \left(\cfrac{L}{\pi\,\omega^{n}}\, \xi_{j_{d'},d'} + (\lambda-1)\,u_{d'}^{n}\right) \, G_{\mathbf{j}}^{n} \, ;
\end{split}
\]

\item {The entropy for the classical case $\mathcal{H}_{c}^{n}$:}
\[
\mathcal{H}_{c}^{n} := \Delta v^{n} \, \sum_{\mathbf{j} \, \in \, I_{N}} f_{\mathbf{j}}^{n} \, \ln(f_{\mathbf{j}}^{n}) =\mu\, \Delta \xi \, \sum_{\mathbf{j} \, \in \, I_{N}}G_{\mathbf{j}}^{n} \, \ln\left(\mu\, \left(\cfrac{\pi\,\omega^{n}}{L}\right)^{d_{\mathbf{v}}} \, G_{\mathbf{j}}^{n}\right) \, ;
\]

\item {The entropy for the quantum case $\mathcal{H}_{q}^{n}$:}
\[
\begin{split}
\mathcal{H}_{q}^{n} &:= \Delta v^{n} \, \sum_{\mathbf{j} \, \in \, I_{N}} \left[f_{\mathbf{j}}^{n} \, \ln(f_{\mathbf{j}}^{n})+\cfrac{1-\alpha\,f_{\mathbf{j}}^{n}}{\alpha}\, \ln\left(1-\alpha\,f_{\mathbf{j}}^{n}\right) \right] \\
&= \Delta \xi \, \mu\, \sum_{\mathbf{j} \, \in \, I_{N}} \Bigg[G_{\mathbf{j}}^{n} \, \ln\left(\mu\, \left(\cfrac{\pi\omega^{n}}{L}\right)^{d_{\mathbf{v}}} \, G_{\mathbf{j}}^{n}\right)\\
&\qquad \qquad +\cfrac{\mu^{-1}\, \left(\frac{\pi\omega^{n}}{L}\right)^{-d_{\mathbf{v}}}-\alpha\,G_{\mathbf{j}}^{n}}{\alpha}\, \ln\left(1-\alpha\,\mu\, \left(\cfrac{\pi\omega^{n}}{L}\right)^{d_{\mathbf{v}}} \, G_{\mathbf{j}}^{n}\right) \Bigg] \, ;
\end{split}
\]

\item {$\ell^{p}$ norms in space, for $p \in (1,\infty)$:}
\[
\|f_{h}^{n}\|_{\ell^{2}} = \left(\Delta v^{n} \, \sum_{\mathbf{j} \, \in \, I_{N}} |f_{\mathbf{j}}^{n}|^{p}\right)^{1/p} = \left(\Delta \xi \, \mu^{2}\, \left(\cfrac{\pi\,\omega^{n}}{L}\right)^{d_{\mathbf{v}}} \, \sum_{\mathbf{j} \, \in \, I_{N}} \left|G_{\mathbf{j}}^{n}\right|^{p}\right)^{1/p} \, .
\]
%
\end{itemize}

\subsection{Validation of spectral accuracy}
\label{sec:SpecAccu}

The first test sequence is dedicated to the validation of the spectral accuracy of the numerical methods presented in Section \ref{spectral_homogeneous} in the non-degenerative case. To do this, we proceed as follows:
\begin{enumerate}
\item Set the initial state $f^{0} = f(t=0,\cdot)$ as a candidate for the limit state;
\item Discretize it on a velocity grid to get $f_{h}^{0}$ and compute the discrete moments $\rho_{h}^{0}, \mathbf{u}_{h}^{0}, e_{h}^{0}$;
\item If the velocity rescaling method is used, we set $\mathbf{u}_{h}^{0} = \mathbf{u}^{0}$ where $\mathbf{u}^{0}$ is the exact mean velocity associated to $f^{0}$, and deduce the discrete temperature $T_{h}^{0}$ (and the discrete fugacity $z_{h}^{0}$ for the quantum case) and the discrete velocity scale $\omega_{h}^{0}$ (default value is $\omega_{h}^{0} = 1$);
\item Set $\rho^{\infty} = \rho_{h}^{0}$, $\mathbf{u}^{\infty} = \mathbf{u}^{0}$, $T^{\infty} = T_{h}^{0}$, $z^{\infty} = z_{h}^{0}$, $\omega^{\infty} = \omega_{h}^{0}$ and define the associated limit state $f^{\infty}$;
\item Discretize $f^{\infty}$ to get $f_{h}^{\infty}$.
\end{enumerate}
To highlight the spectral accuracy of our scheme, we compute the value of the numerical residual $\|\mathcal{Q}(f_{h}^{\infty})\|_{\ell^{\infty}}$ as the velocity grid is refined, over the simulation domain
\begin{displaymath}
\Omega_{h} = \lambda \, \mathbf{u}_{h}^{0} + \left[ -\cfrac{L}{\omega_{h}^{0}}, \cfrac{L}{\omega_{h}^{0}}\right] \, ,
\end{displaymath}
where $\lambda = 1$ if the velocity rescaling is activated and $\lambda = 0$ in the contrary case. We also choose the truncation parameters $R$ and $S$ such that
\begin{displaymath}
R = S \, , \qquad \pi = \cfrac{(3+\sqrt{2})\, S}{2} \, ,
\end{displaymath}
so the artificial truncation of $f^{0}$ depends on the value of $L$ provided by the user.

\subsubsection{2D quantum non-degenerative case}
For testing the accuracy of the method in the 2D quantum case, we select the maxwellian function $\mathcal{M}_{q}$ as the initial state. Since such tests were already performed in \cite{FilbetHuJin:2012}, we chose to focus on test cases involving $\hbar = 3$ for highlighting the quantum effects in the collision operator. In order to avoid Fermi-Dirac quantum degeneracy, we define $f^{0}$ as
\begin{equation*}
f^{0}(\mathbf{v}) = \cfrac{1}{|\alpha|} \, \cfrac{1}{z^{-1}\,\exp(\frac{|\mathbf{v}-\mathbf{u}^{0}|^{2}}{2\sigma}) + \textnormal{sgn}(\alpha)} \, , \quad
\alpha = \left\{
\begin{array}{ll}
+\hbar^{d_{\mathbf{v}}} \, , \quad & \textnormal{for Fermi-Dirac particles,} \\
-\hbar^{d_{\mathbf{v}}} \, , \quad & \textnormal{for Bose-Einstein particles,}
\end{array}
\right.
\end{equation*}
where $\sigma$ is set to 1 or 0.5. In addition, since we want to present the advantages of the velocity rescaling, we enrich the test list by considering 2 values of the initial average velocity $\mathbf{u}^{0}$:
\begin{equation*}
\rho^{0} = 1\, , \qquad \mathbf{u}^{0} = \mathbf{0} \, , \qquad \textnormal{or}\qquad \mathbf{u}^{0} = \mathbf{u}_{b} = \cfrac{8}{3\sqrt{2}+2} \, (1,1)^{T} \, ,
\end{equation*}
and several values of the limit $L$ of the simulation bounding box. \\
\indent According to Tables \ref{spectral_accuracy_fermions_2d_sigma05}-\ref{spectral_accuracy_bosons_2d_sigma1} where the numerical residual $\|\mathcal{Q}_{\alpha}^{R}(f_{h}^{\infty})\|_{\ell^{\infty}}$ is computed for various values of $L$, $\sigma$, $\mathbf{u}^{0}$ and grid sizes, it appears that our method is spectrally accurate if we consider a simulation bounding box with $L$ well suited to the moments $\mathbf{u}^{0}$ and $T^{0} = \sigma$. This is highlighted for the 2D Bose-Einstein case with small values of $\sigma$: indeed, such choice for $\sigma$ leads to a distribution $f^{0}$ that is very sharp and concentrated close to $\mathbf{u}^{0}$ (see the profiles $v_{x} \mapsto f^{0}(v_{x},0)$ in Figures \ref{profiles_bosons_2d_qmaxwell}). Consequently, the spectral accuracy becomes very difficult to observe in such configuration even with the velocity rescaling method (see Table \ref{spectral_accuracy_bosons_2d_sigma05}). \\
\indent In addition, we remark that the velocity rescaling method provides much smaller residuals  than the classical method for most cases that have been studied. This is not surprising for the case where $\mathbf{u}^{0} = \mathbf{u}_{b}$ because of the support $B(0,S)$ of the initial distribution that leads to the omission of a part of $f_{h}^{0}$. For the case where $\mathbf{u}^{0} = \mathbf{0}$, this can be explained by the fact that the rescaled distribution $G_{h}^{0}$ is exactly centered on $\mathbf{0}$ without any discretization error which is not the case for the classical method.

\begin{table}[ht]
\begin{tabular}{|c|c|c|c|c|c|}
\hline
\multirow{2}{*}{$L$} & \multirow{2}{*}{Grid} & \multicolumn{2}{|c|}{$\mathbf{u}^{0} = \mathbf{0}$} & \multicolumn{2}{|c|}{$\mathbf{u}^{0} = \mathbf{u}_{b}$} \\
\cline{3-6}
& & with rescaling & without rescaling & with rescaling & without rescaling \\
\hline
\multirow{5}{*}{\begin{tabular}{c}$L = 4.0$ \\ $\sigma = 0.5$\end{tabular}} & $16^{2}$ & 7.54518e-04 & 2.40084e-02 & 7.54292e-04 & 3.02826e-02\\
& $32^{2}$ & 5.45433e-06 & 2.94865e-04 & 6.33224e-06 & 1.61790e-04\\
& $64^{2}$ & 2.16212e-09 & 6.57029e-08 & 2.16212e-09 & 1.05112e-08\\
& $128^{2}$ & 8.72616e-15 & 2.25632e-09 & 7.69842e-15 & 3.79580e-10\\
& $256^{2}$ & 5.81693e-14 & 2.40773e-09 & 2.35945e-14 & 3.58093e-10\\
\hline
\multirow{5}{*}{\begin{tabular}{c}$L = 6.0$ \\ $\sigma = 0.5$\end{tabular}} & $16^{2}$ & 1.23417e-02 & 7.02218e-02 & 7.44099e-03 & 7.96784e-02\\
& $32^{2}$ & 2.52998e-04 & 8.61508e-04 & 3.77004e-04 & 1.77791e-03\\
& $64^{2}$ & 6.98180e-07 & 4.98332e-07 & 6.98180e-07 & 2.52066e-06\\
& $128^{2}$ & 9.92269e-12 & 5.42394e-13 & 1.04077e-11 & 3.52803e-12\\
& $256^{2}$ & 2.75075e-14 & 3.16561e-12 & 2.02706e-14 & 3.85932e-12\\
\hline
\multirow{5}{*}{\begin{tabular}{c}$L = 8.0$ \\ $\sigma = 0.5$\end{tabular}} & $16^{2}$ & 1.41403e-02 & 4.56558e-01 & 1.41403e-02 & 3.60279e-01\\
& $32^{2}$ & 1.28337e-03 & 1.18105e-02 & 1.29406e-03 & 1.31630e-02\\
& $64^{2}$ & 1.39047e-05 & 3.52678e-05 & 1.40178e-05 & 2.21314e-05\\
& $128^{2}$ & 1.76916e-09 & 2.23699e-10 & 2.07287e-09 & 3.34106e-10\\
& $256^{2}$ & 1.40392e-14 & 2.18273e-12 & 1.41762e-14 & 2.84017e-12\\
\hline
\multirow{5}{*}{\begin{tabular}{c}$L = 10.0$ \\ $\sigma = 0.5$\end{tabular}} & $16^{2}$ & 4.84180e-03 & 1.32821e+00 & 2.57422e-02 & 1.37525e+00\\
& $32^{2}$ & 2.46369e-03 & 3.09332e-02 & 2.63280e-03 & 2.59714e-02\\
& $64^{2}$ & 5.28147e-05 & 2.47577e-04 & 5.28147e-05 & 2.60236e-04\\
& $128^{2}$ & 6.16592e-08 & 3.62246e-08 & 1.22383e-07 & 3.34353e-08\\
& $256^{2}$ & 8.33644e-14 & 1.59568e-12 & 8.33644e-14 & 2.60240e-12\\
\hline
\end{tabular}
\caption{\textbf{Spectral accuracy.} Values of the numerical residual $\|\mathcal{Q}_{\alpha}^{R}(f_{h}^{\infty})\|_{\ell^{\infty}}$ for 2D Fermi-Dirac non-degenerative case with $\hbar = 3$, $\rho^{0} = 1$, and $\sigma = 0.5$ ($z = 16.5454$).} \label{spectral_accuracy_fermions_2d_sigma05}
\end{table}

\begin{table}
\begin{tabular}{|c|c|c|c|c|c|}
\hline
\multirow{2}{*}{$L$} & \multirow{2}{*}{Grid} & \multicolumn{2}{|c|}{$\mathbf{u}^{0} = \mathbf{0}$} & \multicolumn{2}{|c|}{$\mathbf{u}^{0} = \mathbf{u}_{b}$} \\
\cline{3-6}
& & with rescaling & without rescaling & with rescaling & without rescaling \\
\hline
\multirow{5}{*}{\begin{tabular}{c}$L = 4.0$ \\ $\sigma = 1.0$\end{tabular}} & $16^{2}$ & 2.95094e-04 & 2.50504e-03 & 2.49797e-04 & 9.54405e-04\\
& $32^{2}$ & 5.07180e-06 & 2.44424e-06 & 4.88457e-06 & 3.92317e-06\\
& $64^{2}$ & 1.84924e-09 & 5.91877e-07 & 1.36283e-09 & 2.04740e-07\\
& $128^{2}$ & 4.20809e-15 & 2.71389e-07 & 3.31905e-15 & 1.54671e-07\\
& $256^{2}$ & 1.18195e-14 & 6.37432e-07 & 1.17001e-14 & 2.68513e-07\\
\hline
\multirow{5}{*}{\begin{tabular}{c}$L = 6.0$ \\ $\sigma = 1.0$\end{tabular}} & $16^{2}$ & 1.31173e-04 & 8.42416e-03 & 1.23660e-04 & 1.45717e-02\\
& $32^{2}$ & 3.76824e-07 & 2.27804e-05 & 3.76824e-07 & 4.70450e-05\\
& $64^{2}$ & 8.21030e-12 & 9.24181e-09 & 8.21063e-12 & 2.59149e-09\\
& $128^{2}$ & 6.68304e-13 & 8.00724e-09 & 5.90658e-13 & 1.93910e-09\\
& $256^{2}$ & 7.16902e-13 & 7.14834e-09 & 7.17939e-13 & 2.18974e-09\\
\hline
\multirow{5}{*}{\begin{tabular}{c}$L = 8.0$ \\ $\sigma = 1.0$\end{tabular}} & $16^{2}$ & 3.15058e-04 & 2.80447e-02 & 2.33615e-04 & 6.62377e-02\\
& $32^{2}$ & 1.12128e-06 & 9.59214e-05 & 1.09318e-06 & 1.69072e-04\\
& $64^{2}$ & 5.98787e-11 & 4.44265e-09 & 5.64793e-11 & 1.26302e-08\\
& $128^{2}$ & 4.73110e-15 & 2.58596e-13 & 3.14199e-15 & 7.56339e-13\\
& $256^{2}$ & 1.85085e-14 & 1.10513e-12 & 1.18911e-14 & 1.60808e-12\\
\hline
\multirow{5}{*}{\begin{tabular}{c}$L = 10.0$ \\ $\sigma = 1.0$\end{tabular}} & $16^{2}$ & 3.13769e-03 & 2.99089e-01 & 2.17247e-03 & 2.12744e-01\\
& $32^{2}$ & 2.02320e-05 & 1.96588e-03 & 2.02320e-05 & 1.14376e-03\\
& $64^{2}$ & 3.89202e-09 & 2.44876e-07 & 2.57176e-09 & 2.99708e-07\\
& $128^{2}$ & 6.05016e-15 & 6.74139e-13 & 4.29173e-15 & 4.04056e-13\\
& $256^{2}$ & 1.26569e-14 & 1.08890e-12 & 1.25693e-14 & 1.41077e-12\\
\hline
\end{tabular}
\caption{\textbf{Spectral accuracy.} Values of the numerical residual $\|\mathcal{Q}_{\alpha}^{R}(f_{h}^{\infty})\|_{\ell^{\infty}}$  for 2D Fermi-Dirac non-degenerative case with $\hbar = 3$, $\rho^{0} = 1$ and $\sigma = 1$ ($z = 3.1887$).} \label{spectral_accuracy_fermions_2d_sigma1}
\end{table}

\begin{table}
\begin{tabular}{|c|c|c|c|c|c|}
\hline
\multirow{2}{*}{$L$} & \multirow{2}{*}{Grid} & \multicolumn{2}{|c|}{$\mathbf{u}^{0} = \mathbf{0}$} & \multicolumn{2}{|c|}{$\mathbf{u}^{0} = \mathbf{u}_{b}$} \\
\cline{3-6}
& & with rescaling & without rescaling & with rescaling & without rescaling \\
\hline
\multirow{5}{*}{\begin{tabular}{c}$L = 4.0$ \\ $\sigma = 0.5$\end{tabular}} & $16^{2}$ & 1.26588e+02 & 1.10944e+03 & 1.59361e+02 & 6.13676e+01\\
& $32^{2}$ & 2.25994e+00 & 4.33543e+01 & 2.24858e+00 & 1.25675e+02\\
& $64^{2}$ & 6.00369e-02 & 1.86373e+00 & 6.00369e-02 & 7.76987e+00\\
& $128^{2}$ & 8.25844e-04 & 7.93633e-03 & 8.25844e-04 & 1.82181e-01\\
& $256^{2}$ & 6.20715e-08 & 4.30685e-08 & 5.94336e-08 & 1.62130e-05\\
\hline
\multirow{5}{*}{\begin{tabular}{c}$L = 6.0$ \\ $\sigma = 0.5$\end{tabular}} & $16^{2}$ & 2.46280e+05 & 2.34093e+04 & 2.54047e+05 & 3.67843e+02\\
& $32^{2}$ & 1.38535e+01 & 2.68651e+02 & 4.19147e+00 & 9.30524e+00\\
& $64^{2}$ & 6.82833e-01 & 7.49992e+00 & 6.81247e-01 & 1.41910e+01\\
& $128^{2}$ & 1.99389e-02 & 1.31040e-01 & 1.07615e-02 & 1.72672e-01\\
& $256^{2}$ & 1.50878e-05 & 3.03279e-05 & 2.88911e-05 & 6.28249e-05\\
\hline
\end{tabular}
\caption{\textbf{Spectral accuracy.} Values of the numerical residual $\|\mathcal{Q}_{\alpha}^{R}(f_{h}^{\infty})\|_{\ell^{\infty}}$ for 2D Bose-Einstein case with $\hbar = 3$, $\rho^{0} = 1$ and $\sigma = 0.5$ ($z = 0.943$).} \label{spectral_accuracy_bosons_2d_sigma05}
\end{table}

\begin{table}
\begin{tabular}{|c|c|c|c|c|c|}
\hline
\multirow{2}{*}{$L$} & \multirow{2}{*}{Grid} & \multicolumn{2}{|c|}{$\mathbf{u}^{0} = \mathbf{0}$} & \multicolumn{2}{|c|}{$\mathbf{u}^{0} = \mathbf{u}_{b}$} \\
\cline{3-6}
& & with rescaling & without rescaling & with rescaling & without rescaling \\
\hline
\multirow{5}{*}{\begin{tabular}{c}$L = 4.0$ \\ $\sigma = 1.0$\end{tabular}} & $16^{2}$ & 8.75850e-02 & 4.77941e+00 & 8.75850e-02 & 2.09808e+00\\
& $32^{2}$ & 4.34564e-03 & 1.34244e-01 & 4.34564e-03 & 6.84172e-02\\
& $64^{2}$ & 6.97576e-06 & 7.27236e-05 & 3.31081e-06 & 1.13026e-04\\
& $128^{2}$ & 8.18675e-10 & 8.90710e-06 & 1.14691e-09 & 2.36536e-05\\
& $256^{2}$ & 9.06688e-10 & 8.27488e-06 & 9.06688e-10 & 2.52684e-05\\
\hline
\multirow{5}{*}{\begin{tabular}{c}$L = 6.0$ \\ $\sigma = 1.0$\end{tabular}} & $16^{2}$ & 8.00293e-02 & 3.30798e+00 & 9.34962e-02 & 5.15603e+00\\
& $32^{2}$ & 3.39670e-03 & 2.43273e-01 & 1.14900e-03 & 3.48709e-01\\
& $64^{2}$ & 8.44540e-06 & 2.84543e-04 & 4.22200e-06 & 5.26459e-04\\
& $128^{2}$ & 2.13457e-11 & 5.45586e-09 & 1.90767e-11 & 2.75652e-08\\
& $256^{2}$ & 1.24522e-11 & 2.90246e-09 & 1.98063e-11 & 2.91932e-08\\
\hline
\multirow{5}{*}{\begin{tabular}{c}$L = 8.0$ \\ $\sigma = 1.0$\end{tabular}} & $16^{2}$ & 8.59697e-02 & 8.09303e+00 & 8.59697e-02 & 7.84054e+00\\
& $32^{2}$ & 3.20116e-03 & 3.22578e-01 & 1.01646e-02 & 7.22433e-01\\
& $64^{2}$ & 1.89958e-05 & 3.67449e-03 & 3.76830e-05 & 2.61653e-03\\
& $128^{2}$ & 1.25941e-09 & 8.52464e-08 & 1.60801e-09 & 2.33918e-07\\
& $256^{2}$ & 8.48209e-14 & 4.54954e-12 & 4.54874e-14 & 7.25384e-12\\
\hline
\end{tabular}

\caption{\textbf{Spectral accuracy.} Values of the numerical residual $\|\mathcal{Q}_{\alpha}^{R}(f_{h}^{\infty})\|_{\ell^{\infty}}$  for 2D Bose-Einstein case with $\hbar = 3$, $\rho^{0} = 1$ and $\sigma = 0.5$ ($z = 0.943$) and $\sigma = 1$ ($z = 0.7613$).} \label{spectral_accuracy_bosons_2d_sigma1}
\end{table}

\begin{figure}
\begin{center}
\begin{tabular}{cc}
 \includegraphics[width=.48\textwidth]{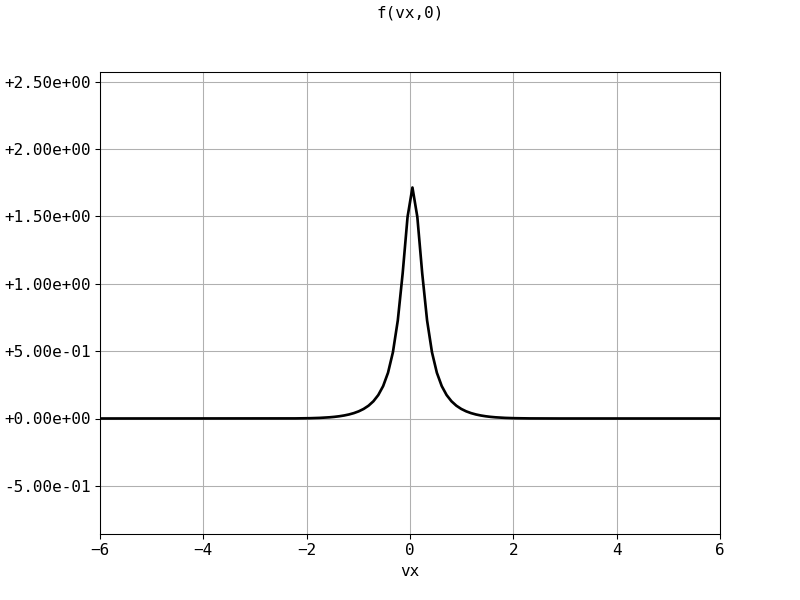}&
\includegraphics[width=.48\textwidth]{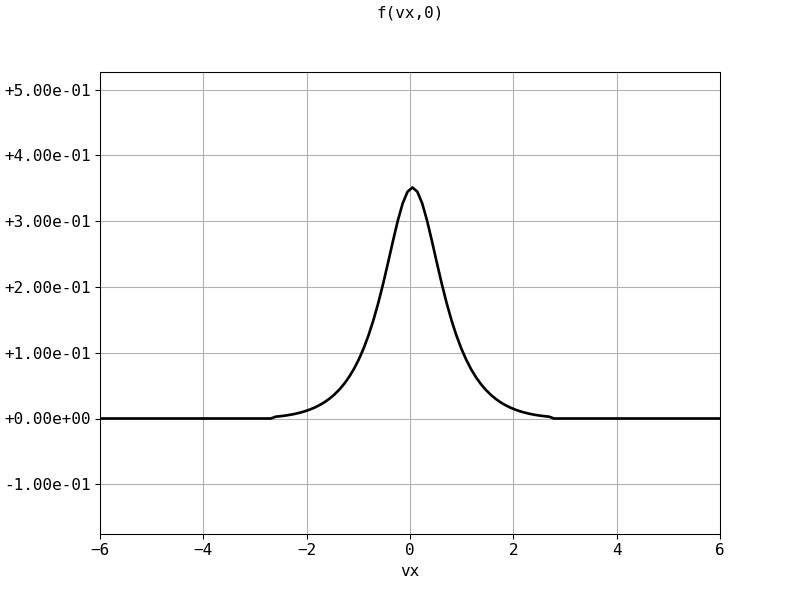}  \\
(a) & (b)
\end{tabular}
\end{center}
\caption{Profiles $v_{x} \mapsto \mathcal{M}_{q}(v_{x},0)$ for 2D Bose-Einstein particles with $\sigma = 0.5$ (a) and $\sigma = 1$ (b). For both profiles, we considered $\hbar = 3$, $\rho^{0} = 1$ and $\mathbf{u}^{0} = \mathbf{0}$. } \label{profiles_bosons_2d_qmaxwell}
\end{figure}

\subsubsection{3D quantum non-degenerative case}
We then perform a similar test sequence for the 3D non-degenerative quantum case. More precisely, for the Fermi-Dirac case, we chose the following parameters:
\begin{displaymath}
\rho^{0} = 1\, , \qquad \mathbf{u}^{0} = \mathbf{0} \quad \textnormal{or}\quad \mathbf{u}^{0} = \mathbf{u}_{b} = \frac{4}{3+\sqrt{2}}\, \left( 1, 1, \sqrt{2} \right)^{T} \, , \qquad \sigma \in \{ 0.5, 1\} \, ,
\end{displaymath}
and for the Bose-Einstein case, we chose
\begin{displaymath}
\mathbf{u}^{0} = \mathbf{0} \quad \textnormal{or}\quad \mathbf{u}^{0} = \mathbf{u}_{b} = \frac{4}{3+\sqrt{2}}\, \left( 1, 1, \sqrt{2} \right)^{T} \, , \qquad \sigma \in \{ 0.5, 1\} \, .
\end{displaymath}
In order to avoid Bose-Einstein condensation (cases with $m_{0} > 0$), we chose $\rho^{0} = 0.5$ in Tables \ref{spectral_accuracy_bosons_3d_sigma05_rho05}-\ref{spectral_accuracy_bosons_3d_sigma1_rho05} and $\rho^{0} = 0.2$ in Tables \ref{spectral_accuracy_bosons_3d_sigma05_rho02}-\ref{spectral_accuracy_bosons_3d_sigma1_rho02}. \\
\indent Since the computational cost is multiplied by 64 each time we multiply by 2 the number of grid points in each direction, we skipped the tests for $128^{3}$ and $256^{3}$ grids (see Tables \ref{spectral_accuracy_fermions_3d_sigma05}-\ref{spectral_accuracy_fermions_3d_sigma1}-\ref{spectral_accuracy_bosons_3d_sigma05_rho05}-\ref{spectral_accuracy_bosons_3d_sigma1_rho05}-\ref{spectral_accuracy_bosons_3d_sigma05_rho02}-\ref{spectral_accuracy_bosons_3d_sigma1_rho02}). \\
\indent Even without these highly refined grids, we can formulate some conclusions that complete those we have made for the 2D non-degenerative case: if we consider a case that is reasonably far from quantum degeneracy (Tables \ref{spectral_accuracy_fermions_3d_sigma05}-\ref{spectral_accuracy_fermions_3d_sigma1}-\ref{spectral_accuracy_bosons_3d_sigma05_rho02}-\ref{spectral_accuracy_bosons_3d_sigma1_rho02}), and up to a good choice for $L$ according to the prescribed values of $\rho^{0}$, $\mathbf{u}^{0}$ and $\sigma$, both classical and velocity rescaling methods are spectrally accurate when the velocity grid is refined. In addition, the velocity rescaling method provides smaller residuals $\|\mathcal{Q}_{\alpha}^{R}(f_{h}^{\infty})\|_{\ell^{\infty}}$ than the classical method for almost all tests. This last remark will be very useful for future 3D tests since it allows to reduce the size of velocity grid and, consequently, the computational cost of the 3D time-space numerical scheme. \\
\indent We can formulate an additional conclusion that is specific to the test case close to Bose-Einstein condensation (see Tables \ref{spectral_accuracy_bosons_3d_sigma05_rho05}-\ref{spectral_accuracy_bosons_3d_sigma1_rho05}): in such case, the discrete distribution $f_{h}$ is almost singular in a neighborhood of $\mathbf{u}^{0}$, so the computation of $\|\mathcal{Q}_{\alpha}^{R}(f_{h}^{\infty})\|_{\ell^{\infty}}$ may produce \texttt{NaN} values. This phenomenon is particularly significant when we set $\mathbf{u}^{0} = \mathbf{0}$ or when we activate the velocity rescaling method. Indeed, in such case, $\bm{\xi} = \mathbf{0}$ is a grid point and $G_{h}(\mathbf{0})$ becomes very large. The only studied way to bypass this problem is to deactivate the velocity rescaling method and to choose $\mathbf{u}^{0}$ outside of the velocity grid points.

\begin{table}
\begin{tabular}{|c|c|c|c|c|c|}
\hline
\multirow{2}{*}{$L$} & \multirow{2}{*}{Grid} & \multicolumn{2}{|c|}{$\mathbf{u}^{0} = \mathbf{0}$} & \multicolumn{2}{|c|}{$\mathbf{u}^{0} = \mathbf{u}_{b}$} \\
\cline{3-6}
& & with rescaling & without rescaling & with rescaling & without rescaling \\
\hline
\multirow{3}{*}{\begin{tabular}{c}$L = 4.0$ \\ $\sigma = 0.5$\end{tabular}} & $16^{3}$ & 6.65068e-04 & 2.64044e-01 & 6.65068e-04 & 4.38601e-02\\
& $32^{3}$ & 1.38282e-05 & 3.51620e-03 & 1.38282e-05 & 3.47688e-04\\
& $64^{3}$ & 4.02472e-09 & 3.11989e-06 & 4.02472e-09 & 8.48441e-08\\
\hline
\multirow{3}{*}{\begin{tabular}{c}$L = 6.0$ \\ $\sigma = 0.5$\end{tabular}} & $16^{3}$ & 1.46161e-02 & 4.93491e-01 & 1.46161e-02 & 6.27339e-01\\
& $32^{3}$ & 4.64216e-04 & 1.27742e-02 & 4.64216e-04 & 1.44628e-02\\
& $64^{3}$ & 5.17232e-07 & 3.23812e-06 & 5.17232e-07 & 1.34607e-05\\
\hline
\multirow{3}{*}{\begin{tabular}{c}$L = 8.0$ \\ $\sigma = 0.5$\end{tabular}} & $16^{3}$ & 6.86816e-02 & 3.05672e+00 & 6.86816e-02 & 3.86836e+00\\
& $32^{3}$ & 3.63615e-03 & 1.49763e-01 & 3.63615e-03 & 1.29305e-01\\
& $64^{3}$ & 2.29417e-05 & 3.12185e-04 & 1.55549e-05 & 2.64884e-04\\
\hline
\multirow{3}{*}{\begin{tabular}{c}$L = 10.0$ \\ $\sigma = 0.5$\end{tabular}} & $16^{3}$ & 6.02018e-02 & 5.16432e+01 & 6.02018e-02 & 3.06977e+01\\
& $32^{3}$ & 1.57951e-02 & 9.41582e-01 & 1.57951e-02 & 6.93264e-01\\
& $64^{3}$ & 2.09198e-04 & 6.20599e-03 & 2.09198e-04 & 5.89099e-03\\
\hline
\end{tabular}

\caption{\textbf{Spectral accuracy.} Values of the numerical residual $\|\mathcal{Q}_{\alpha}^{R}(f_{h}^{\infty})\|_{\ell^{\infty}}$ for 3D Fermi-Dirac non-degenerative case with $\hbar = 3$, $\rho^{0} = 1$ and $\sigma = 0.5$ ($z = 2.43228e+01$).} \label{spectral_accuracy_fermions_3d_sigma05}
\end{table}

\begin{table}
\begin{tabular}{|c|c|c|c|c|c|}
\hline
\multirow{2}{*}{$L$} & \multirow{2}{*}{Grid} & \multicolumn{2}{|c|}{$\mathbf{u}^{0} = \mathbf{0}$} & \multicolumn{2}{|c|}{$\mathbf{u}^{0} = \mathbf{u}_{b}$} \\
\cline{3-6}
& & with rescaling & without rescaling & with rescaling & without rescaling \\
\hline
\multirow{3}{*}{\begin{tabular}{c}$L = 4.0$ \\ $\sigma = 1.0$\end{tabular}} & $16^{3}$ & 9.28576e-04 & 4.02042e-02 & 9.28576e-04 & 4.32287e-03\\
& $32^{3}$ & 5.11585e-05 & 8.89902e-05 & 5.11585e-05 & 9.91060e-06\\
& $64^{3}$ & 2.93880e-08 & 1.19765e-07 & 2.27659e-08 & 1.72883e-07\\
\hline
\multirow{3}{*}{\begin{tabular}{c}$L = 6.0$ \\ $\sigma = 1.0$\end{tabular}} & $16^{3}$ & 3.60402e-04 & 8.95561e-02 & 3.60402e-04 & 1.18694e-01\\
& $32^{3}$ & 9.32472e-07 & 1.70466e-04 & 9.32472e-07 & 2.05708e-04\\
& $64^{3}$ & 4.40082e-11 & 2.41131e-08 & 3.46700e-11 & 7.59905e-09\\
\hline
\multirow{3}{*}{\begin{tabular}{c}$L = 8.0$ \\ $\sigma = 1.0$\end{tabular}} & $16^{3}$ & 4.72058e-04 & 4.90556e-01 & 4.72058e-04 & 2.76318e-01\\
& $32^{3}$ & 1.79828e-06 & 1.73790e-03 & 1.79828e-06 & 2.09096e-03\\
& $64^{3}$ & 5.38083e-11 & 5.71195e-08 & 5.82632e-11 & 1.47296e-07\\
\hline
\multirow{3}{*}{\begin{tabular}{c}$L = 10.0$ \\ $\sigma = 1.0$\end{tabular}} & $16^{3}$ & 6.82783e-03 & 6.30273e+00 & 6.82783e-03 & 4.16489e+00\\
& $32^{3}$ & 2.16863e-05 & 2.08866e-02 & 2.16863e-05 & 9.24323e-03\\
& $64^{3}$ & 2.20701e-09 & 4.38071e-06 & 2.20701e-09 & 4.40874e-06\\
\hline
\end{tabular}
\caption{\textbf{Spectral accuracy.} Values of the numerical residual $\|\mathcal{Q}_{\alpha}^{R}(f_{h}^{\infty})\|_{\ell^{\infty}}$ for 3D Fermi-Dirac non-degenerative case with $\hbar = 3$, $\rho^{0} = 1$ and $\sigma = 1$ ($z = 3.09922$).} \label{spectral_accuracy_fermions_3d_sigma1}
\end{table}

\begin{table}
\begin{tabular}{|c|c|c|c|c|c|}
\hline
\multirow{2}{*}{$L$} & \multirow{2}{*}{Grid} & \multicolumn{2}{|c|}{$\mathbf{u}^{0} = \mathbf{0}$} & \multicolumn{2}{|c|}{$\mathbf{u}^{0} = \mathbf{u}_{b}$} \\
\cline{3-6}
& & with rescaling & without rescaling & with rescaling & without rescaling \\
\hline
\multirow{3}{*}{\begin{tabular}{c}$L = 4.0$ \\ $\sigma = 0.5$\end{tabular}} & $16^{3}$ & - & - & - & 2.01861e+01\\
& $32^{3}$ & - & - & - & 3.07997e+02\\
& $64^{3}$ & - & - & - & 3.30064e+01\\
\hline
\multirow{3}{*}{\begin{tabular}{c}$L = 6.0$ \\ $\sigma = 0.5$\end{tabular}} & $16^{3}$ & - & - & - & 1.33886e+03\\
& $32^{3}$ & - & - & - & 5.39380e+02\\
& $64^{3}$ & - & - & - & 1.83015e+02\\
\hline
\multirow{3}{*}{\begin{tabular}{c}$L = 8.0$ \\ $\sigma = 0.5$\end{tabular}} & $16^{3}$ & - & - & - & 2.09609e+04\\
& $32^{3}$ & - & - & - & 3.26936e+02\\
& $64^{3}$ & - & - & - & 1.87903e+02\\
\hline
\multirow{3}{*}{\begin{tabular}{c}$L = 10.0$ \\ $\sigma = 0.5$\end{tabular}} & $16^{3}$ & - & - & - & 5.58828e+03\\
& $32^{3}$ & - & - & - & 2.73300e+01\\
& $64^{3}$ & - & - & - & 5.21282e+04\\
\hline
\end{tabular}

\caption{\textbf{Spectral accuracy.} Values of the numerical residual $\|\mathcal{Q}_{\alpha}^{R}(f_{h}^{\infty})\|_{\ell^{\infty}}$ for the 3D Bose-Einstein non-degenerative case with $\hbar = 3$, $\rho^{0} = 0.5$ and $\sigma = 0.5$ ($z = 0.99706$).} \label{spectral_accuracy_bosons_3d_sigma05_rho05}
\end{table}

\begin{table}
\begin{tabular}{|c|c|c|c|c|c|}
\hline
\multirow{2}{*}{$L$} & \multirow{2}{*}{Grid} & \multicolumn{2}{|c|}{$\mathbf{u}^{0} = \mathbf{0}$} & \multicolumn{2}{|c|}{$\mathbf{u}^{0} = \mathbf{u}_{b}$} \\
\cline{3-6}
& & with rescaling & without rescaling & with rescaling & without rescaling \\
\hline
\multirow{3}{*}{\begin{tabular}{c}$L = 4.0$ \\ $\sigma = 1.0$\end{tabular}} & $16^{3}$ & 4.85493e-02 & 1.23198e+01 & 4.85493e-02 & 6.20351e-01\\
& $32^{3}$ & 6.60721e-04 & 6.57365e-02 & 6.60721e-04 & 1.40535e-03\\
& $64^{3}$ & 4.36371e-07 & 1.63056e-05 & 4.36373e-07 & 2.17488e-06\\
\hline
\multirow{3}{*}{\begin{tabular}{c}$L = 6.0$ \\ $\sigma = 1.0$\end{tabular}} & $16^{3}$ & 1.52861e-02 & 1.02566e+01 & 1.52861e-02 & 9.62265e+00\\
& $32^{3}$ & 7.38435e-05 & 3.75522e-02 & 7.38434e-05 & 1.13989e-01\\
& $64^{3}$ & 1.24358e-08 & 4.78434e-06 & 1.24358e-08 & 4.57443e-05\\
\hline
\multirow{3}{*}{\begin{tabular}{c}$L = 8.0$ \\ $\sigma = 1.0$\end{tabular}} & $16^{3}$ & 3.53206e-02 & 3.20586e+01 & 3.53206e-02 & 3.08672e+01\\
& $32^{3}$ & 2.16009e-04 & 1.93870e-01 & 2.16009e-04 & 1.46767e-01\\
& $64^{3}$ & 5.59894e-08 & 4.92734e-05 & 5.59894e-08 & 1.52454e-04\\
\hline
\multirow{3}{*}{\begin{tabular}{c}$L = 10.0$ \\ $\sigma = 1.0$\end{tabular}} & $16^{3}$ & 1.59747e-01 & 1.41633e+02 & 1.59740e-01 & 4.01087e+01\\
& $32^{3}$ & 1.34162e-03 & 1.28551e+00 & 1.34162e-03 & 2.69977e-01\\
& $64^{3}$ & 8.77263e-07 & 8.40196e-04 & 8.77263e-07 & 1.13530e-03\\
\hline
\end{tabular}

\caption{\textbf{Spectral accuracy.} Values of the numerical residual $\|\mathcal{Q}_{\alpha}^{R}(f_{h}^{\infty})\|_{\ell^{\infty}}$ for the  3D Bose-Einstein non-degenerative case with $\hbar = 3$, $\rho^{0} = 0.5$ and $\sigma = 1$ ($z = 0.63071$).} \label{spectral_accuracy_bosons_3d_sigma1_rho05}
\end{table}

\begin{table}
\begin{tabular}{|c|c|c|c|c|c|}
\hline
\multirow{2}{*}{$L$} & \multirow{2}{*}{Grid} & \multicolumn{2}{|c|}{$\mathbf{u}^{0} = \mathbf{0}$} & \multicolumn{2}{|c|}{$\mathbf{u}^{0} = \mathbf{u}_{b}$} \\
\cline{3-6}
& & with rescaling & without rescaling & with rescaling & without rescaling \\
\hline
\multirow{3}{*}{\begin{tabular}{c}$L = 4.0$ \\ $\sigma = 0.5$\end{tabular}} & $16^{3}$ & 1.89242e-02 & 5.59497e+00 & 1.89242e-02 & 6.51227e-01\\
& $32^{3}$ & 1.54426e-04 & 3.17405e-02 & 1.54426e-04 & 2.91398e-03\\
& $64^{3}$ & 6.62168e-08 & 7.87246e-06 & 6.62168e-08 & 2.68118e-06\\
\hline
\multirow{3}{*}{\begin{tabular}{c}$L = 6.0$ \\ $\sigma = 0.5$\end{tabular}} & $16^{3}$ & 6.16653e-01 & 2.43214e+01 & 6.16653e-01 & 9.89261e+00\\
& $32^{3}$ & 3.25459e-03 & 2.18171e-01 & 3.25459e-03 & 4.97853e-01\\
& $64^{3}$ & 7.62635e-06 & 1.16549e-04 & 7.62635e-06 & 7.32149e-04\\
\hline
\multirow{3}{*}{\begin{tabular}{c}$L = 8.0$ \\ $\sigma = 0.5$\end{tabular}} & $16^{3}$ & 2.04123e+03 & 2.46424e+02 & 2.04123e+03 & 7.08188e+01\\
& $32^{3}$ & 2.43393e-02 & 2.03212e+00 & 2.43393e-02 & 6.57090e-01\\
& $64^{3}$ & 1.51689e-04 & 3.61768e-03 & 1.51689e-04 & 2.87658e-03\\
\hline
\multirow{3}{*}{\begin{tabular}{c}$L = 10.0$ \\ $\sigma = 0.5$\end{tabular}} & $16^{3}$ & - & 7.36733e+03 & - & 1.07883e+02\\
& $32^{3}$ & 1.38333e-01 & 9.78043e+00 & 1.38333e-01 & 7.43831e-01\\
& $64^{3}$ & 1.03698e-03 & 4.39372e-02 & 1.03698e-03 & 4.01027e-02\\
\hline
\end{tabular}

\caption{\textbf{Spectral accuracy.} Values of the numerical residual $\|\mathcal{Q}_{\alpha}^{R}(f_{h}^{\infty})\|_{\ell^{\infty}}$ for the  3D Bose-Einstein non-degenerative case with $\hbar = 3$, $\rho^{0} = 0.2$ and $\sigma = 0.5$ ($z = 0.68499$).} \label{spectral_accuracy_bosons_3d_sigma05_rho02}
\end{table}

\begin{table}
\begin{tabular}{|c|c|c|c|c|c|}
\hline
\multirow{2}{*}{$L$} & \multirow{2}{*}{Grid} & \multicolumn{2}{|c|}{$\mathbf{u}^{0} = \mathbf{0}$} & \multicolumn{2}{|c|}{$\mathbf{u}^{0} = \mathbf{u}_{b}$} \\
\cline{3-6}
& & with rescaling & without rescaling & with rescaling & without rescaling \\
\hline
\multirow{3}{*}{\begin{tabular}{c}$L = 4.0$ \\ $\sigma = 1.0$\end{tabular}} & $16^{3}$ & 3.74496e-04 & 2.77192e-02 & 3.74496e-04 & 5.36638e-03\\
& $32^{3}$ & 4.12493e-07 & 5.76615e-06 & 4.12493e-07 & 7.85561e-07\\
& $64^{3}$ & 6.86211e-12 & 3.46635e-07 & 6.86211e-12 & 1.63647e-07\\
\hline
\multirow{3}{*}{\begin{tabular}{c}$L = 6.0$ \\ $\sigma = 1.0$\end{tabular}} & $16^{3}$ & 1.49832e-04 & 7.65264e-02 & 1.49832e-04 & 1.11485e-01\\
& $32^{3}$ & 4.51218e-08 & 1.52468e-05 & 4.51218e-08 & 6.22715e-05\\
& $64^{3}$ & 2.65642e-13 & 1.00237e-09 & 2.65642e-13 & 4.21875e-09\\
\hline
\multirow{3}{*}{\begin{tabular}{c}$L = 8.0$ \\ $\sigma = 1.0$\end{tabular}} & $16^{3}$ & 5.11241e-04 & 4.59837e-01 & 5.11241e-04 & 5.35585e-01\\
& $32^{3}$ & 3.22287e-07 & 2.73180e-04 & 3.22287e-07 & 4.08919e-04\\
& $64^{3}$ & 3.36264e-12 & 2.79570e-09 & 3.36264e-12 & 9.71734e-09\\
\hline
\multirow{3}{*}{\begin{tabular}{c}$L = 10.0$ \\ $\sigma = 1.0$\end{tabular}} & $16^{3}$ & 2.78127e-03 & 2.60318e+00 & 2.78127e-03 & 1.18199e+00\\
& $32^{3}$ & 5.03331e-06 & 4.81342e-03 & 5.03331e-06 & 1.81840e-03\\
& $64^{3}$ & 2.03119e-10 & 1.93959e-07 & 2.03099e-10 & 2.62032e-07\\
\hline
\end{tabular}

\caption{\textbf{Spectral accuracy.} Values of the numerical residual $\|\mathcal{Q}_{\alpha}^{R}(f_{h}^{\infty})\|_{\ell^{\infty}}$ for the  3D Bose-Einstein non-degenerative case with $\hbar = 3$, $\rho^{0} = 0.2$ and $\sigma = 1$ ($z = 0.30354$).} \label{spectral_accuracy_bosons_3d_sigma1_rho02}
\end{table}

\subsection{Relaxation towards equilibrium}
\label{sec:NumTrend}

We have seen in Section \ref{sec:stationary_state} that the distribution function $f(t,\cdot)$ is expected to converges to an equilibrium state $f^{\infty}$ as $t \to +\infty$ . We present here some tests that highlight this phenomenon for all treated cases and  provide a numerical validation of the so-called \textit{Fermi-Dirac saturation} 
and \textit{Bose-Einstein condensation}. 
For this, let us associate to a given discrete initial distribution $f^{0}_h$ (and its related macroscopic quantities) a threshold $\hbar^{*}$ defined as
\begin{equation*}
\hbar^{*} = \left[\left(\cfrac{4\pi e_{h}^{0}}{d_{v}}\right)^{\frac{d_{\mathbf{v}}}{2}} \, \cfrac{1}{\rho_{h}^{0}}\right]^{\frac{1}{d_{\mathbf{v}}}}
\end{equation*}
According to Section \ref{sec:stationary_state}, we expect  the following behavior of the discrete distribution $f_{h}$ as $t \to +\infty$:
\begin{itemize}
\item \underline{2D Bose-Einstein particles:} $f_{h}$ converges to the quantum maxwellian defined as
\begin{displaymath}
f^{\infty}(\mathbf{v}) = \mathcal{M}_{q}(\mathbf{v}) = \cfrac{1}{\hbar^{2}}\, \cfrac{1}{\frac{1}{z_{h}^{0}}\,\exp\left(\frac{|\mathbf{v}-\mathbf{u}_{h}^{0}|^{2}}{2T_{h}^{0}}\right) - 1} \, ,
\end{displaymath}
where $(z_{h}^{0},T_{h}^{0})$ solves
\begin{displaymath}
\left\{
\begin{array}{rcl}
\rho_{h}^{0} &=& \cfrac{2\pi T_{h}^{0}}{\hbar^{2}}\, B_{1}(z_{h}^{0}) \, , \\
e_{h}^{0} &=& T_{h}^{0}\, \cfrac{B_{2}(z_{h}^{0})}{B_{1}(z_{h}^{0})} \, ,
\end{array}
\right.
\end{displaymath}

\item \underline{3D Bose-Einstein particles:} $f_{h}$ converges to the limit state $f^{\infty}$ defined as
\begin{displaymath}
f^{\infty}(\mathbf{v}) = \left\{
\begin{array}{ll}
\mathcal{M}_{q}(\mathbf{v}) \, , & \textnormal{if } m_{0,h} = \rho_{h}^{0} - \cfrac{1}{\hbar^{3}}\, \left(\cfrac{4\pi\,e_{h}^{0}}{3}\right)^{3/2} \, \cfrac{\zeta(3/2)^{5/2}}{\zeta(5/2)^{3/2}} \leq 0, \\
\widetilde{M_{q}}(\mathbf{v}) \, , &  \textnormal{else,}

\end{array}
\right.
\end{displaymath}
where, in the former case,
\begin{displaymath}
\mathcal{M}_{q}(\mathbf{v}) = \cfrac{1}{\hbar^{3}}\, \cfrac{1}{\frac{1}{z_{h}^{0}}\,\exp\left(\frac{|\mathbf{v}-\mathbf{u}_{h}^{0}|^{2}}{2T_{h}^{0}}\right) - 1} \, ,
\end{displaymath}
with $(z_{h}^{0},T_{h}^{0})$ solving
\begin{displaymath}
\left\{
\begin{array}{rcl}
\rho_{h}^{0} &=& \cfrac{(2\pi T_{h}^{0})^{3/2}}{\hbar^{3}}\, B_{3/2}(z_{h}^{0}) \, , \\
e_{h}^{0} &=& \cfrac{3T_{h}^{0}}{2}\, \cfrac{B_{5/2}(z_{h}^{0})}{B_{3/2}(z_{h}^{0})} \, ,
\end{array}
\right.
\end{displaymath}
and, in the latter case,
\begin{displaymath}
\widetilde{\mathcal{M}_{q}}(\mathbf{v}) = m_{0,h}\,\delta_{0}(\mathbf{v}-\mathbf{u}_{h}^{0}) + \cfrac{1}{\hbar^{3}} \, \cfrac{1}{\exp\left(\frac{|\mathbf{v}-\mathbf{u}_{h}^{0}|^{2}}{2T_{h}^{0}}\right) - 1} \, , \quad T_{h}^{0} = \cfrac{2\,e_{h}^{0}\,\zeta(3/2)}{3\zeta(5/2)}.
\end{displaymath}

\item \underline{Fermi-Dirac particles:} $f_{h}$ converges to the limit state $f^{\infty}$ defined as
\begin{displaymath}
f^{\infty}(\mathbf{v}) = \left\{
\begin{array}{ll}
\mathcal{M}_{q}(\mathbf{v}) \, , & \textnormal{if } m_{0,h} = \rho_{h}^{0} - \cfrac{K_{d_{\mathbf{v}}}}{\hbar^{d_{\mathbf{v}}}}\, \left(\cfrac{4\pi\,e_{h}^{0}}{d_{\mathbf{v}}}\right)^{\frac{d_{\mathbf{v}}}{2}} < 0, \\
\chi(\mathbf{v}) \, , & \textnormal{if } m_{0,h} = 0, 
\end{array}
\right.
\end{displaymath}
with $K_{2} = 2$, $K_{3} = \cfrac{5}{3}\sqrt{\cfrac{10}{\pi}}$, and where
\begin{displaymath}
\mathcal{M}_{q}(\mathbf{v}) = \cfrac{1}{\hbar^{d_{\mathbf{v}}}}\, \cfrac{1}{\frac{1}{z_{h}^{0}}\,\exp\left(\frac{|\mathbf{v}-\mathbf{u}_{h}^{0}|^{2}}{2T_{h}^{0}}\right) + 1} \, ,
\end{displaymath}
with $(z_{h}^{0},T_{h}^{0})$ solving
\begin{displaymath}
\left\{
\begin{array}{rcl}
\rho_{h}^{0} &=& \cfrac{(2\pi T_{h}^{0})^{\frac{d_{\mathbf{v}}}{2}}}{\hbar^{d_{\mathbf{v}}}}\, F_{\frac{d_{\mathbf{v}}}{2}}(z_{h}^{0}) \, , \\
e_{h}^{0} &=& \cfrac{d_{\mathbf{v}}\,T_{h}^{0}}{2}\, \cfrac{F_{\frac{d_{\mathbf{v}}}{2}+1}(z_{h}^{0})}{F_{\frac{d_{\mathbf{v}}}{2}}(z_{h}^{0})} \, ,
\end{array}
\right.
\end{displaymath}
and
\begin{displaymath}
\chi(\mathbf{v}) = \cfrac{1}{\hbar^{d_{\mathbf{v}}}} \, \mathds{1}_{B(\mathbf{u}_{h}^{0},A_{h}^{0})}(\mathbf{v}) \, ,  \quad A_{h}^{0} = \sqrt{\cfrac{2e_{h}^{0}\,(d_{\mathbf{v}}+2)}{d_{\mathbf{v}}}}.
\end{displaymath}

\end{itemize}

\subsubsection{2D Fermi-Dirac relaxation}

We consider a $128 \times 128$ uniform grid in the velocity domain $[-8,8]^{2}$ and a time step $\Delta t = 0.025$. We do not use  velocity rescaling by taking $(\lambda, \omega, \mu) = (0,1,(\frac{8}{\pi})^{2})$ and use a Runge-Kutta 2 SSP time integrator. We consider 2 different values for $f^{0}$:
\begin{itemize}
\item Fermi-Dirac saturation: we set
\begin{equation} \label{fermions_2d_init_indicator}
f^{0,1}(\mathbf{v}) = \cfrac{\rho^{0}}{\pi (A^{0})^{2}}\, \mathds{1}_{B(\mathbf{u}^{0},A^{0})}(\mathbf{v}) \, ,
\end{equation}
with $A^{0} = 2\sqrt{e^{0}}$. Even if the initial distribution is not continuous, we ensure that the entropy is always defined since the maximal value of $f^{0}$ is $\frac{\rho^{0}}{4\pi e^{0}} \approx (\hbar^{*})^{-2}$. Hence we expect a relaxation towards a quantum maxwellian $\mathcal{M}_{q}$ for any $\hbar < \hbar^{*}$, and no time variation when $\hbar = \hbar^{*}$ (or numerical instabilities leading to a non-physical phenomenon).

\item Classical maxwellian distribution: we set
\begin{equation}  \label{fermions_2d_init_cmax}
f^{0,2}(\mathbf{v}) = \cfrac{\rho^{0}}{2\pi \sigma^{0}} \, \exp\left(-\cfrac{|\mathbf{v}-\mathbf{u}^{0}|^{2}}{2\sigma^{0}}\right) \, ,
\end{equation}
with $\rho^{0} = 1$, $\mathbf{u}^{0} = (0,0)$ and $\sigma^{0} = 1$. In that case, one has $\hbar^{*} \approx 3.5449$. The motivation for choosing such initial datum is that it is a regular distribution whose maximal value is equal to $\frac{\rho^{0}}{2\pi \sigma^{0}} \approx 0.15915$. Let us remind  the reader that an entropic solution must satisfy $\max(f) \leq (\hbar^{*})^{-2}$, and that entropy decay is not ensured in the opposite case. 
Consequently, we must pay attention to the chosen value for $\hbar$ for running the simulation with respect to the value of $(\hbar^{*})^{-2} \approx 0.07958$ prescribed by the moments of $f^{0}$ since taking $\hbar$ too large will break the entropy definition criterion above (even if we do not reach the saturation criterion $\hbar = \hbar^{*}$).
\end{itemize}

Starting with $f^{0,1}$ from \eqref{fermions_2d_init_indicator}, we indeed observe a relaxation of the distribution to a quantum maxwellian for any $\hbar = r\, \hbar^{*}$ with $r \in \{0.1, 0.5, 0.8, 0.9, 0.95, 0.99\}$ (see Figures \ref{fermions_2d_ball_indicator_r01_05}-\ref{fermions_2d_ball_indicator_r099}). We also notice that the shape of the distribution at time $t = 30$ is close to an indicator function of the form $\chi$ as we consider values of $r$ that are close to $1^{-}$. We observe that for the particular case $r = 0.99$, the entropy increases after a finite time. This is not surprising since the limit state as $t \to +\infty$ is almost discontinous, so the Fourier transforms applied to such data induces Gibbs numerical phenomenon. 

\begin{figure}
\begin{tabular}{cc}

\begin{tabular}{c}
\includegraphics[width=0.45\textwidth]{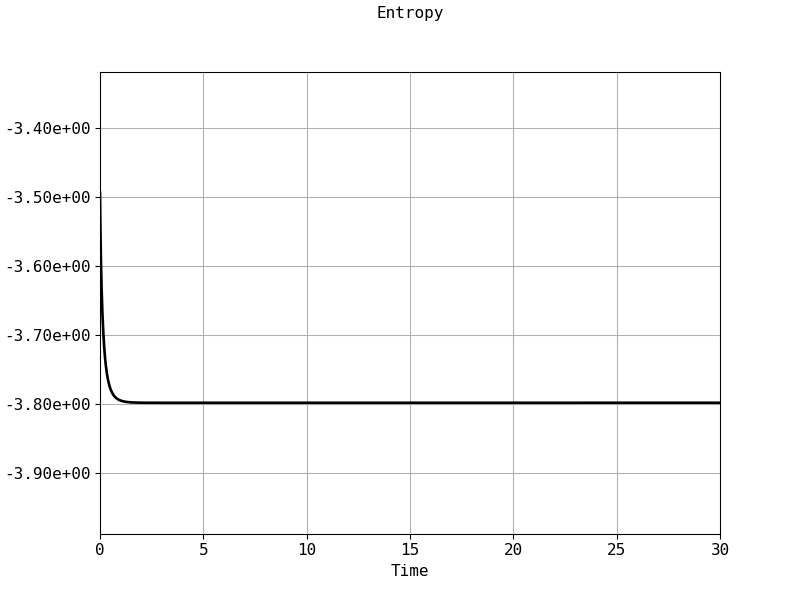}
\end{tabular}
& 
\begin{tabular}{c}\includegraphics[width=0.45\textwidth]{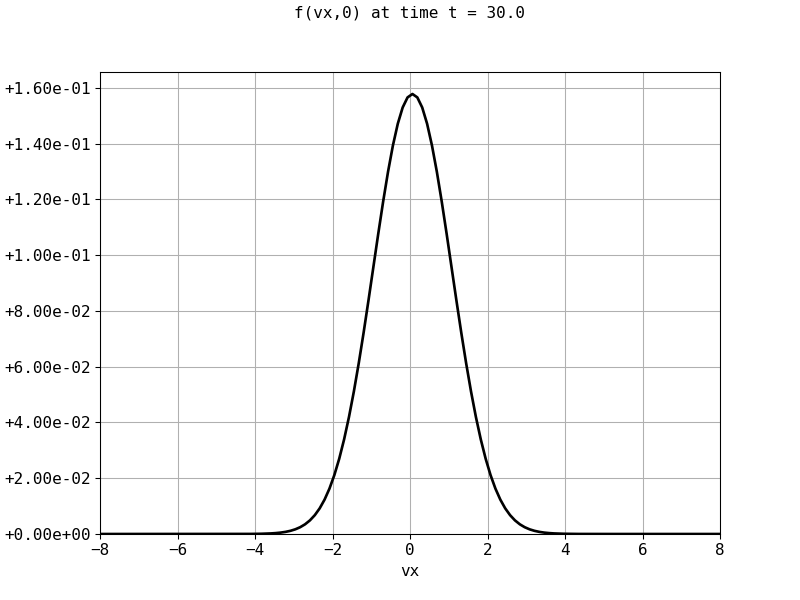}
\end{tabular}
\\

%

\begin{tabular}{c}
\includegraphics[width=0.45\textwidth]{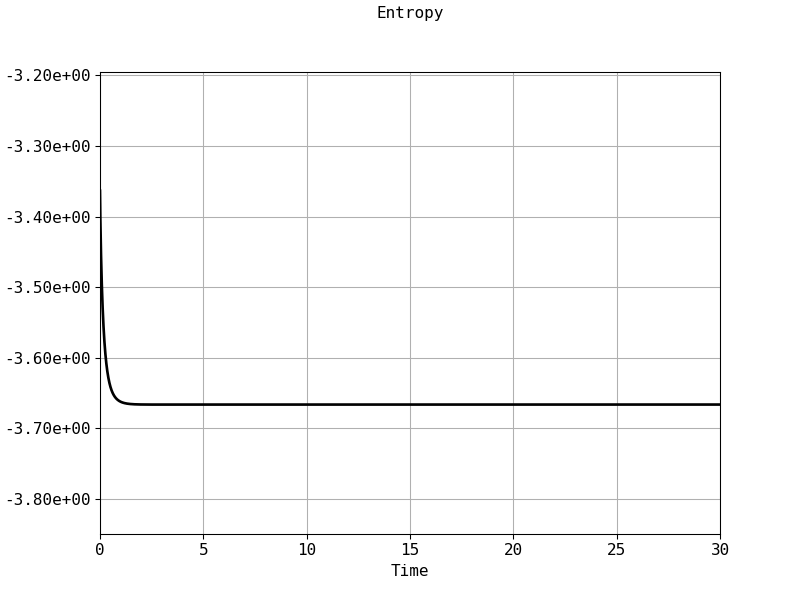}
\end{tabular}
& 
\begin{tabular}{c}\includegraphics[width=0.45\textwidth]{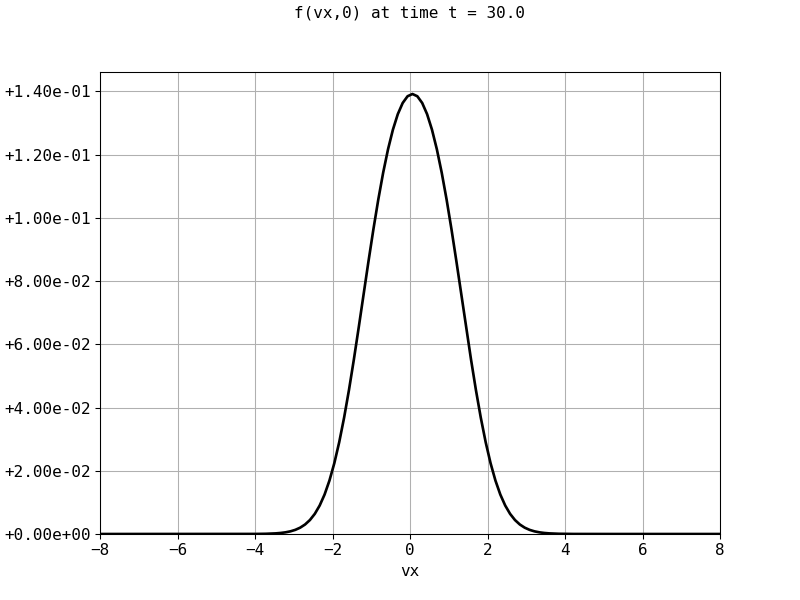}
\end{tabular}
\\

\begin{tabular}{c}
\includegraphics[width=0.45\textwidth]{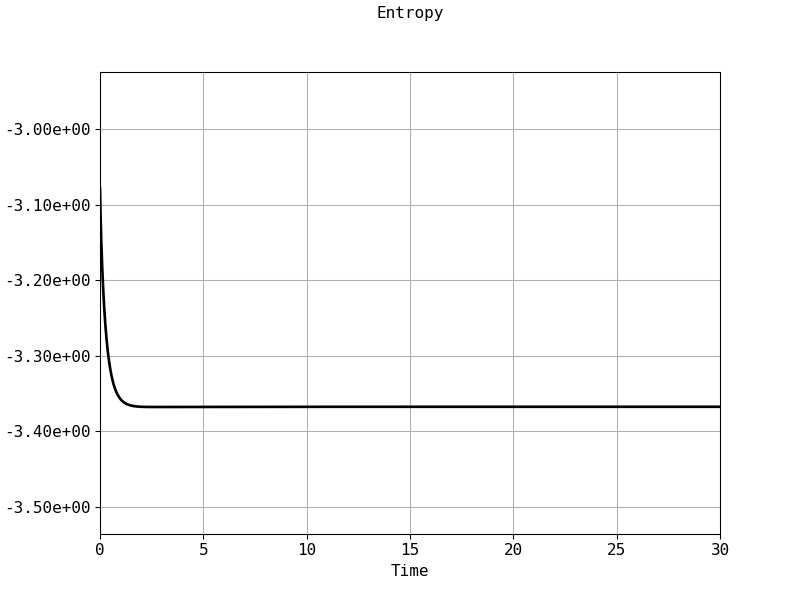}
\end{tabular}
& 
\begin{tabular}{c}\includegraphics[width=0.45\textwidth]{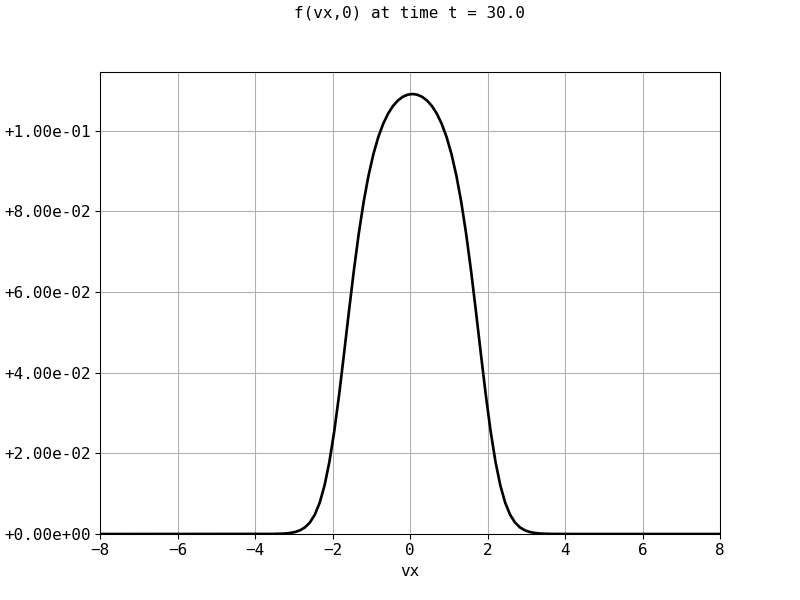}
\end{tabular}

\end{tabular}
\caption{\textbf{Trends to equilibrium.} Time evolution of the entropy (left) and $\{v_y=0\}$--profile of the numerical equilibria (right) for the 2D Fermi-Dirac case with initial datum $f^{0,1}$ and $\hbar = r \hbar^{*}$ with  $r \in \{0.1, 0.5, 0.8\}$ (top to bottom).} \label{fermions_2d_ball_indicator_r01_05}
\end{figure}

\begin{figure}
\begin{tabular}{cc}

\begin{tabular}{c}
\includegraphics[width=0.45\textwidth]{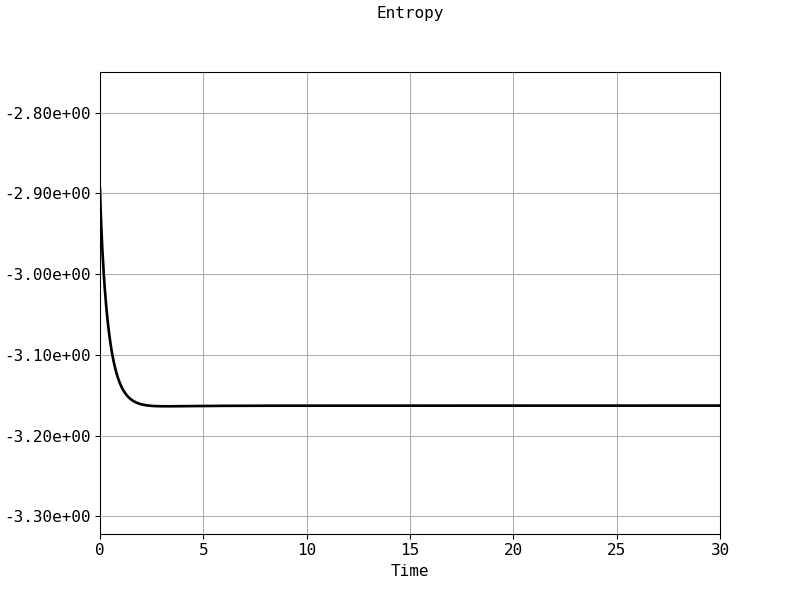}
\end{tabular}
& 
\begin{tabular}{c}\includegraphics[width=0.45\textwidth]{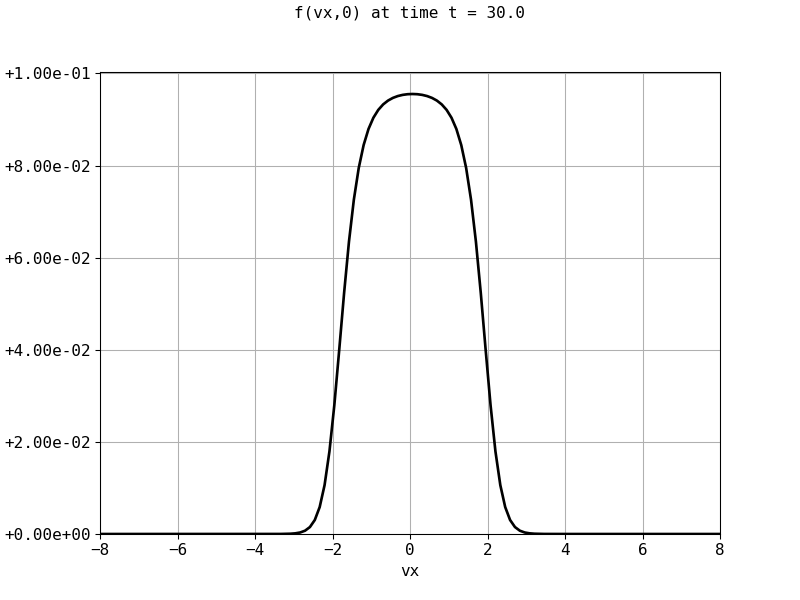}
\end{tabular}
\\

\begin{tabular}{c}
\includegraphics[width=0.45\textwidth]{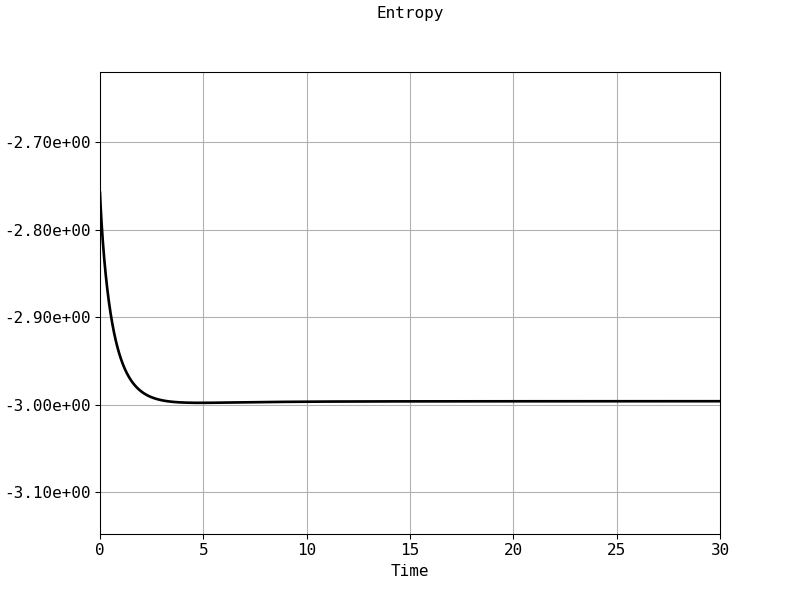}
\end{tabular}
& 
\begin{tabular}{c}\includegraphics[width=0.45\textwidth]{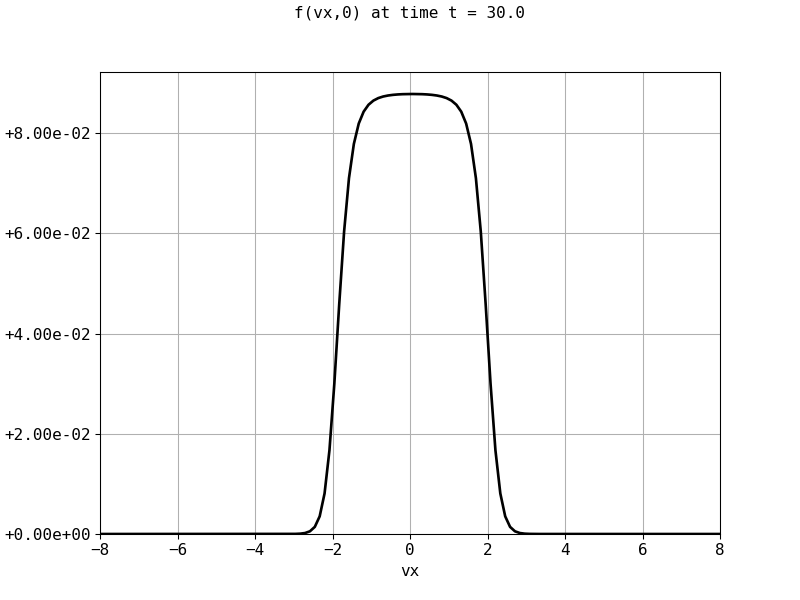}
\end{tabular}

\\

\begin{tabular}{c}
\includegraphics[width=0.45\textwidth]{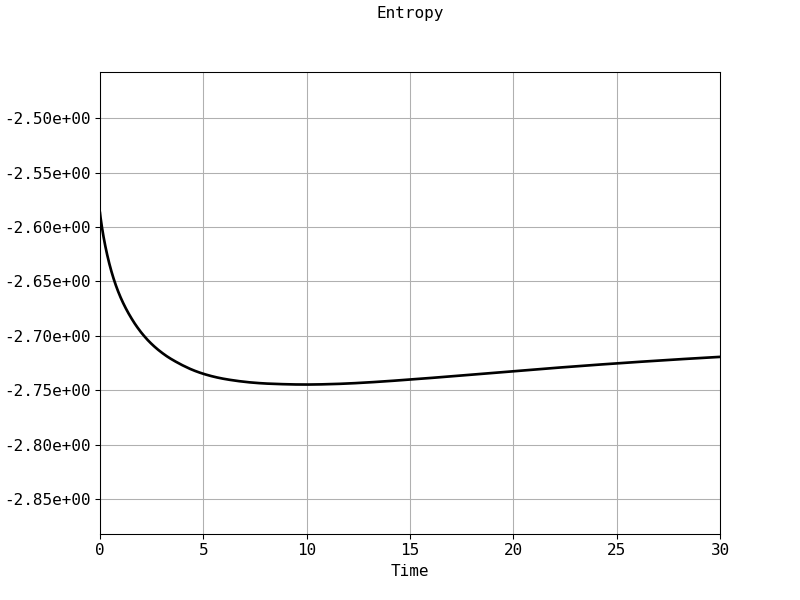}
\end{tabular}
& 
\begin{tabular}{c}\includegraphics[width=0.45\textwidth]{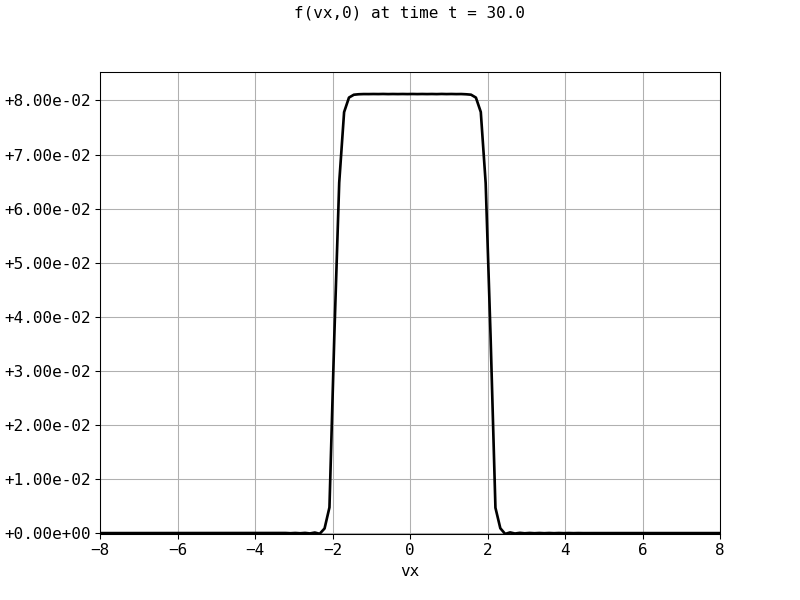}
\end{tabular}
\end{tabular}

\caption{\textbf{Trends to equilibrium.} Time evolution of the entropy (left) and $\{v_y=0\}$--profile of the numerical equilibria (right) for the 2D Fermi-Dirac case with initial datum $f^{0,1}$ and $\hbar = r \hbar^{*}$ with $r = \in \{0.9, 0.95, 0.99\}$ (top to bottom).} \label{fermions_2d_ball_indicator_r099}
\end{figure}

 We then run simulations with $r \geq 1$ for investigating the behavior of the distribution when Fermi-Dirac saturation occurs. For the specific case $r = 1$ where we expect a time independent solution, the errors due to the treatment of a non-continous initial datum by Fourier transform make the approximate solution blow up numerically within a finite time (see Figure \ref{fermions_2d_ball_indicator_r11}). For higher values of $r$, we also observe a numerical blow up of the approximate solution within a finite time which seems to depend on $\hbar$. This last observation is not surprising since we considered the case where $\rho_{h}^{0} > \frac{4\pi e_{h}^{0}}{\hbar^{2}}$ where we cannot identify a limit state as $t \to +\infty$.

\begin{figure}
\begin{tabular}{cc}

\begin{tabular}{c}\includegraphics[width=0.45\textwidth]{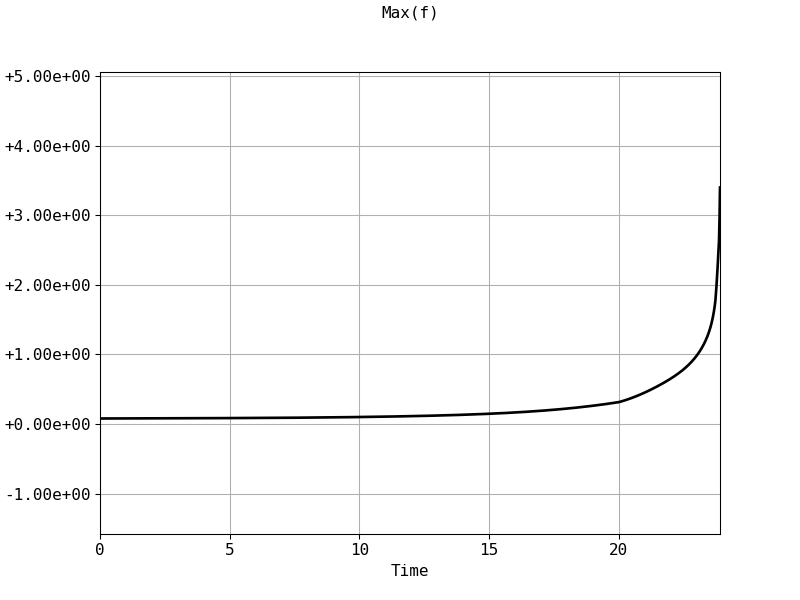}
\end{tabular}
&
\begin{tabular}{c}\includegraphics[width=0.45\textwidth]{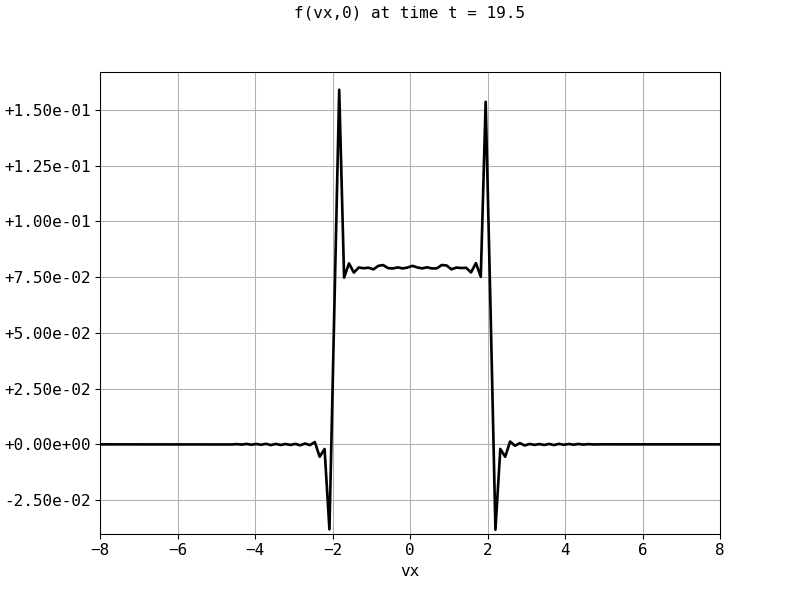}
\end{tabular}

\\

 \begin{tabular}{c}\includegraphics[width=0.45\textwidth]{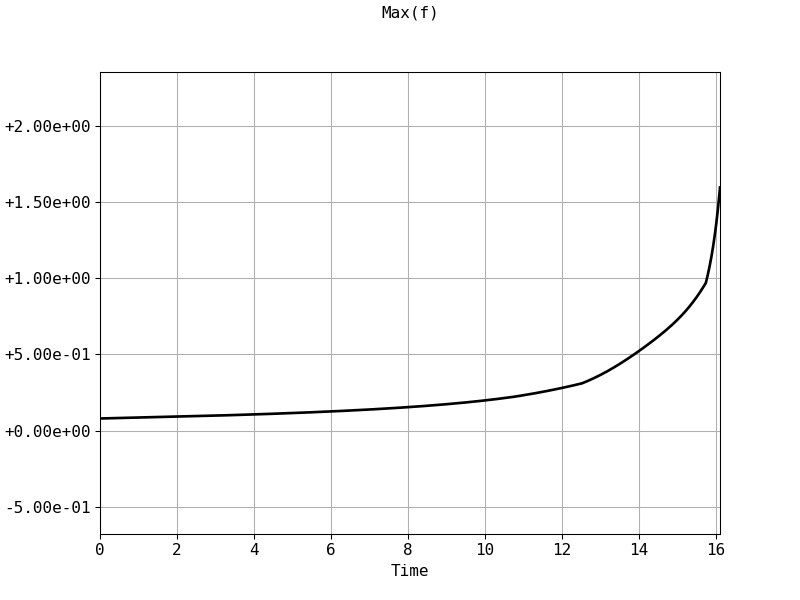}
\end{tabular}
&
\begin{tabular}{c}\includegraphics[width=0.45\textwidth]{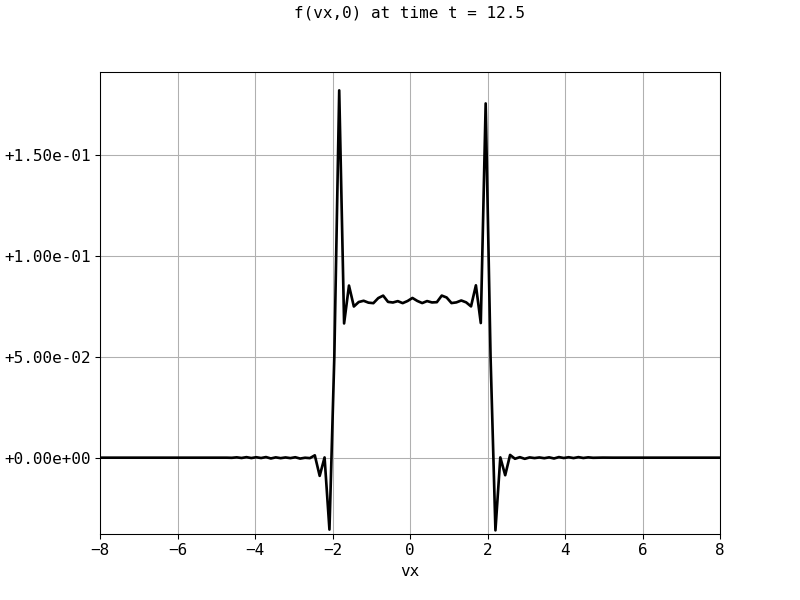}
\end{tabular}

\\

\begin{tabular}{c}\includegraphics[width=0.45\textwidth]{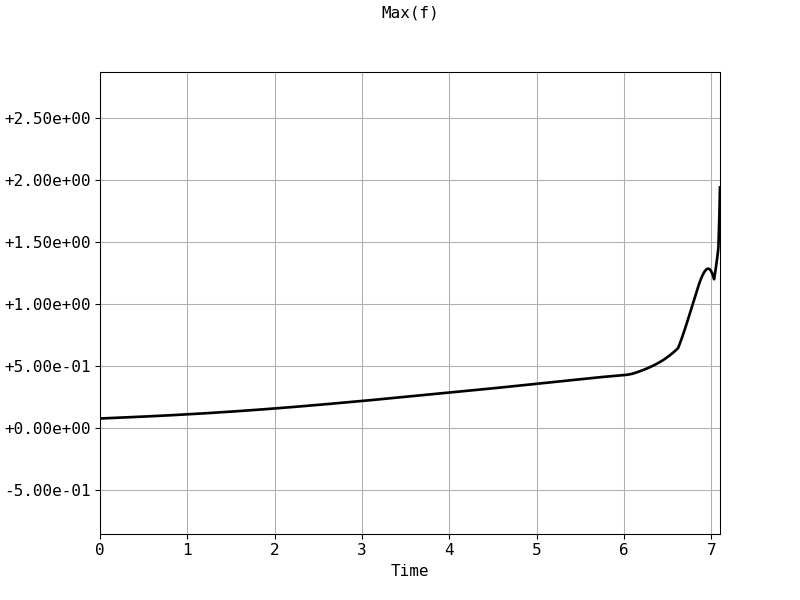}
\end{tabular}
&
\begin{tabular}{c}\includegraphics[width=0.45\textwidth]{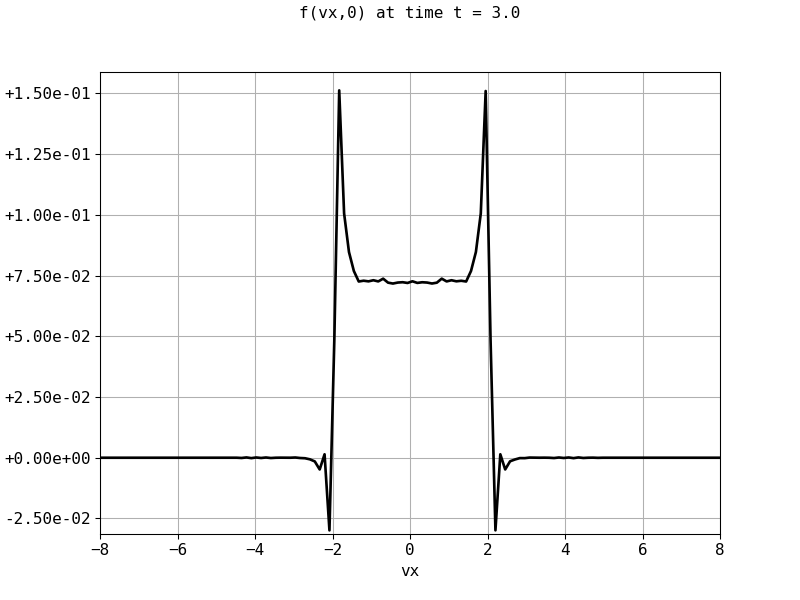}
\end{tabular}

%
\end{tabular}

\caption{\textbf{Trends to equilibrium.} Time evolution of $\max f_h(t,\cdot)$ (left) and $\{v_y=0\}$--profile of the numerical blow up profile (right) for the 2D Fermi-Dirac case with initial datum $f^{0,1}$ and $\hbar = r \hbar^{*}$ with  $r \in \{1, 1.01, 1.05\}$ (top to bottom).} \label{fermions_2d_ball_indicator_r11}
\end{figure}

\medskip
 We then start with the classical maxwellian $f^{0,2}$ from \eqref{fermions_2d_init_cmax}. In that case,  since $\hbar^{*} \approx 3.54491$ and $\max(f^{0,2}) \approx 0.15915$, the entropy definition criteria may be broken at initial time. Taking $\hbar = r\hbar^{*}$ with $r\geq 0$ indicates that it is impossible to define the entropy for $r \geq 0.8$ (see Table \ref{fermions_2d_table_hbar}). 

\begin{table}
\begin{center}
\begin{tabular}{|c|c|}
\hline
$r$ & $\hbar^{-2} = (r\hbar^{*})^{-2}$ \\
\hline
0.1  & 7.95775 \\
0.2  & 1.98944 \\
0.5  & 3.1832e-01 \\
0.8  & 1.2434e-01 \\
0.9  & 9.82438e-02 \\
0.95 & 8.81745e-02 \\
0.99 & 8.11932e-02 \\
1    & 7.95775e-02 \\
\hline
\end{tabular}
\caption{Values of $\hbar^{-2}$ when the initial distribution is a classical maxwellian.} \label{fermions_2d_table_hbar}
\end{center}
\end{table}

By taking $r \in \{0.1, 0.5, 0.8\}$, the entropy is correctly defined at any time since $t \mapsto \max(f(t,\cdot))$ is maximal at time $t = 0$. In such case, we observe the relaxation of the distribution to the associated limit state $f^{\infty} = \mathcal{M}_{q}$ as $t \to +\infty$ (see Figure \ref{fermions_2d_classical_maxwellian_r01_05}). 

\begin{figure}
\begin{tabular}{cc}

\begin{tabular}{c}
\includegraphics[width=0.45\textwidth]{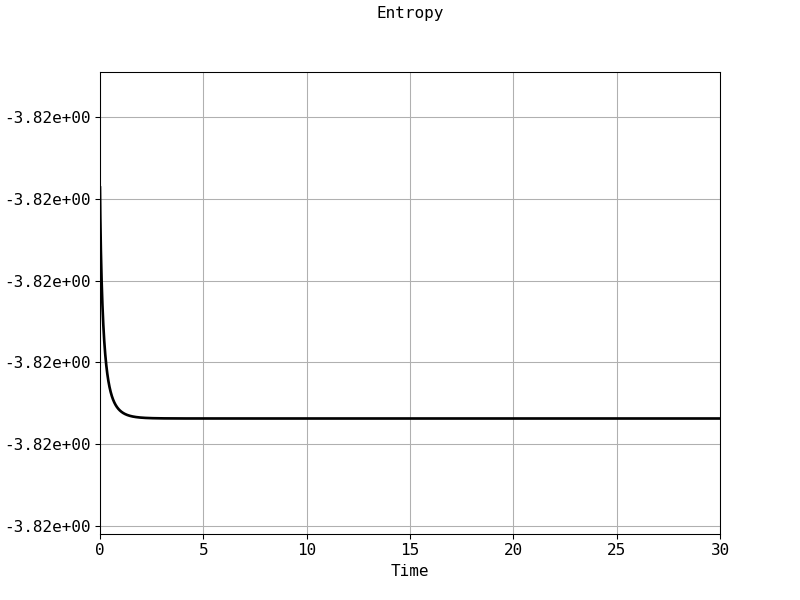}
\end{tabular}
& 
\begin{tabular}{c}\includegraphics[width=0.45\textwidth]{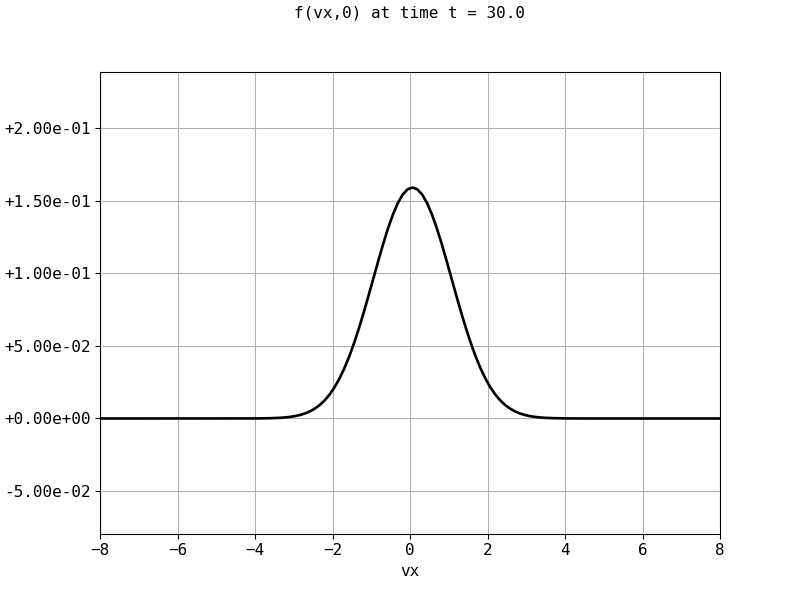}
\end{tabular}
\\


\begin{tabular}{c}
\includegraphics[width=0.45\textwidth]{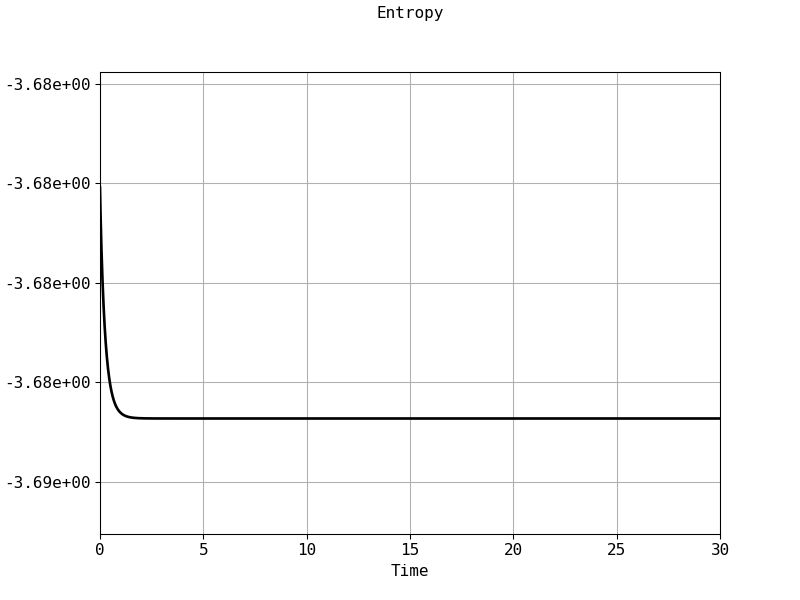}
\end{tabular}
& 
\begin{tabular}{c}\includegraphics[width=0.45\textwidth]{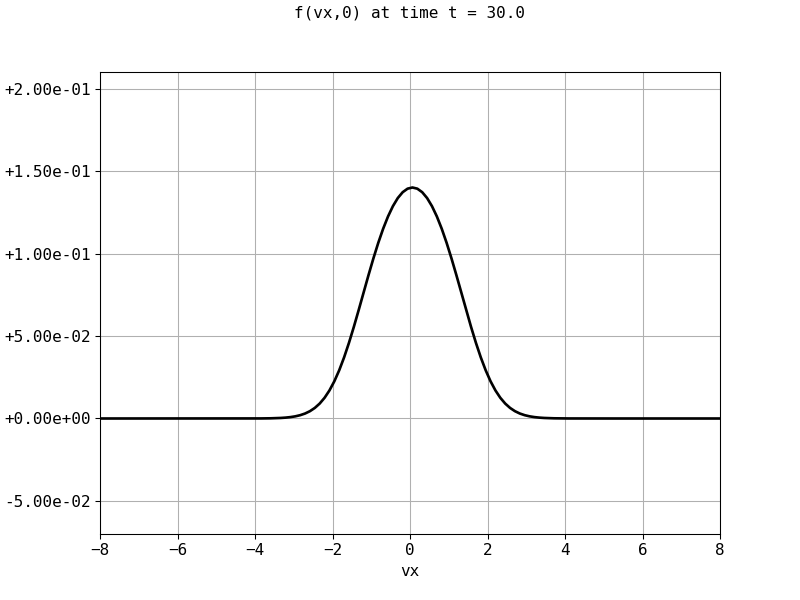}
\end{tabular}

\\
\begin{tabular}{c}
\includegraphics[width=0.45\textwidth]{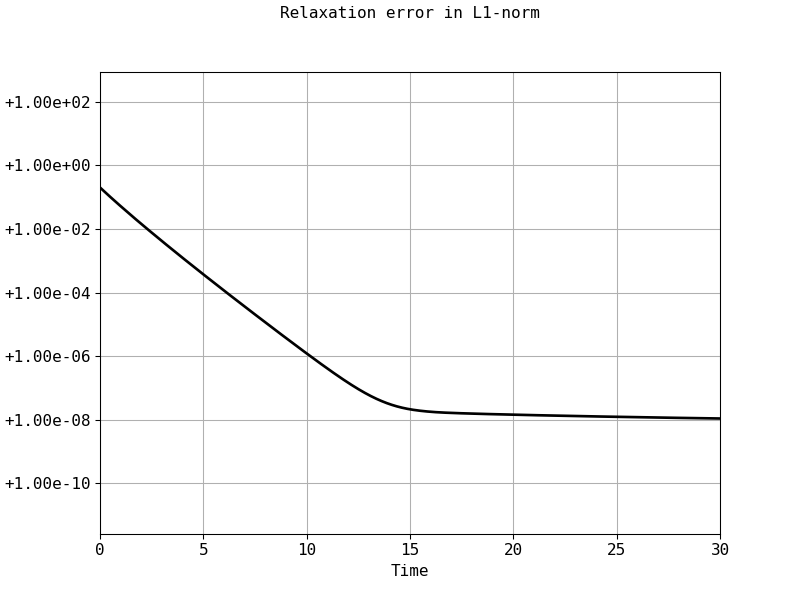}
\end{tabular}
& 
\begin{tabular}{c}\includegraphics[width=0.45\textwidth]{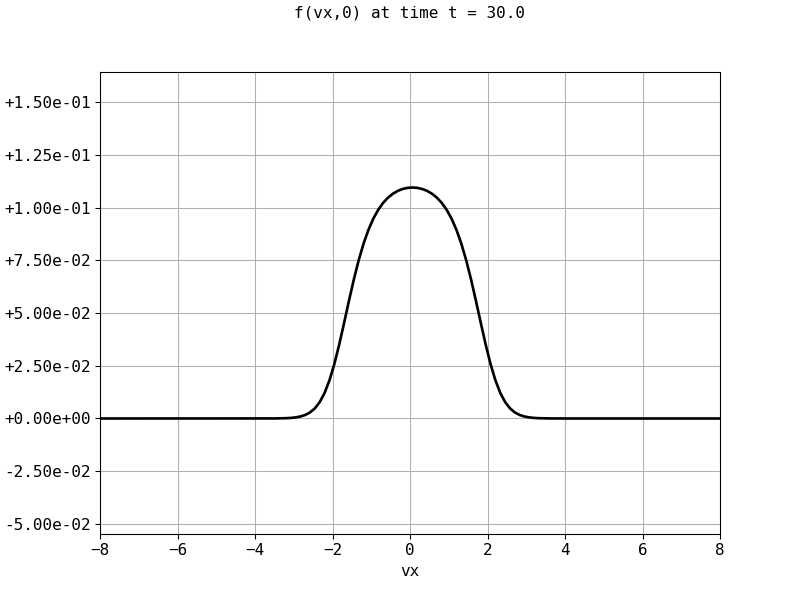}
\end{tabular}

\end{tabular}

\caption{\textbf{Trends to equilibrium.} Time evolution of the entropy (left) and $\{v_y=0\}$--profile of the numerical equilibria (right) for the 2D Fermi-Dirac case with initial datum $f^{0,2}$ and $\hbar = r \hbar^{*}$ with  $r \in \{0.1, 0.5, 0.8\}$ (top to bottom).} \label{fermions_2d_classical_maxwellian_r01_05}
\end{figure}

For the cases $r \in \{0.9, 0.95, 0.99\}$, even if the entropy cannot be correctly defined, we observe that the numerical distribution  $f_{h}$ still converges to the steady state $f_{h}^{\infty} = \mathcal{M}_{q}$ defined with the moments $(\rho_{h}^{0},\mathbf{u}_{h}^{0},e_{h}^{0})$ for any $r < 1$ since the relaxation error $\|f_h(t,\cdot)-f_{h}^{\infty}\|_{\ell^{1}}$ decreases as $t \to +\infty$ (see Figure \ref{fermions_2d_classical_maxwellian_r099}). 

\begin{figure}
\begin{tabular}{cc}

\\

\begin{tabular}{c}
\includegraphics[width=0.45\textwidth]{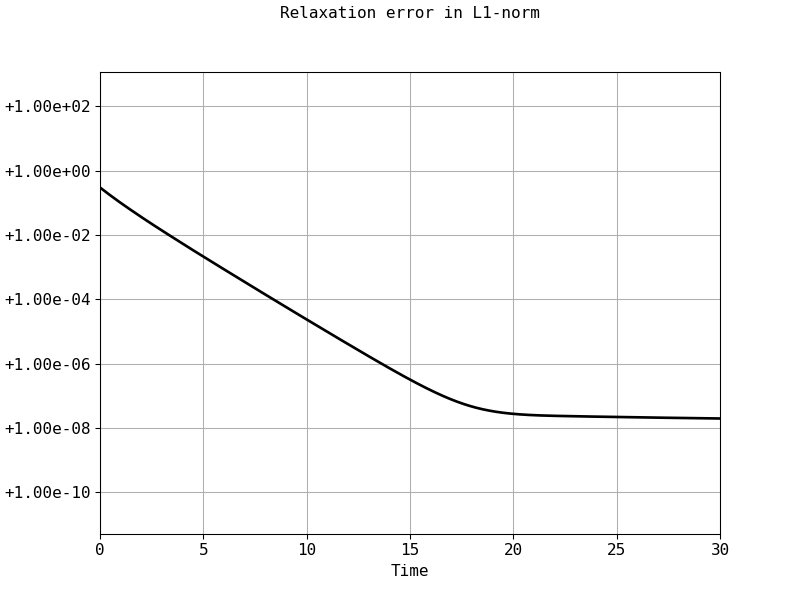}
\end{tabular}
& 
\begin{tabular}{c}\includegraphics[width=0.45\textwidth]{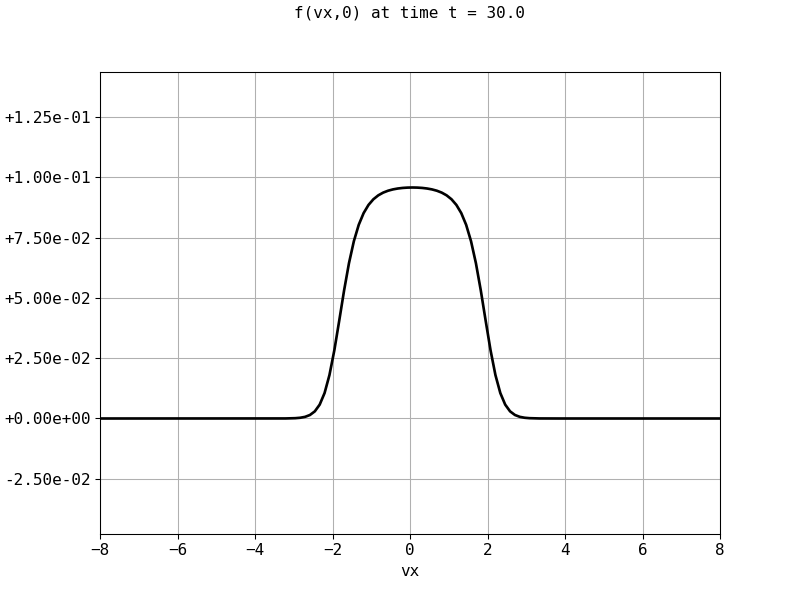}
\end{tabular}
\\

\begin{tabular}{c}
\includegraphics[width=0.45\textwidth]{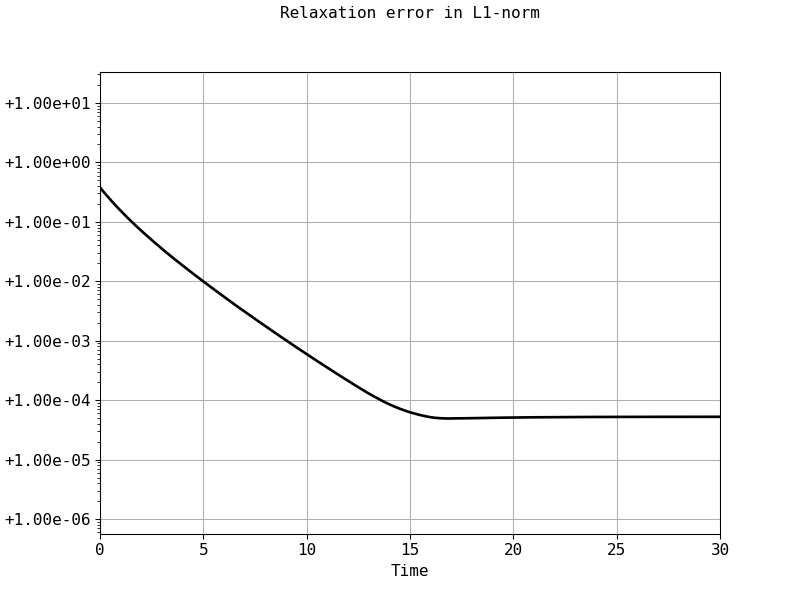}
\end{tabular}
& 
\begin{tabular}{c}\includegraphics[width=0.45\textwidth]{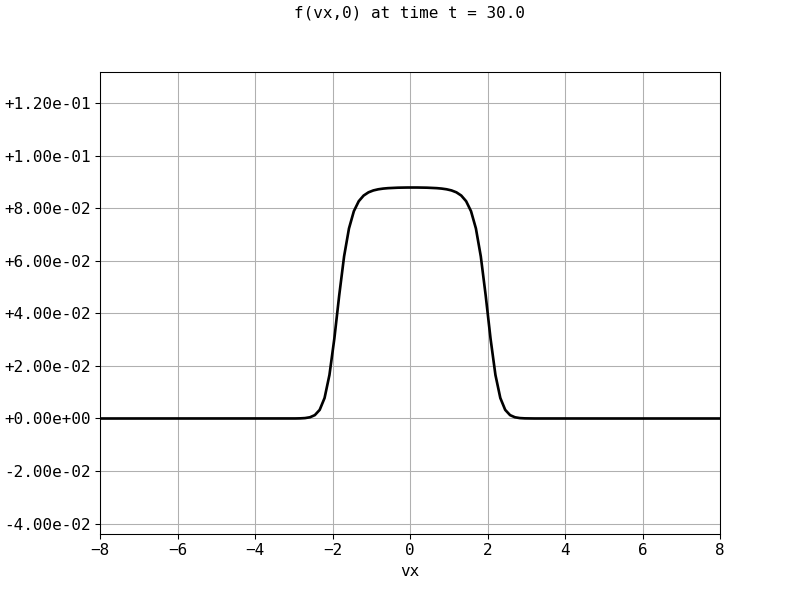}
\end{tabular}

\\

\begin{tabular}{c}
\includegraphics[width=0.45\textwidth]{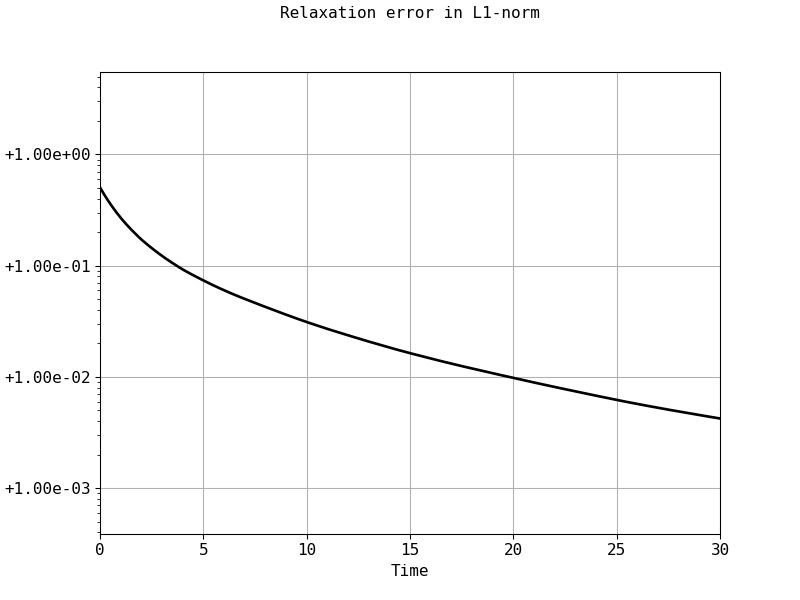}
\end{tabular}
& 
\begin{tabular}{c}\includegraphics[width=0.45\textwidth]{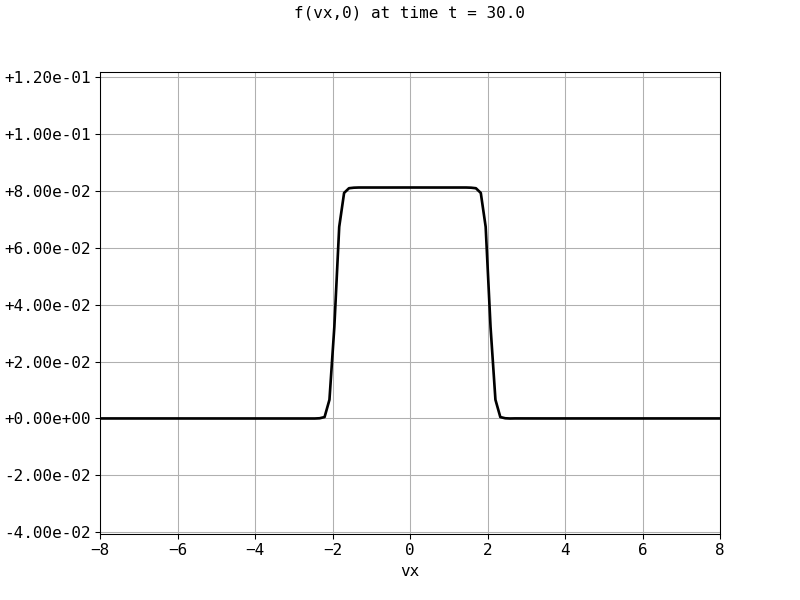}
\end{tabular}
\end{tabular}

\caption{\textbf{Trends to equilibrium.} Time evolution of $\|f_h(t,\cdot)- f_h^\infty\|_{\ell^1}$ (left) and $\{v_y=0\}$--profile of the numerical equilibria (right) for the 2D Fermi-Dirac case with initial datum $f^{0,2}$ and $\hbar = r \hbar^{*}$ with  $r \in \{0.9, 0.95, 0.99\}$ (top to bottom).} \label{fermions_2d_classical_maxwellian_r099}
\end{figure}

Finally, the cases $r \geq 1$ are similar to the indicator case presented in Figure \ref{fermions_2d_ball_indicator_r01_05}: the entropy decay is not observed numerically, because of the oscillations induced by the singular equilibria. The limiting case $r=1$ is presented in Figure \ref{fermions_2d_classical_maxwellian_r1}, as well as the time evolution of $t \mapsto \max(f_h(t,\cdot))$.

\begin{figure}
\begin{tabular}{cc}
\begin{tabular}{c}
\includegraphics[width=0.45\textwidth]{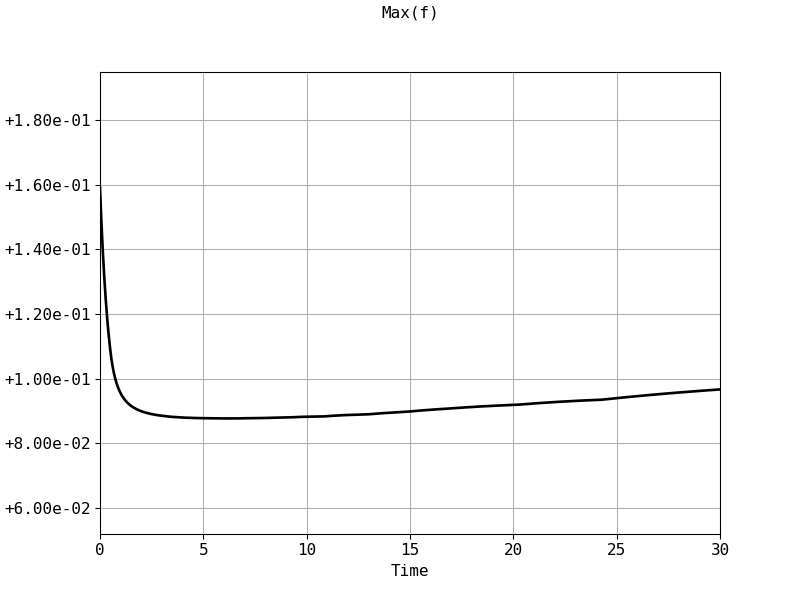}
\end{tabular}
& 
\begin{tabular}{c}\includegraphics[width=0.45\textwidth]{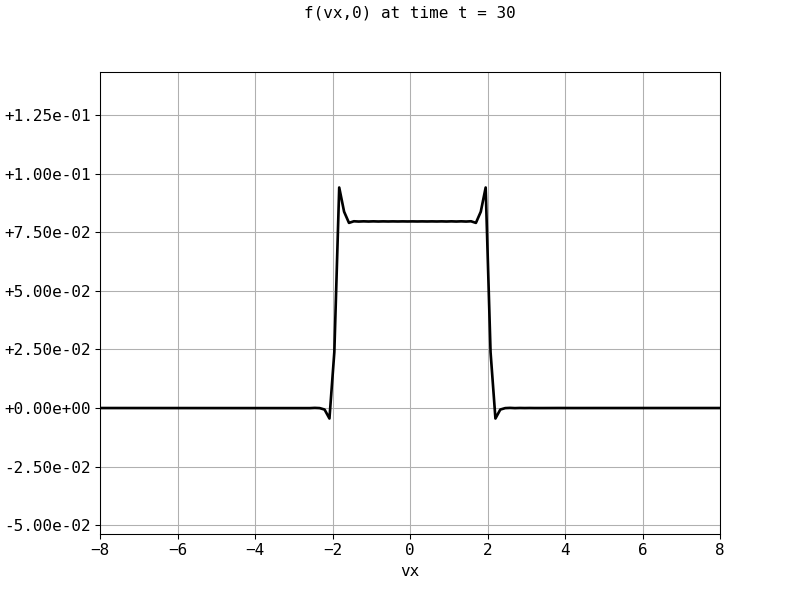}
\end{tabular}
\end{tabular}

\caption{\textbf{Trends to equilibrium.} Time evolution of $\max f_h(t,\cdot)$ (left) and $\{v_y=0\}$--profile of the numerical equilibrium (right) for the 2D Fermi-Dirac case with initial datum $f^{0,2}$ and $\hbar = \hbar^{*}$.} \label{fermions_2d_classical_maxwellian_r1}
\end{figure}

\subsubsection{3D Fermi-Dirac relaxation}

Let us now extend the analysis of the previous paragraph to the case of 3D Fermi-Dirac particles. Due to the large time cost required for computing the term $\mathcal{Q}_{1,q}^{R}(G_{h})$, we restrict our simulations to a $32^{3}$ grid in velocity.  Again, we consider two definitions of the initial distribution $f^{0}$:
\begin{itemize}
\item Fermi-Dirac saturation: we set
\begin{equation} \label{fermions_3d_init_cmax}
f^{0,3}(\mathbf{v}) = \cfrac{3\rho^{0}}{4\pi\,(A^{0})^{3}} \, \mathds{1}_{B(\mathbf{u}^{0},A^{0})}(\mathbf{v}) \, ,
\end{equation}
with $A^{0} = \sqrt{\frac{10 e^{0}}{3}}$. As for the 2D case, we expect a relaxation towards a quantum maxwellian $\mathcal{M}_{q}$ for any $\hbar < \hbar^{*}$ with $\left (\hbar^{*}\right )^{-3} \approx \frac{3\rho^{0}}{4\pi\,(\frac{10 e^{0}}{3})^{\frac{3}{2}}}$ and no time variation when $\hbar = \hbar^{*}$.

\item Classical maxwellian distribution: we set
\begin{equation} \label{fermions_3d_init_ball}
f^{0,4}(\mathbf{v}) = \cfrac{\rho^{0}}{(2\pi\sigma^{0})^{\frac{3}{2}}}\, \exp\left(-\cfrac{|\mathbf{v}-\mathbf{u}^{0}|^{2}}{2\sigma^{0}}\right) \, ,
\end{equation}
with $\rho^{0} = 1$, $\mathbf{u}^{0} = (0,0,0)$ and $\sigma^{0} = 1$, leading to the threshold value $\hbar^{*} \approx 3.60452$ and $\max(f^{0}) = \frac{\rho^{0}}{(2\pi\,\sigma^{0})^{\frac{3}{2}}} \approx 6.349364\times 10^{-2}$. Since $\left (\hbar^{*}\right )^{-3} \approx 2.13529\times 10^{-2}$, we may break the definition of the entropy if we consider too large values of $\hbar$.
\end{itemize}

\smallskip

Considering the discontinuous initial distribution \eqref{fermions_3d_init_cmax}, we observe the relaxation of the distribution to a quantum maxwellian for any $\hbar = r\,\hbar^{*}$ with $r \in \{0.1, 0.5, 0.8, 0.9, 0.95\}$. Indeed, we observe a decay in time for the entropy and the relaxation error $\|f_{h}(t,\cdot)-f_{h}^{\infty}\|_{\ell^{1}}$ for these values of $\hbar$ (see Figures \ref{fermions_3d_ball_indicator_r01-05}-\ref{fermions_3d_ball_indicator_r08-095}).

\begin{figure}
\begin{tabular}{cc}

\begin{tabular}{c}
\includegraphics[width=0.45\textwidth]{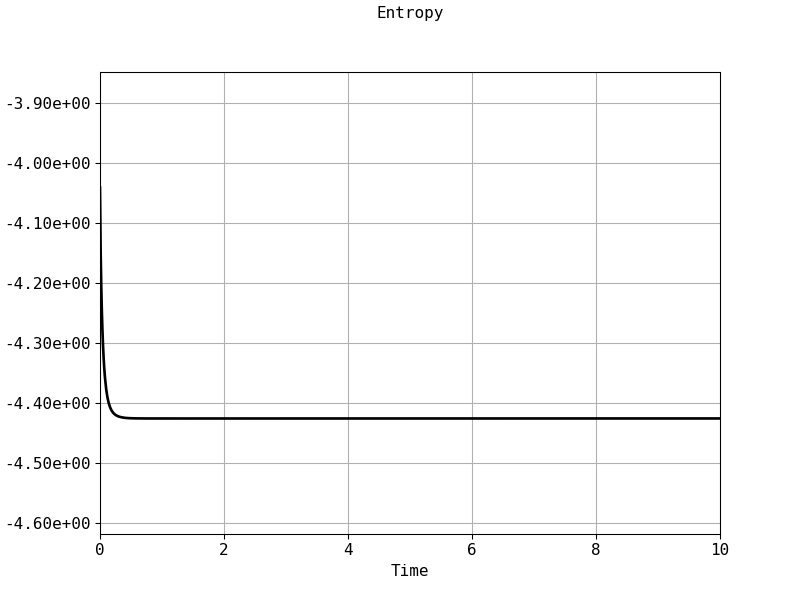}
\end{tabular}
& 
\begin{tabular}{c}\includegraphics[width=0.45\textwidth]{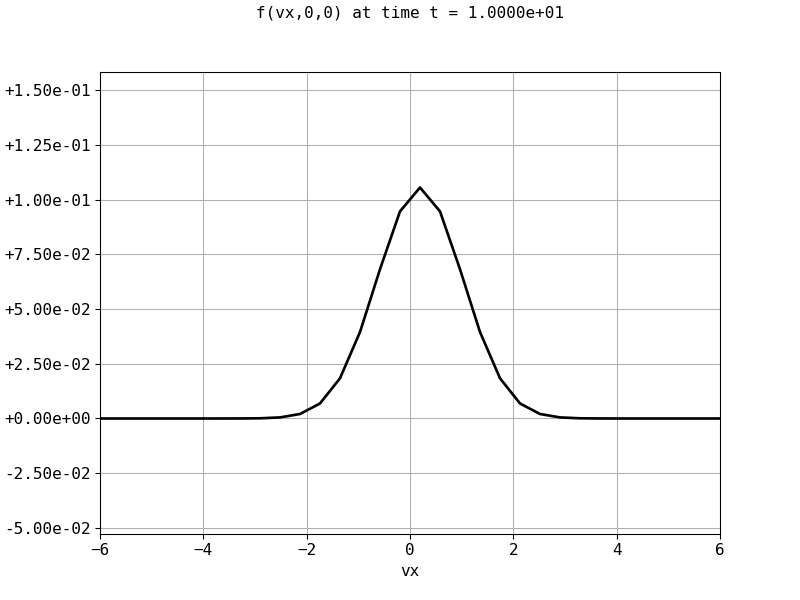}
\end{tabular}
%
%

\\

\begin{tabular}{c}
\includegraphics[width=0.45\textwidth]{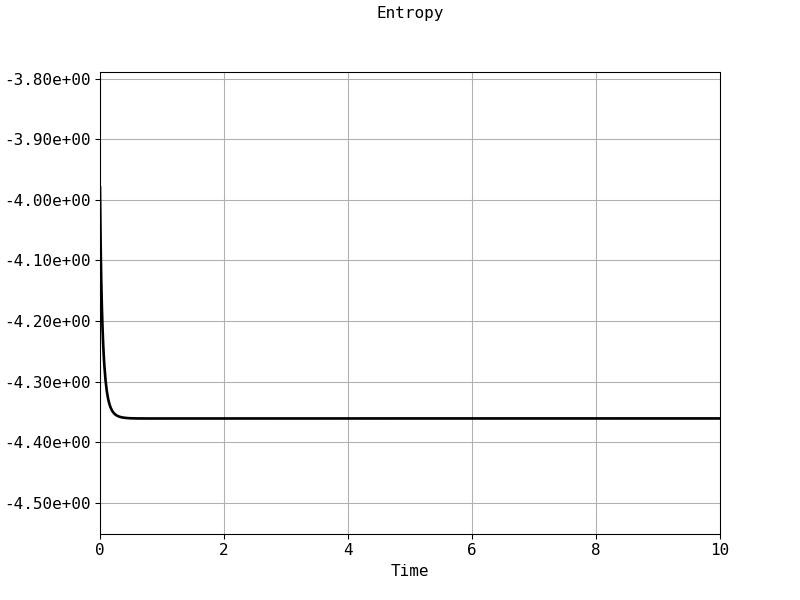}
\end{tabular}
& 
\begin{tabular}{c}\includegraphics[width=0.45\textwidth]{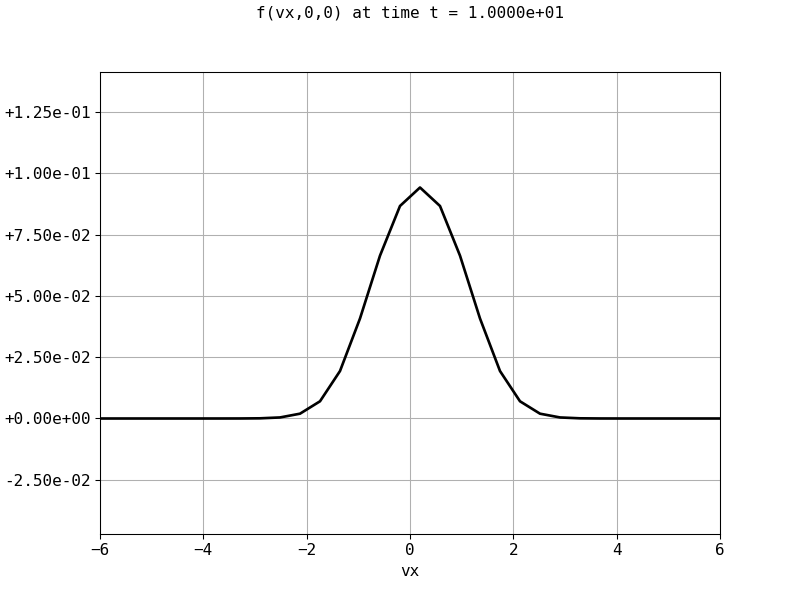}
\end{tabular}

\\

\begin{tabular}{c}
\includegraphics[width=0.45\textwidth]{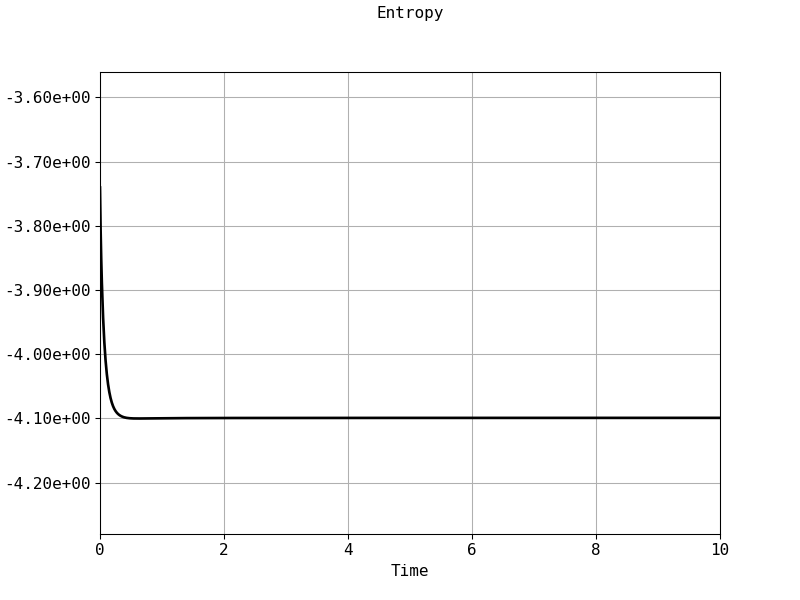}
\end{tabular}
& 
\begin{tabular}{c}\includegraphics[width=0.45\textwidth]{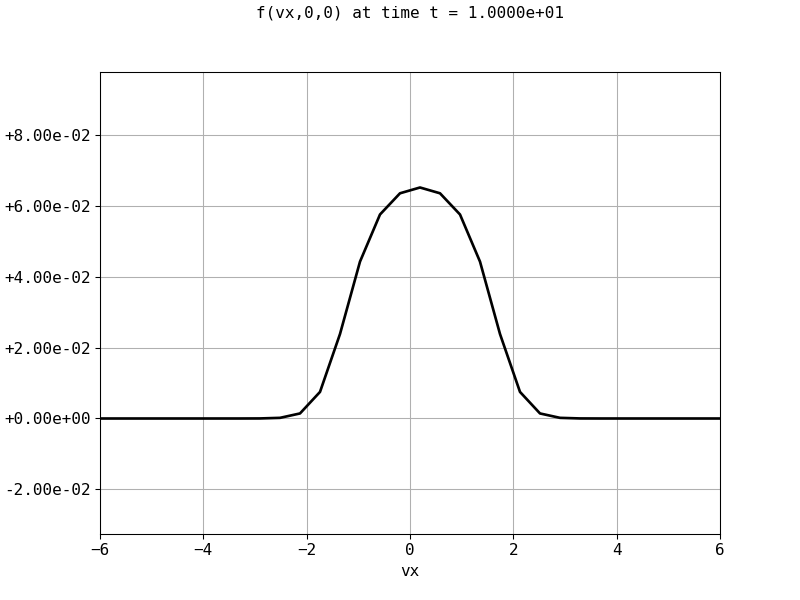}
\end{tabular}

\end{tabular}
\caption{\textbf{Trends to equilibrium.}  Time evolution of the entropy (left) and  profile of $f_{h}(10, (\cdot,0,0))$ (right) for the 3D Fermi-Dirac case with initial datum $f^{0,3}$ and $\hbar = r \hbar^{*}$ with  $r \in \{0.1,  0.5, 0.8\}$ (top to bottom). } \label{fermions_3d_ball_indicator_r01-05}

\end{figure}

\begin{figure}

\begin{tabular}{cc}

\begin{tabular}{c}
\includegraphics[width=0.45\textwidth]{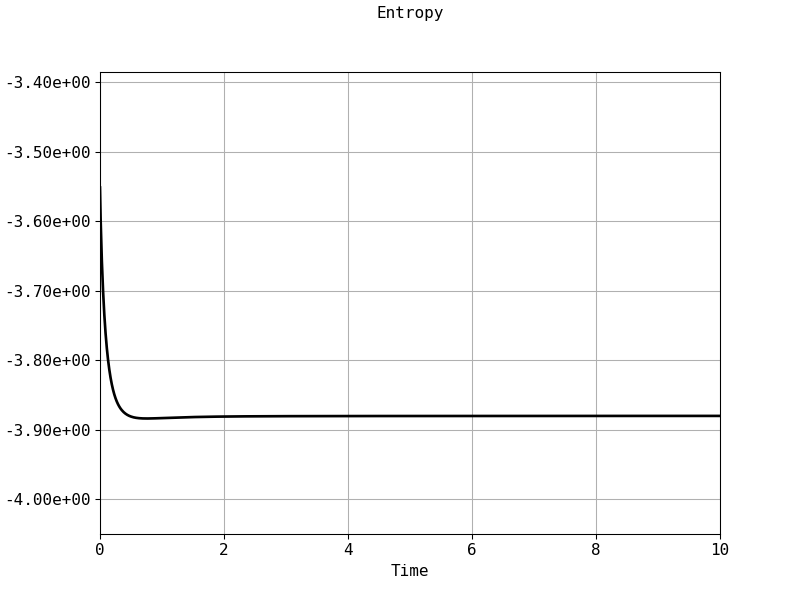}
\end{tabular}
& 
\begin{tabular}{c}\includegraphics[width=0.45\textwidth]{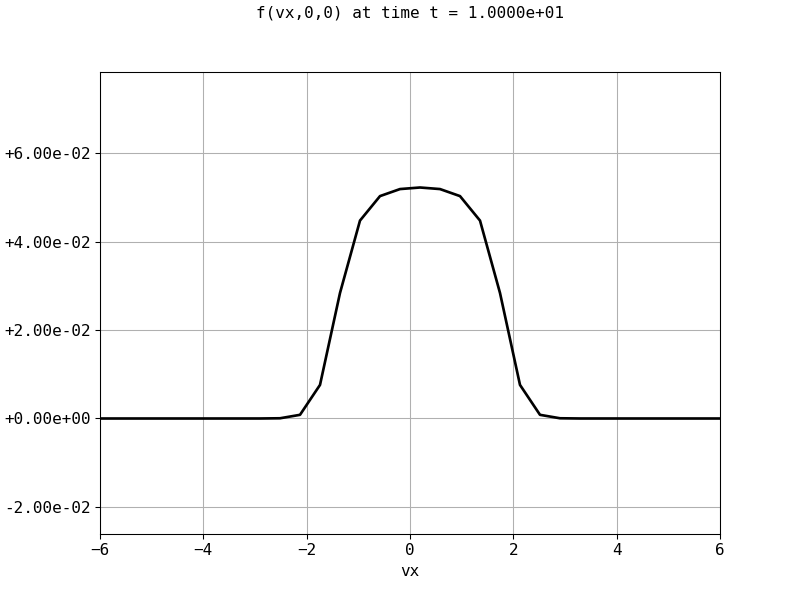}
\end{tabular}

\\

\begin{tabular}{c}
\includegraphics[width=0.45\textwidth]{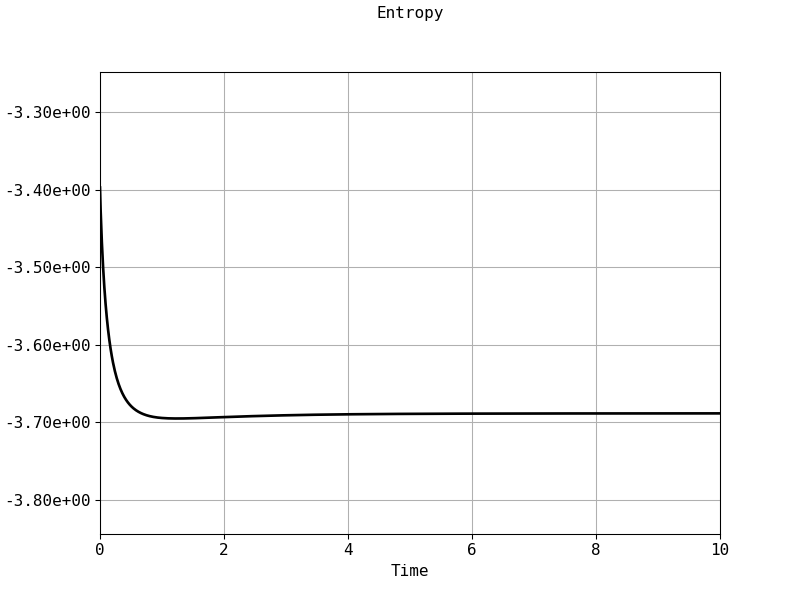}
\end{tabular}
& 
\begin{tabular}{c}\includegraphics[width=0.45\textwidth]{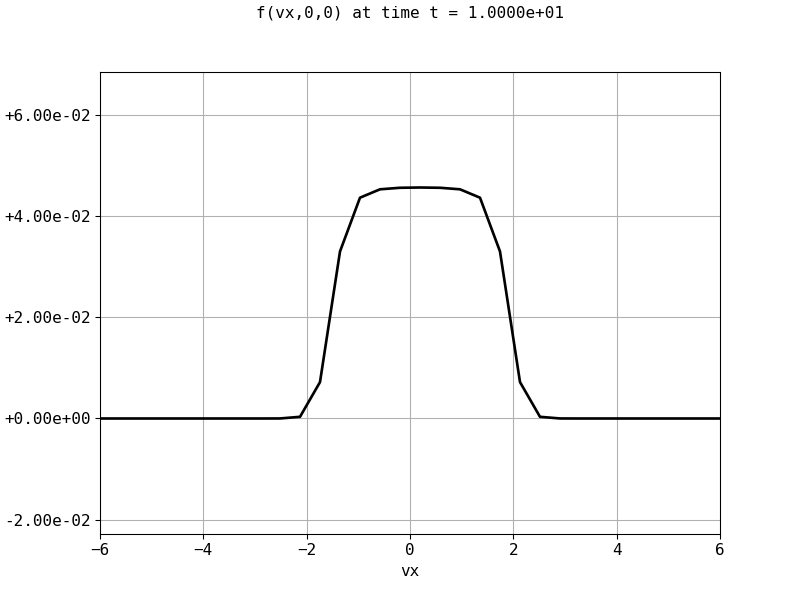}
\end{tabular}

\end{tabular}
\caption{\textbf{Trends to equilibrium.}  Time evolution of the entropy (left) and  profile of $f_{h}(10, (\cdot,0,0))$ (right) for the 3D Fermi-Dirac case with initial datum $f^{0,3}$ and $\hbar = r \hbar^{*}$ with  $r \in \{0.9, 0.95\}$ (top to bottom). } \label{fermions_3d_ball_indicator_r08-095}

\end{figure}

Following the methodology proposed for the 2D Fermi-Dirac case, we also tested the case $\hbar = 0.99\,\hbar^{*}$ where we observe again a numerical blow-up of the distribution instead of the relaxation towards a quantum maxwellian. Again, this blow-up is due to Gibbs phenomenon induced by the treatment of an almost discontinous signal with Fourier transforms. In the present case, this numerical blow-up is highlighted by the loss of the decay of the entropy (see Figure \ref{fermions_3d_ball_indicator_r099}).

\begin{figure}
\begin{tabular}{cc}

\begin{tabular}{c}
\includegraphics[width=0.45\textwidth]{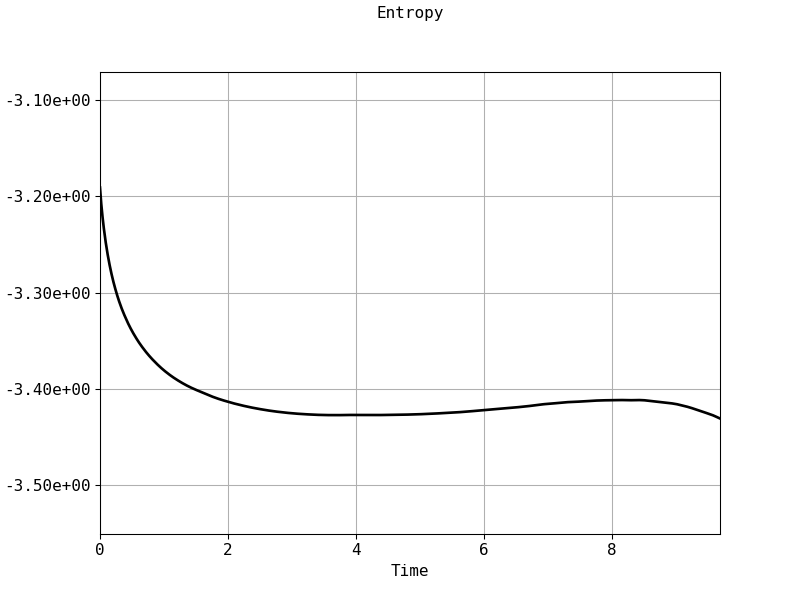}
\end{tabular}
& 
\begin{tabular}{c}\includegraphics[width=0.45\textwidth]{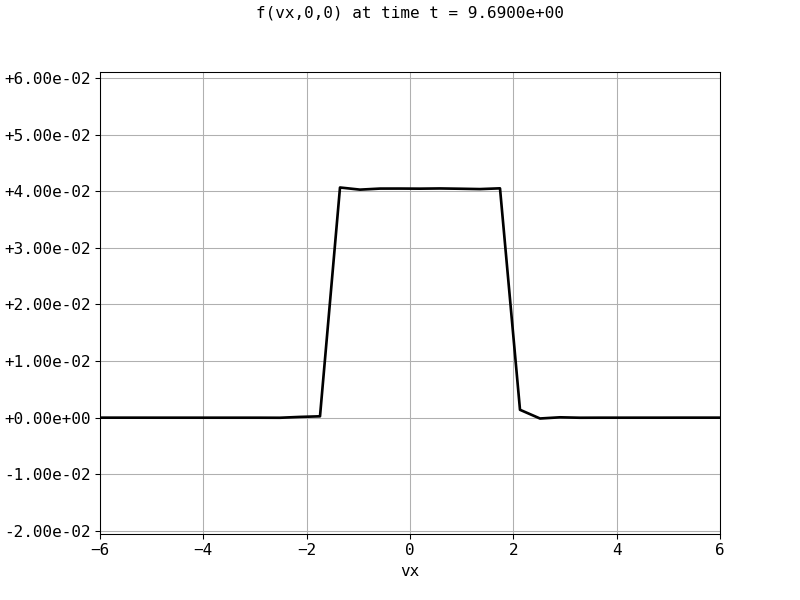}
\end{tabular}

\end{tabular}
\caption{\textbf{Trends to equilibrium.}  Time evolution of the entropy (left) and  profile of $f_{h}(9.69, (\cdot,0,0))$ (right) for the 3D Fermi-Dirac case with initial datum $f^{0,3}$ and $\hbar = 0.99 \hbar^{*}$. } \label{fermions_3d_ball_indicator_r099}

\end{figure}

 To complete this study, we also investigated the values $\hbar = r\,\hbar^{*}$ with $r \geq 1$ and the conclusions we made for the 2D Fermi-Dirac can be extended to the 3D case. For the specific case $r = 1$ where Fermi-Dirac saturation occurs, the accumulation of local errors due to Gibbs phenomenon quickly deteriorates the discrete distribution $f_{h}$ so the numerical blow-up appears after a small number of time iterations (see Figure \ref{fermions_3d_ball_indicator_r1}). For higher values of $r$, we observe a numerical blow-up of $f_{h}$ that is probably due to the fact that we have $\rho_{h}^{0} > \cfrac{5}{3\hbar^{3}}\sqrt{\cfrac{10}{\pi}}\, \left(\cfrac{4\pi e_{h}^{0}}{3}\right)^{\frac{3}{2}}$ and that, in such case, we cannot identify the time behaviour of the distribution (see Figure \ref{fermions_3d_ball_indicator_r101} for $r=1.01$).

\begin{figure}
\begin{tabular}{cc}

\begin{tabular}{c}
\includegraphics[width=0.45\textwidth]{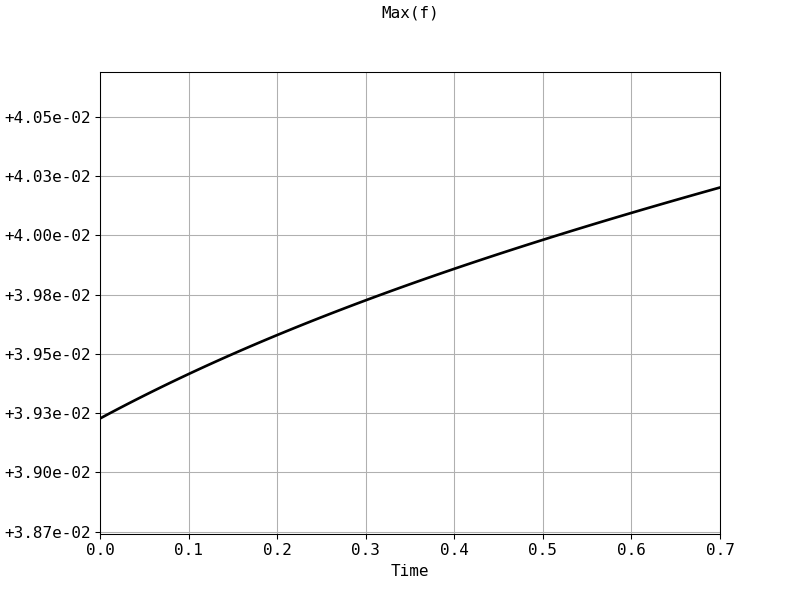}
\end{tabular}
& 
\begin{tabular}{c}\includegraphics[width=0.45\textwidth]{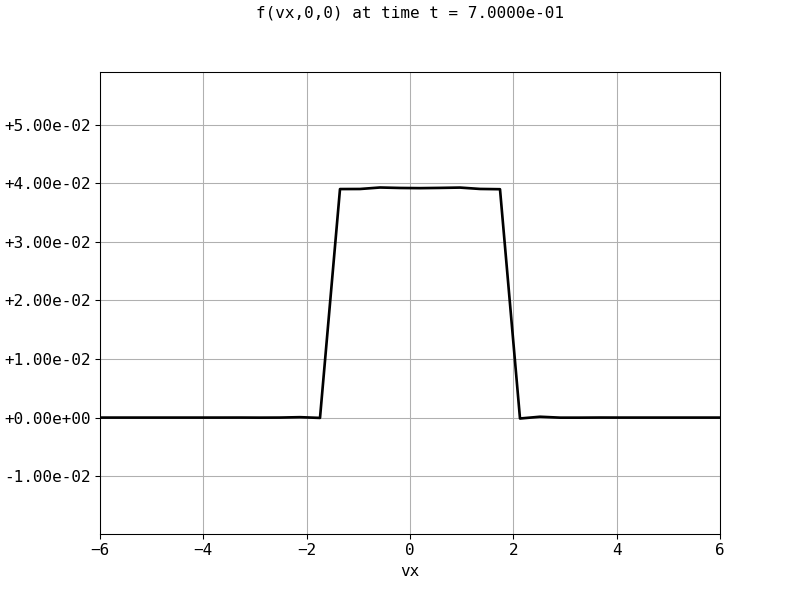}
\end{tabular}

\end{tabular}
\caption{\textbf{Trends to equilibrium.}  Time evolution of $\max f_h(t,\cdot)$ (left) and  profile of $f_{h}(7, (\cdot,0,0))$ (right) for the 3D Fermi-Dirac case with initial datum $f^{0,3}$ and $\hbar = \hbar^{*}$. } \label{fermions_3d_ball_indicator_r1}

\end{figure}

\begin{figure}
\begin{tabular}{cc}

\begin{tabular}{c}
\includegraphics[width=0.45\textwidth]{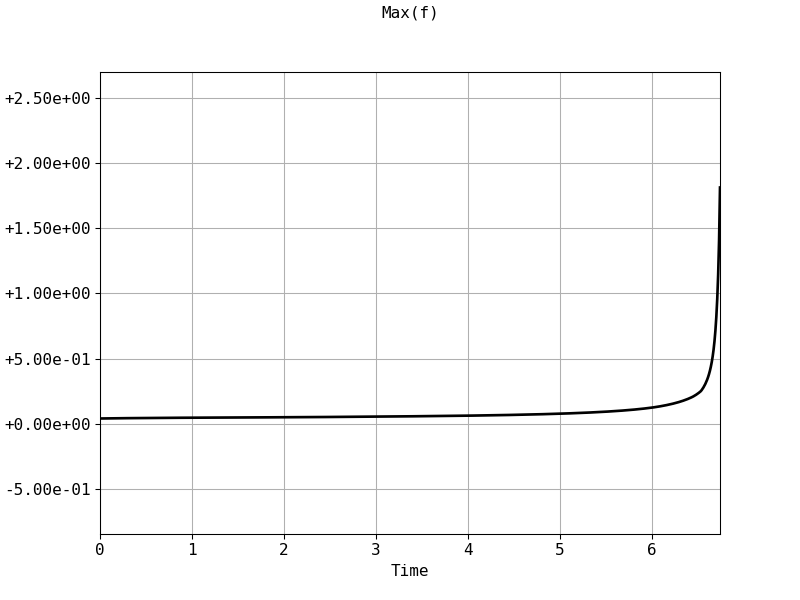}
\end{tabular}
& 
\begin{tabular}{c}\includegraphics[width=0.45\textwidth]{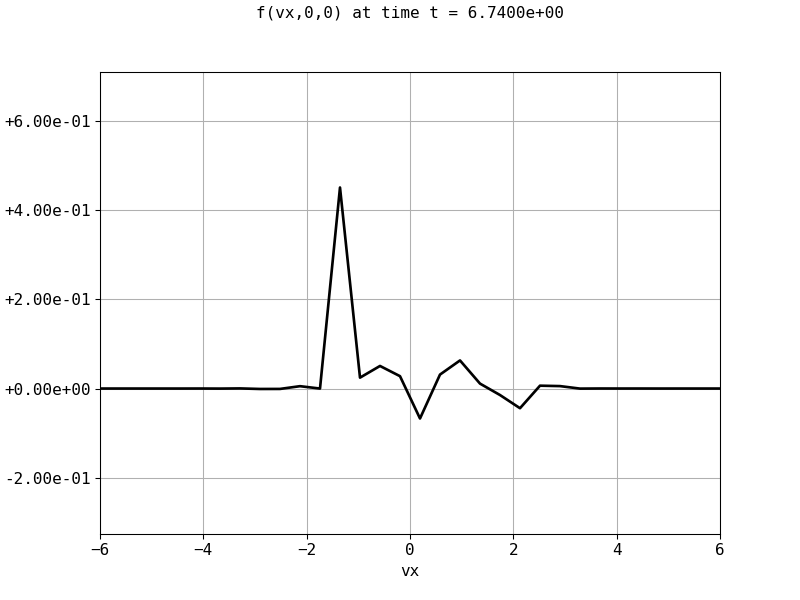}
\end{tabular}

\end{tabular}
\caption{\textbf{Trends to equilibrium.}  Time evolution of $\max f_h(t,\cdot)$ (left) and  profile of $f_{h}(6.74, (\cdot,0,0))$ (right) for the 3D Fermi-Dirac case with initial datum $f^{0,3}$ and $\hbar = 1.01 \hbar^{*}$. } \label{fermions_3d_ball_indicator_r101}

\end{figure}

%
%
%

%
%
%
 
 \medskip
 
 Let us now consider the case of a classical maxwellian $f^{0,4}$. We observe the expected relaxation to the corresponding steady state for $\hbar = r\hbar^{*}$ for $r \in \{0.1,0.5,0.8,0.9\}$ (see Figures \ref{fermions_3d_classical_maxwellian_r01-05} and \ref{fermions_3d_classical_maxwellian_r09}). Note that, for the latter case, even if the entropy cannot be defined since  $\max(f^{0}) > \hbar^{-3}$, we still have the relaxation of the distribution towards to expected steady state.

\begin{figure}
\begin{tabular}{cc}

\begin{tabular}{c}
\includegraphics[width=0.45\textwidth]{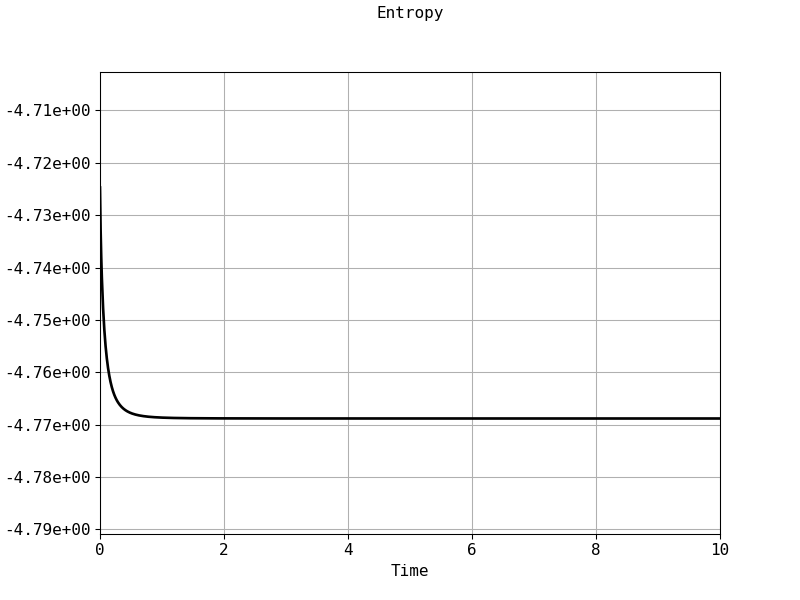}
\end{tabular}
& 
\begin{tabular}{c}\includegraphics[width=0.45\textwidth]{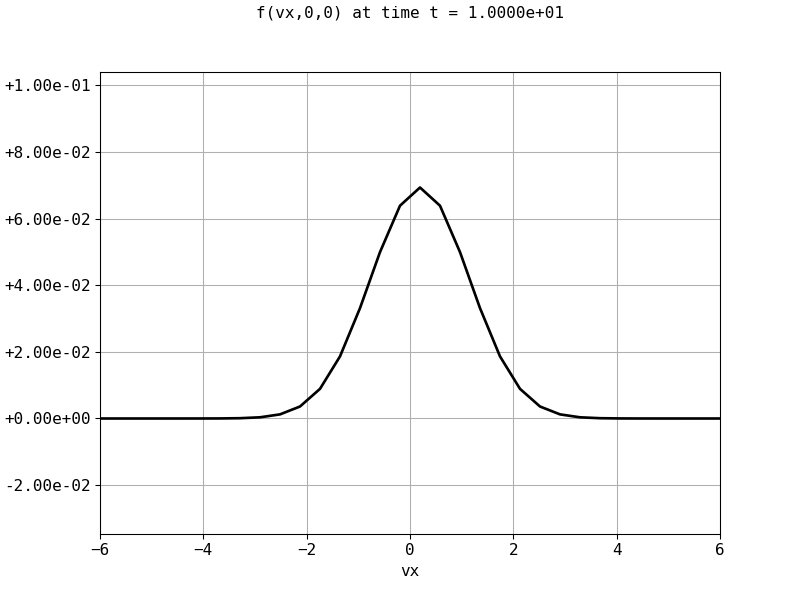}
\end{tabular}
%
%

\\

\begin{tabular}{c}
\includegraphics[width=0.45\textwidth]{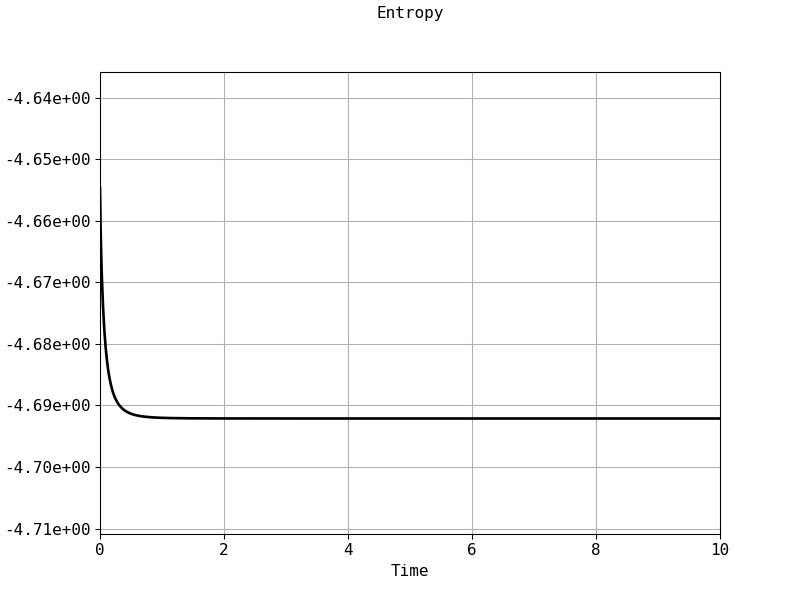}
\end{tabular}
& 
\begin{tabular}{c}\includegraphics[width=0.45\textwidth]{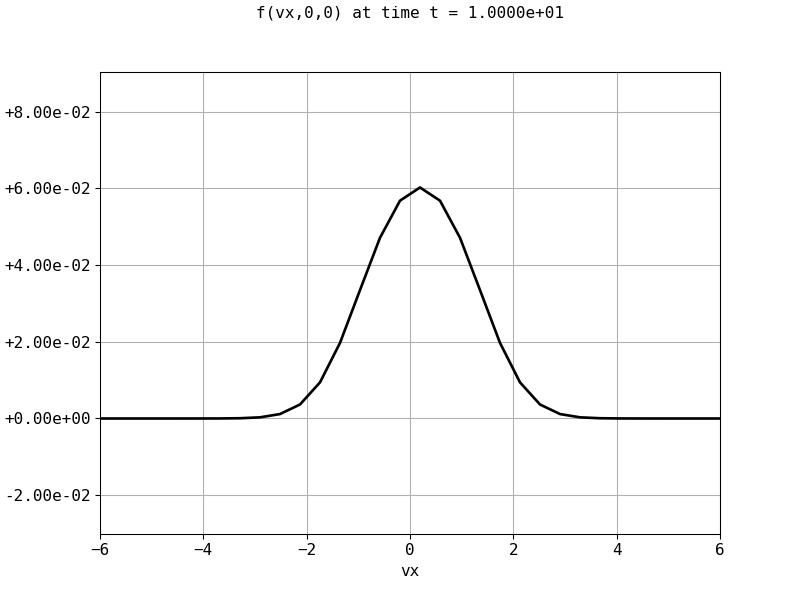}
\end{tabular}

\\

\begin{tabular}{c}
\includegraphics[width=0.45\textwidth]{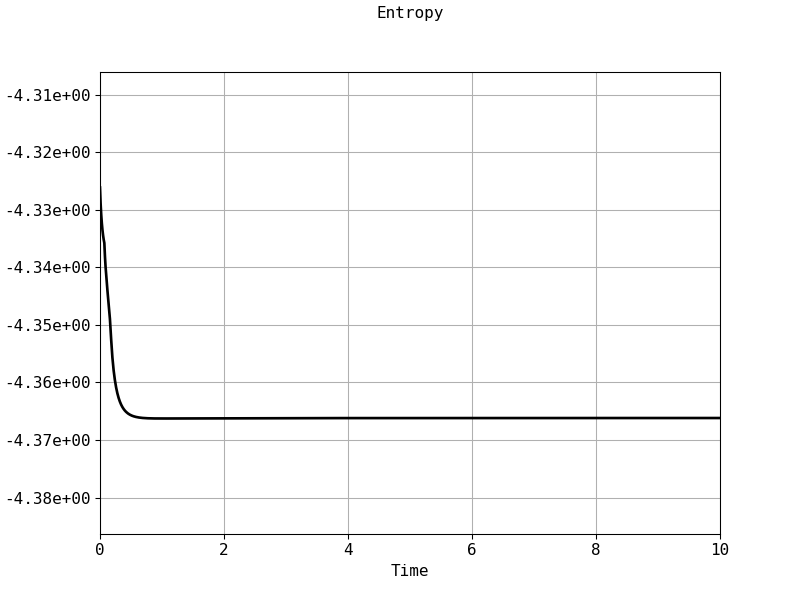}
\end{tabular}
& 
\begin{tabular}{c}\includegraphics[width=0.45\textwidth]{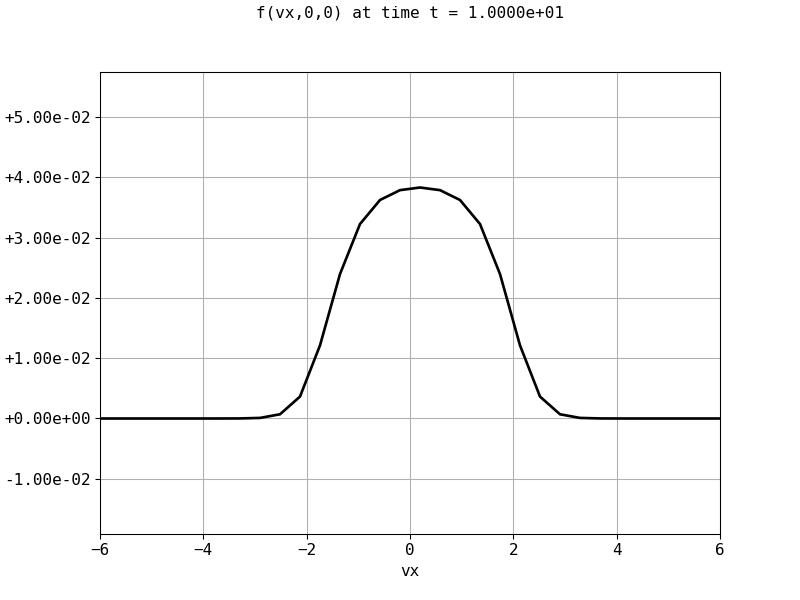}
\end{tabular}

\end{tabular}
\caption{\textbf{Trends to equilibrium.}  Time evolution of the entropy (left) and  profile of $f_{h}(10, (\cdot,0,0))$ (right) for the 3D Fermi-Dirac case with initial datum $f^{0,4}$ and $\hbar = r \hbar^{*}$ with  $r \in \{0.1, 0.2, 0.5, 0.8\}$ (top to bottom).} \label{fermions_3d_classical_maxwellian_r01-05}

\end{figure}

\begin{figure}
\begin{tabular}{cc}

\begin{tabular}{c}
\includegraphics[width=0.45\textwidth]{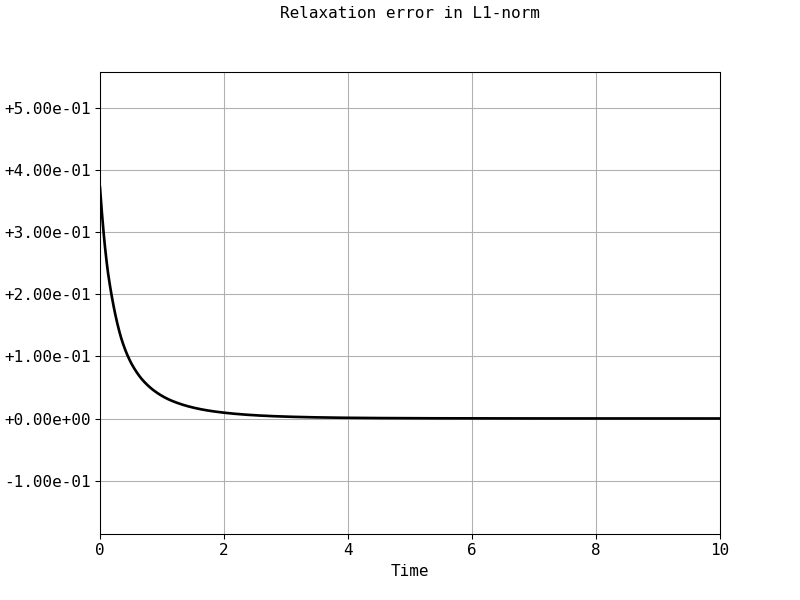}
\end{tabular}
& 
\begin{tabular}{c}\includegraphics[width=0.45\textwidth]{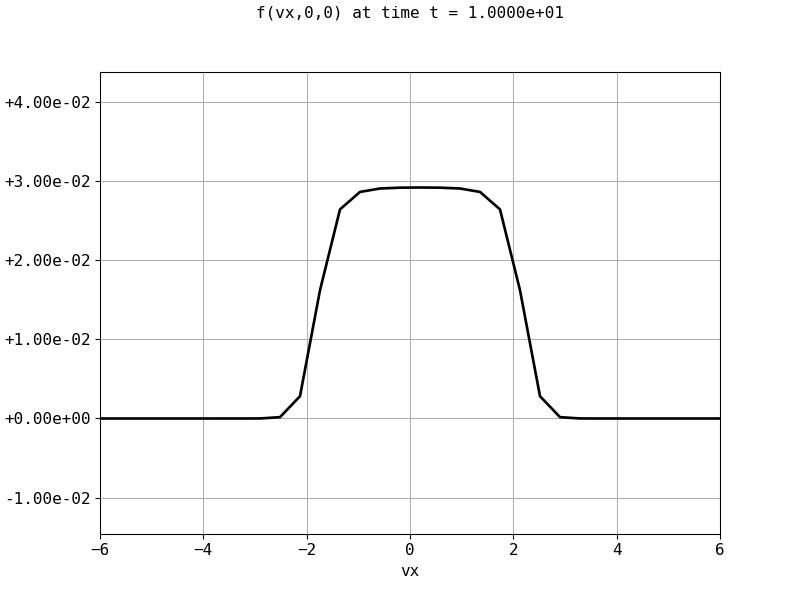}
\end{tabular}

\end{tabular}
\caption{\textbf{Trends to equilibrium.}  Time evolution of $\|f_{h}(t,\cdot)-f_{h}^{\infty}\|_{\ell^{1}}$ (left) and  profile of $f_{h}(10, (\cdot,0,0))$ (right) for the 3D Fermi-Dirac case with initial datum $f^{0,4}$ and $\hbar = 0.9 \hbar^{*}$.} \label{fermions_3d_classical_maxwellian_r09}

\end{figure}

Due to the coarse velocity mesh, taking higher values of $r$ leads to a numerical blow-up, even if $\hbar < \hbar^{*}$ (see Figure \ref{fermions_3d_classical_maxwellian_r095-099}).

\begin{figure}
\begin{tabular}{cc}

\begin{tabular}{c}
\includegraphics[width=0.45\textwidth]{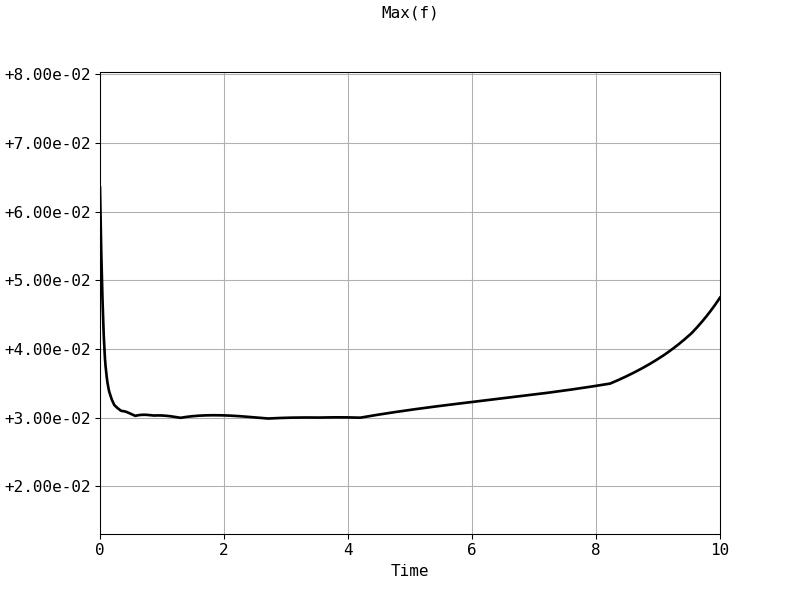}
\end{tabular}
& 
\begin{tabular}{c}\includegraphics[width=0.45\textwidth]{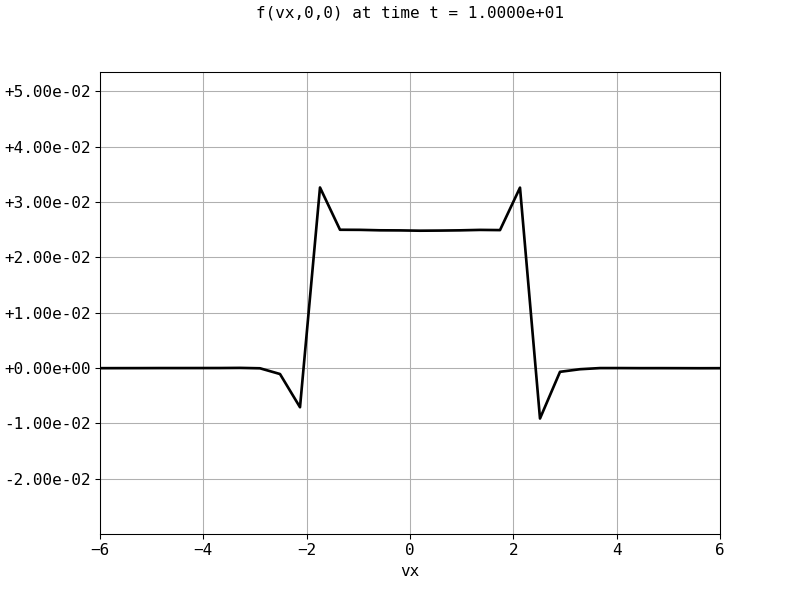}
\end{tabular}

\\

\begin{tabular}{c}
\includegraphics[width=0.45\textwidth]{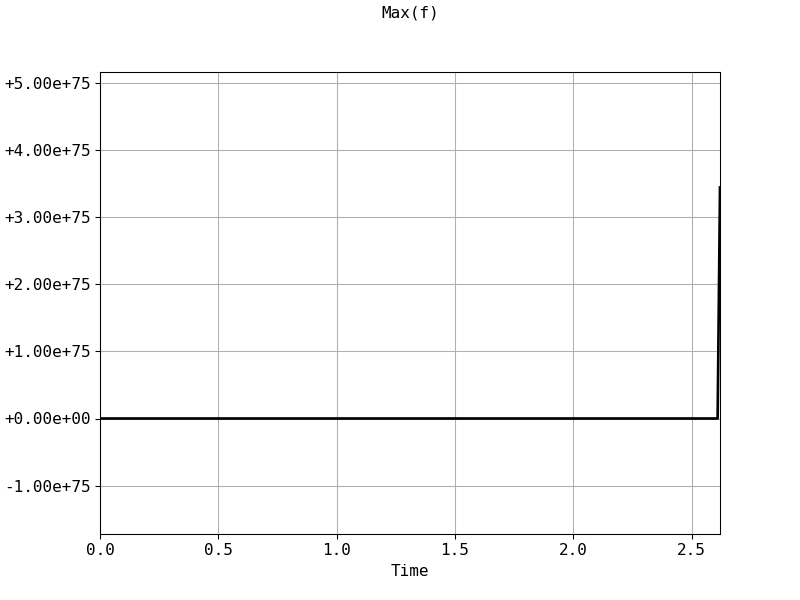}
\end{tabular}
& 
\begin{tabular}{c}\includegraphics[width=0.45\textwidth]{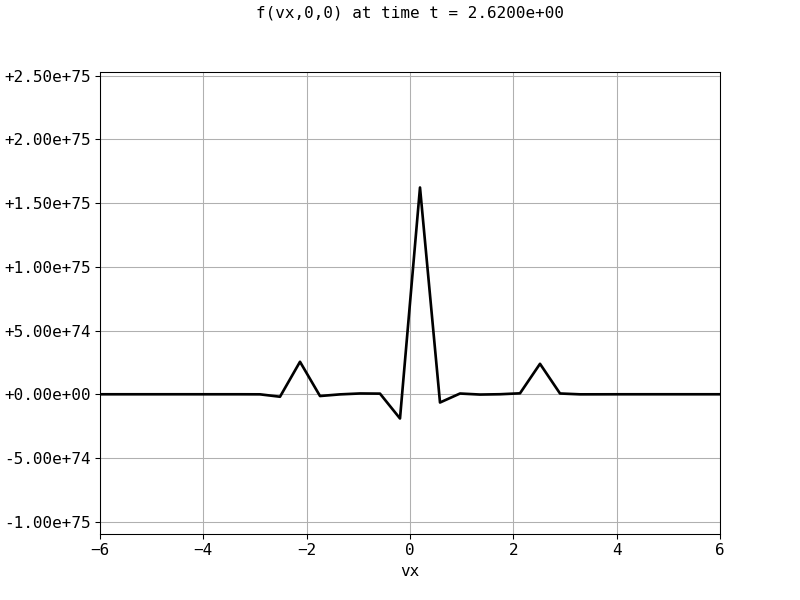}
\end{tabular}

\end{tabular}
\caption{\textbf{Trends to equilibrium.}  Time evolution of $\max f_h(t,\cdot)$ (left) and $\{v_y=v_z=0\}$--blowup profile of $f_{h}$ (right) for the 3D Fermi-Dirac case with initial datum $f^{0,4}$ and $\hbar = r \hbar^{*}$ with $r \in \{0.95,0.99\}$(top to bottom). } \label{fermions_3d_classical_maxwellian_r095-099}

\end{figure}

Finally, taking $r \geq 1$ leads to a numerical blow-up of the distribution as for the 2D case with similar conditions on $\hbar$ (see Figure \ref{fermions_3d_classical_maxwellian_r101}).

\begin{figure}
\begin{tabular}{cc}

\begin{tabular}{c}
\includegraphics[width=0.45\textwidth]{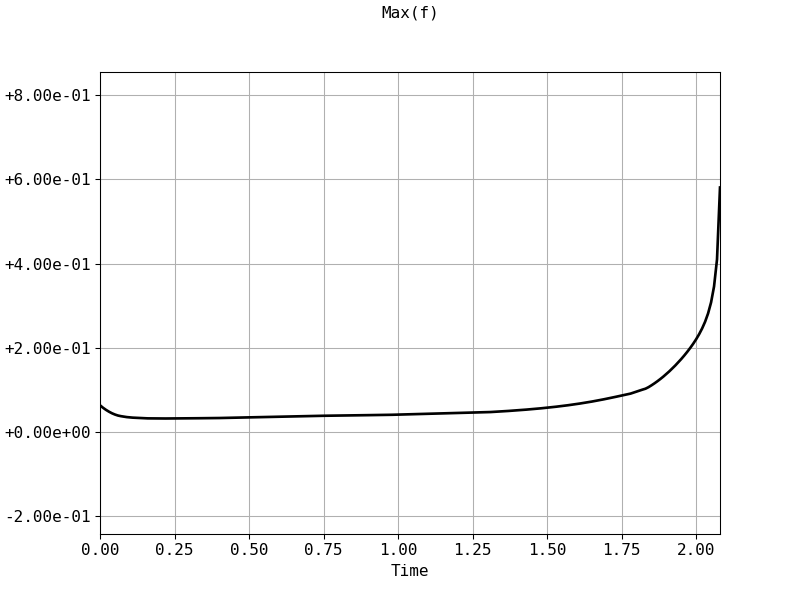}
\end{tabular}
& 
\begin{tabular}{c}\includegraphics[width=0.45\textwidth]{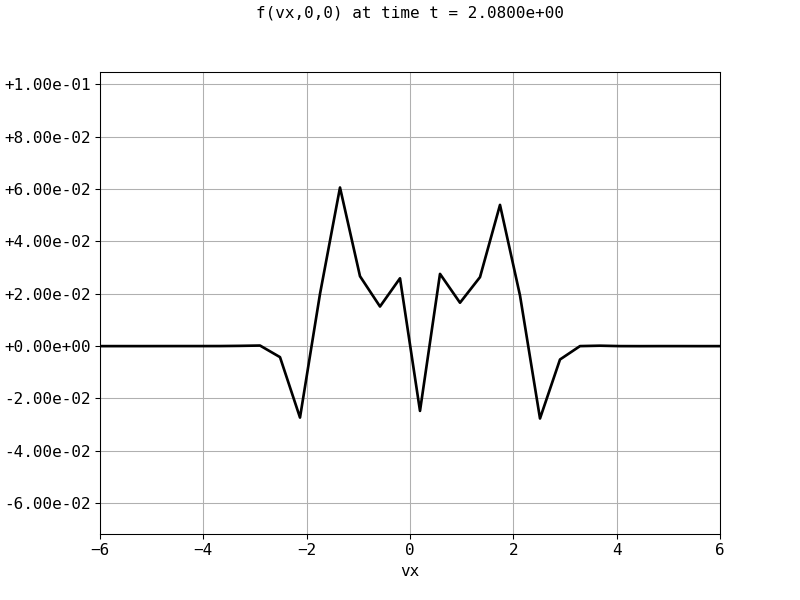}
\end{tabular}

\\

\begin{tabular}{c}
\includegraphics[width=0.45\textwidth]{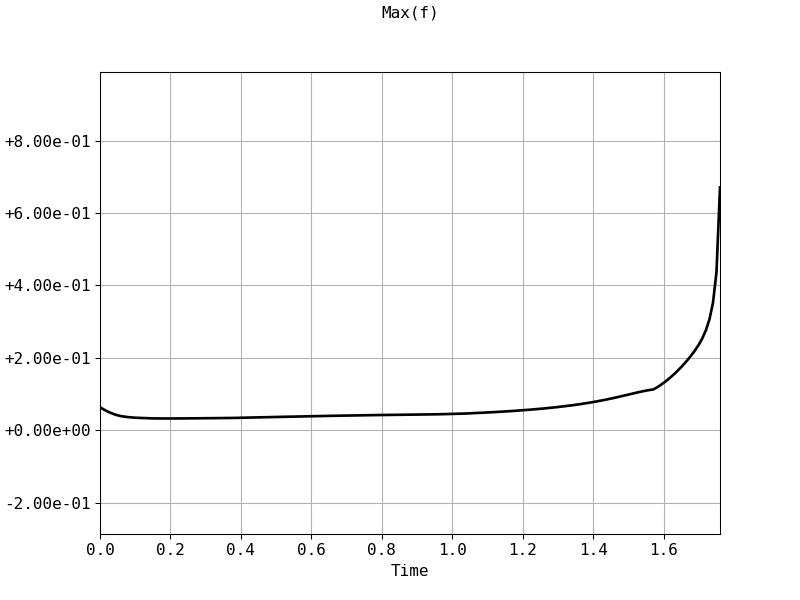}
\end{tabular}
& 
\begin{tabular}{c}\includegraphics[width=0.45\textwidth]{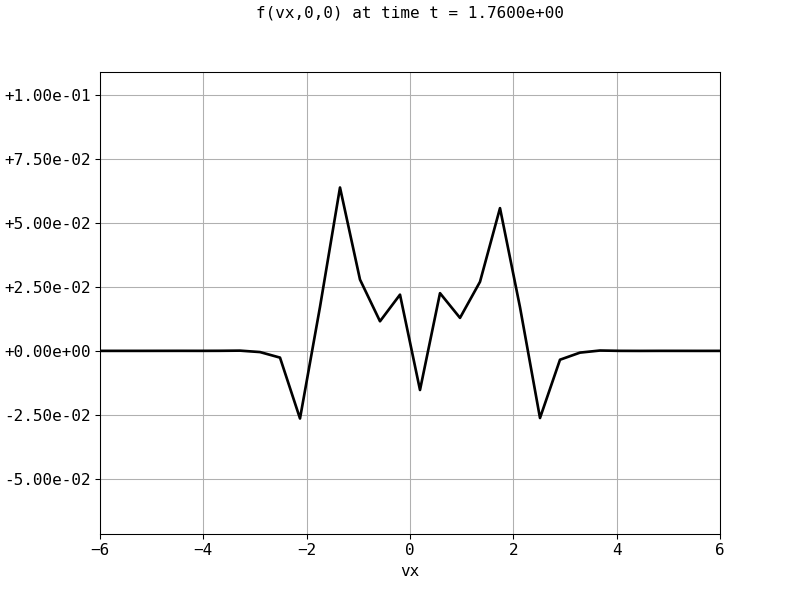}
\end{tabular}

\end{tabular}
\caption{\textbf{Trends to equilibrium.}  Time evolution of $\max f_h(t,\cdot)$ (left) and $\{v_y=v_z=0\}$--blowup profile of $f_{h}$ (right) for the 3D Fermi-Dirac case with initial datum $f^{0,4}$ and $\hbar = r \hbar^{*}$ with $r \in \{1,1.01\}$ (top to bottom).} \label{fermions_3d_classical_maxwellian_r101}

\end{figure}

\subsubsection{3D Bose-Einstein relaxation}

We now investigate the relaxation in time of the discrete distribution $f_{h}$ to the quantum $\mathcal{M}_{q}$ or the condensate distribution $\widetilde{\mathcal{M}_{q}}$ according to the value $\hbar$ with respect to 
\[ \hbar^{*} = \cfrac{1}{(\rho_{h}^{0})^{1/3}} \, \left(\cfrac{4\pi\,e_{h}^{0}}{3}\right)^{1/2} \, \cfrac{\zeta(3/2)^{5/6}}{\zeta(5/2)^{1/2}}.\] 
As for the Fermi-Dirac cases, we consider a $32^{3}$ velocity grid on the domain $[-L,L]^{3} = [-6,6]^{3}$ and an initial distribution defined as \eqref{fermions_3d_init_cmax} or \eqref{fermions_3d_init_ball}. 
This gives us the thresholds $\hbar^{*} \approx 3.97285$, and $\hbar^{*} \approx 4.81755$, respectively. We expect that the computed discrete distribution $f_{h}$ relaxes to $\mathcal{M}_{q}$ when $\hbar < \hbar^{*}$ and to $\widetilde{\mathcal{M}_{q}}$ when $\hbar \geq \hbar^{*}$ (\textit{i.e.} a blow-up of the distribution localized in $\mathbf{u}_{h}^{0}$). 

The computational limitations we noticed in the previous paragraph for handling 3D quantum cases are more restrictive here than for the 3D Fermi-Dirac case. Indeed, if $\hbar \to (\hbar^{*})^{-}$, the $32^{3}$ velocity grid we considered is not sufficient to catch the stiff gradients of the limit state $f^{\infty}$.
We can see in Figure \ref{bosons_3d_ball_indicator_r08} that for $\hbar = r\,\hbar^{*}$ with $r \in \{0.1,0.5,0.8\}$ we can reach successfully the expected limit state as the entropy (and the numerical relaxation error) decrease. 

\begin{figure}
\begin{tabular}{cc}

\begin{minipage}{0.45\textwidth}
\includegraphics[width=\textwidth]{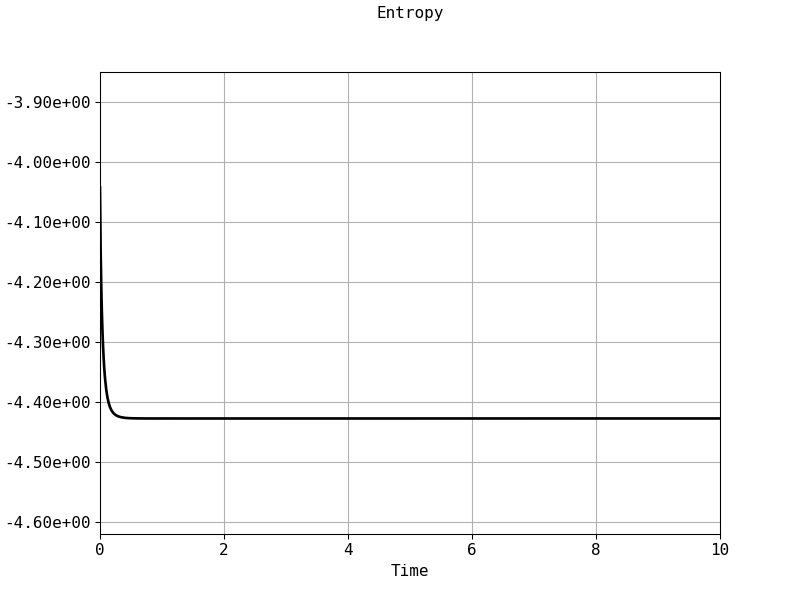}
\end{minipage}
& 
\begin{minipage}{0.45\textwidth}
\includegraphics[width=\textwidth]{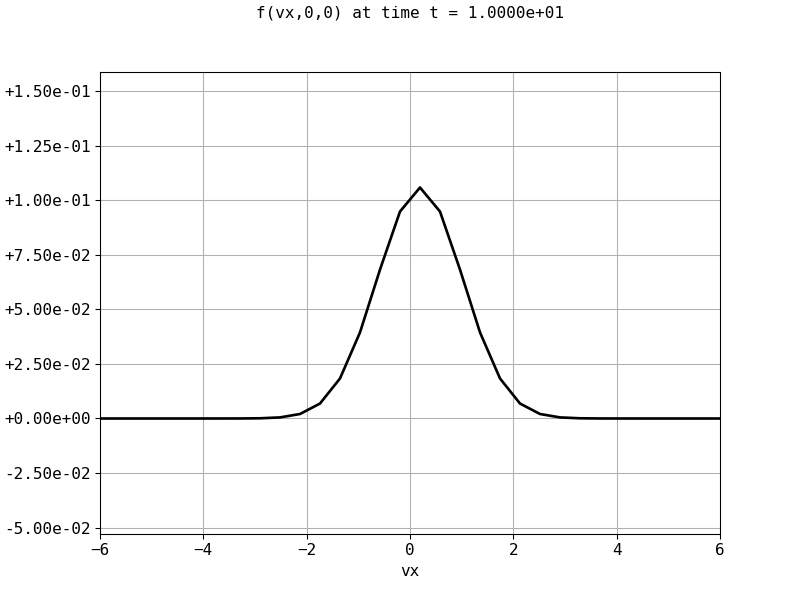}
\end{minipage}

\\

\begin{minipage}{0.45\textwidth}
\includegraphics[width=\textwidth]{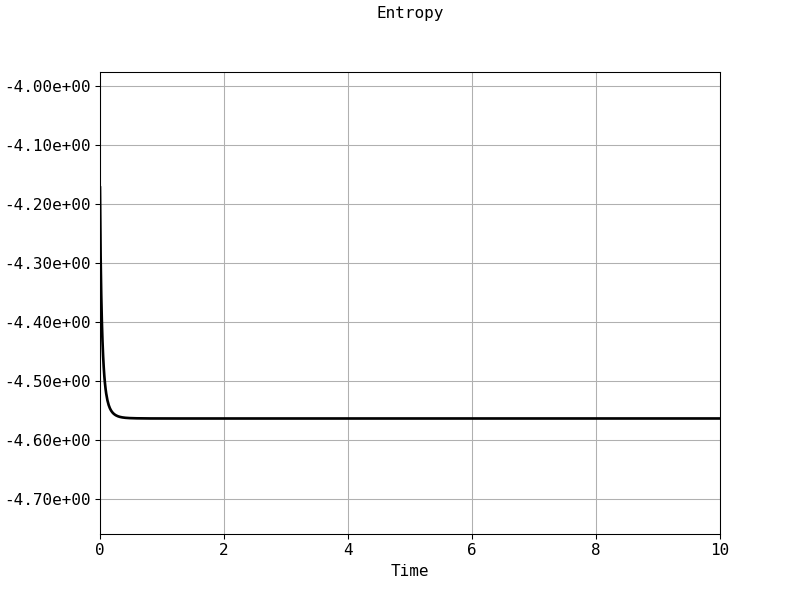}
\end{minipage}
&
\begin{minipage}{0.45\textwidth}
\includegraphics[width=\textwidth]{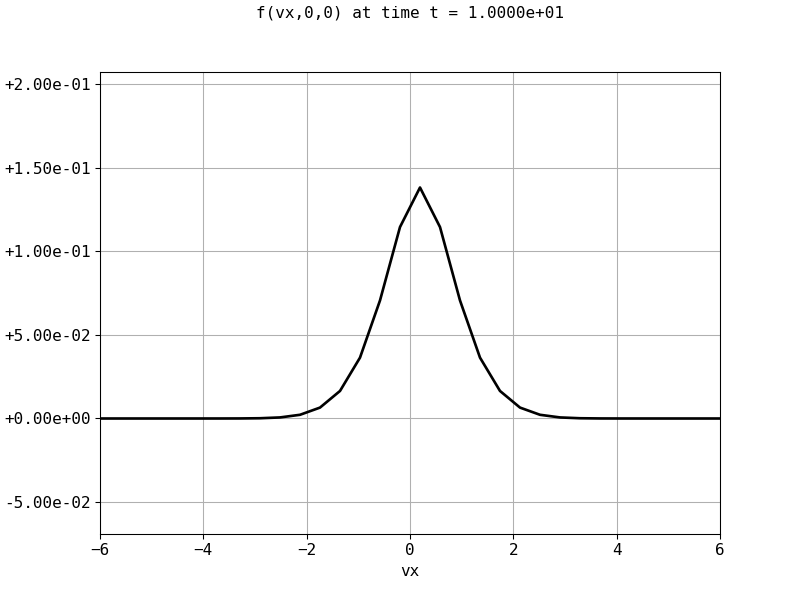}
\end{minipage}

\\

\begin{minipage}{0.45\textwidth}
\includegraphics[width=\textwidth]{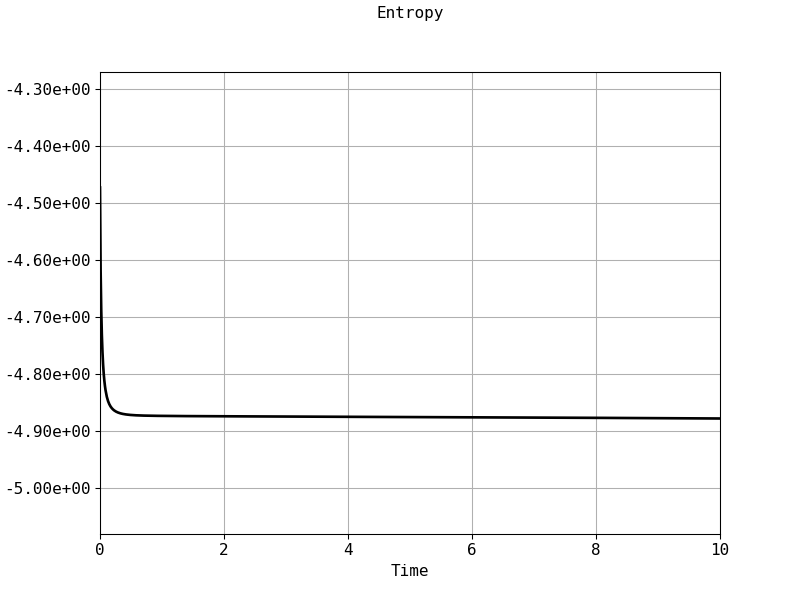}
\end{minipage}
& 
\begin{minipage}{0.45\textwidth}
\includegraphics[width=\textwidth]{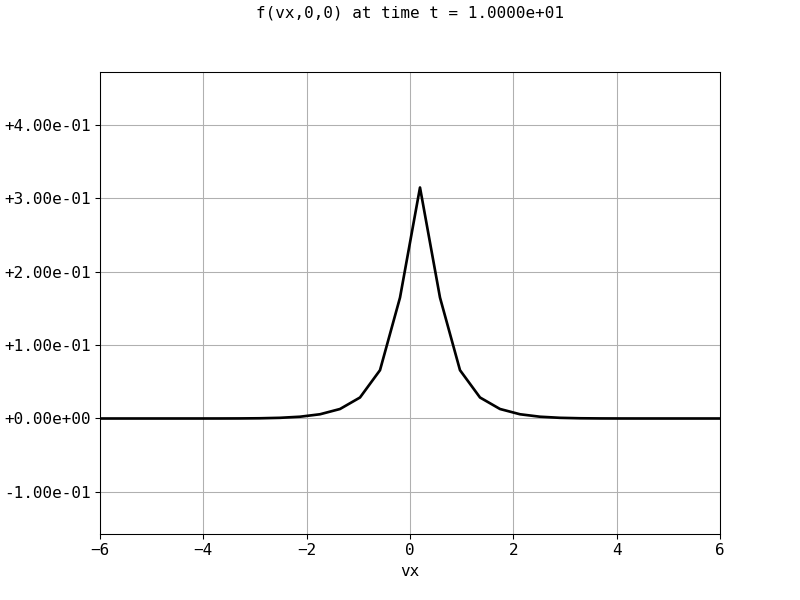}
\end{minipage}

\end{tabular}
\caption{\textbf{Trends to equilibrium.}  Time evolution of the entropy (left) and  profile of $f_{h}(10, (\cdot,0,0))$ (right) for the 3D Bose-Einstein case with initial datum $f^{0,3}$ and $\hbar = r \hbar^{*}$ with  $r \in \{0.1,  0.5, 0.8\}$ (top to bottom).} \label{bosons_3d_ball_indicator_r08}
\end{figure}


However, for higher values of $r \in (0.8,1)$, the coarseness of the velocity grid deteriorates the distribution profile after a small number of time iterations, even if $\hbar < \hbar^{*}$ (see Figure \ref{bosons_3d_ball_indicator_r095}).

\begin{figure}
\begin{tabular}{cc}

\begin{minipage}{0.45\textwidth}
\includegraphics[width=\textwidth]{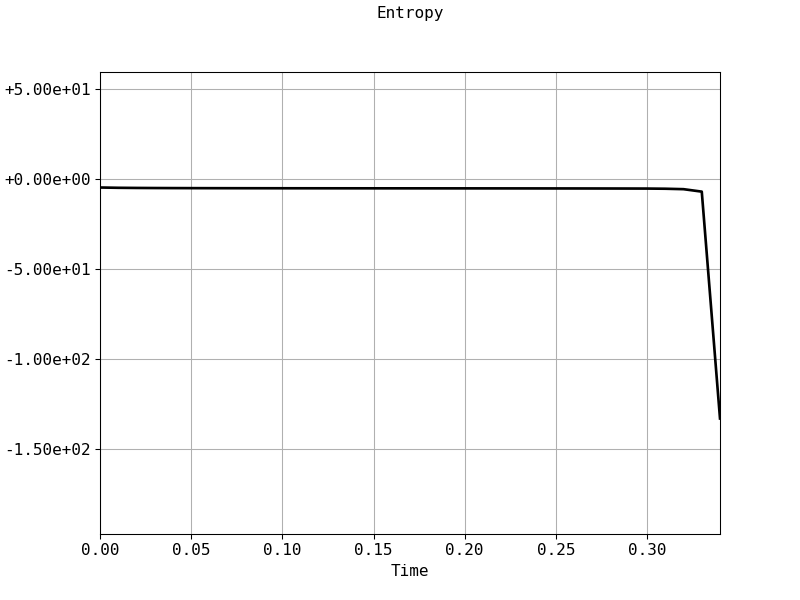}
\end{minipage}
& 
\begin{minipage}{0.45\textwidth}
\includegraphics[width=\textwidth]{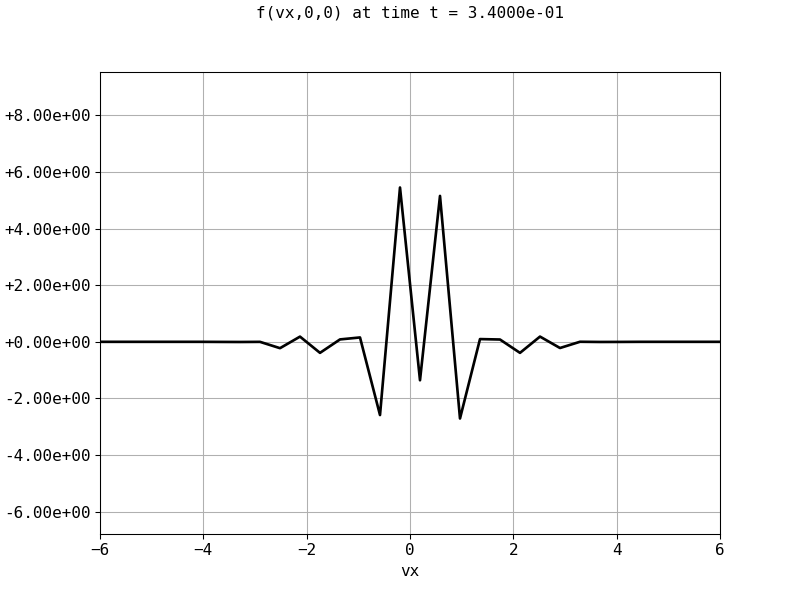}
\end{minipage}

\\

\begin{minipage}{0.45\textwidth}
\includegraphics[width=\textwidth]{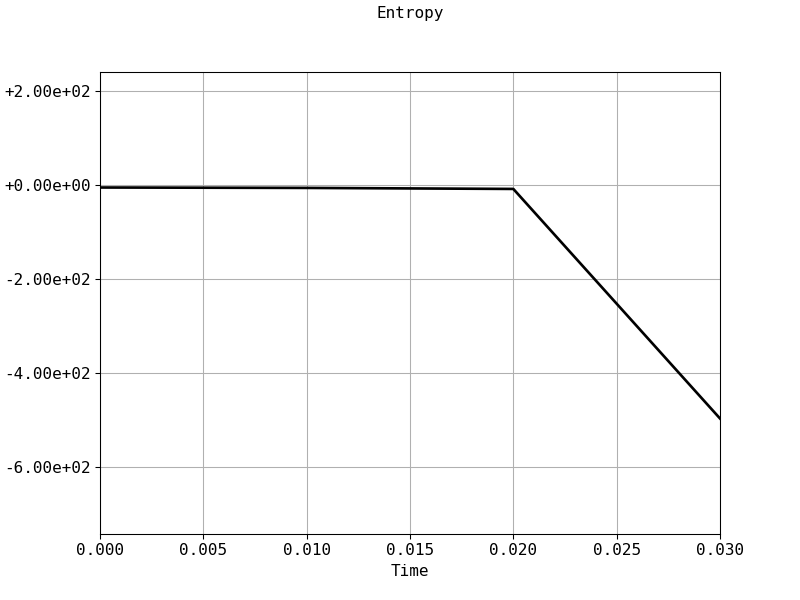}
\end{minipage}
&
\begin{minipage}{0.45\textwidth}
\includegraphics[width=\textwidth]{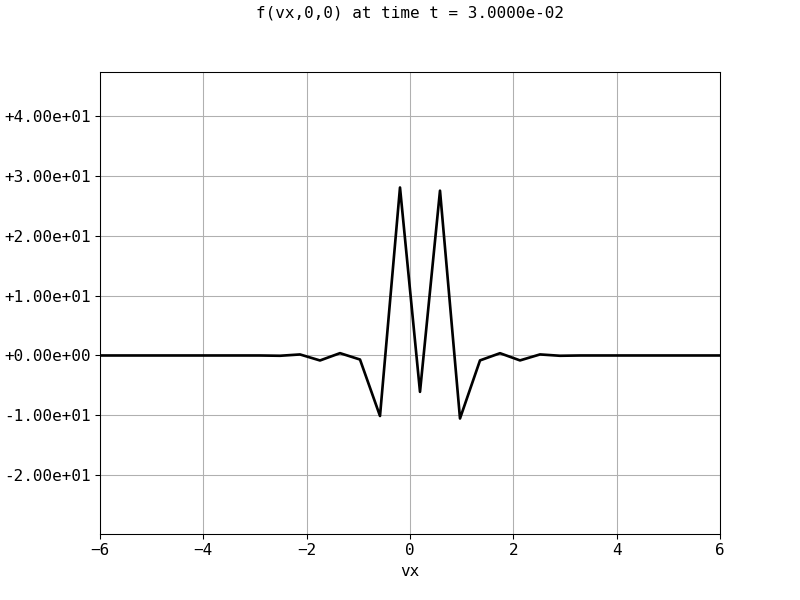}
\end{minipage}

\end{tabular}

\caption{\textbf{Trends to equilibrium.}   Time evolution of the entropy (left) and $\{v_y=v_z=0\}$--blowup profile of $f_{h}$ (right) for the 3D Bose-Einstein case with initial datum $f^{0,3}$ and $\hbar = r \hbar^{*}$ with $r \in \{0.9,0.95\}$ (top to bottom)}.
 \label{bosons_3d_ball_indicator_r095}
\end{figure}

Surprisingly, considering $\hbar \geq \hbar^{*}$, \textit{i.e.} a Bose-Einstein condensation and the associated singular limit state $\widetilde{\mathcal{M}}_{q}$, leads to numerical results that are less deteriorated than for $r \in \{0.9,0.95\}$ (see Figures \ref{bosons_3d_ball_indicator_r105}). 
We observe the expected blow-up in time of the discrete distribution $f_{h}$ but it occurs after a simulation time $T^{*}$ that seems to decreases as long as $\hbar$ becomes large. This phenomenon may be due to the fact that, in such case, the stiffness of the distribution is due to the Dirac mass and the singular quantum maxwellian in the definition of $\widetilde{\mathcal{M}_{q}}$, whereas in the case where $\hbar \to (\hbar^{*})^{-}$, the stiffness of $f^{\infty}$ is only due to the quantum maxwellian $\mathcal{M}_{q}$ that is almost singular in $\mathbf{u}_{h}^{0}$. 

\begin{figure}
\begin{tabular}{cc}
\begin{minipage}{0.45\textwidth}
\includegraphics[width=\textwidth]{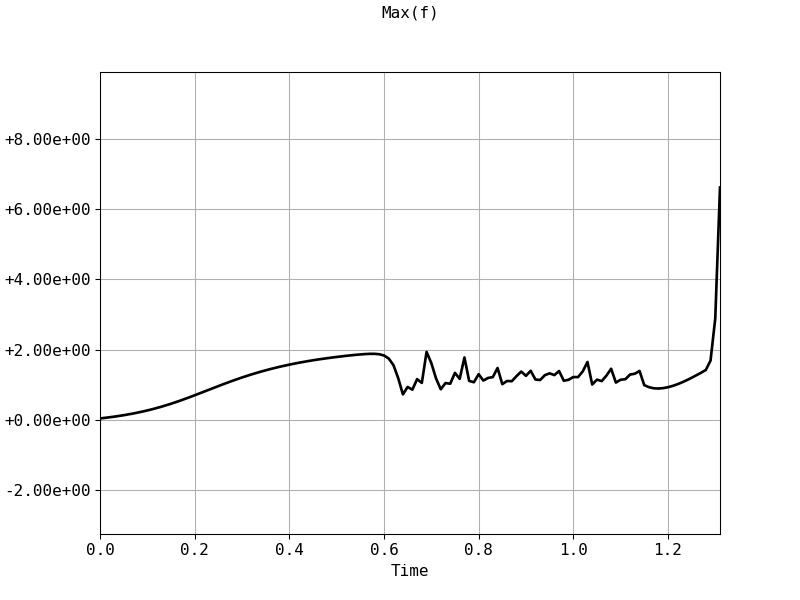}
\end{minipage}
&
\begin{minipage}{0.45\textwidth}
\includegraphics[width=\textwidth]{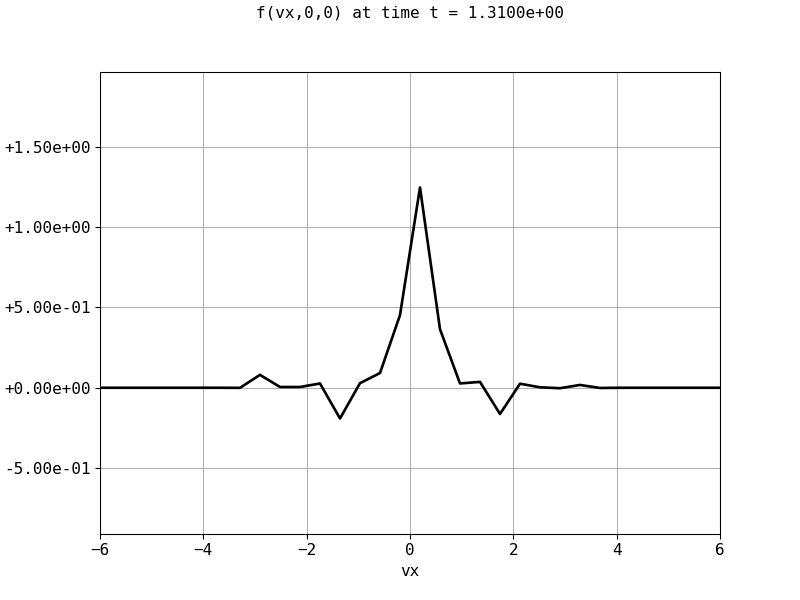}
\end{minipage}


\\

\begin{minipage}{0.45\textwidth}
\includegraphics[width=\textwidth]{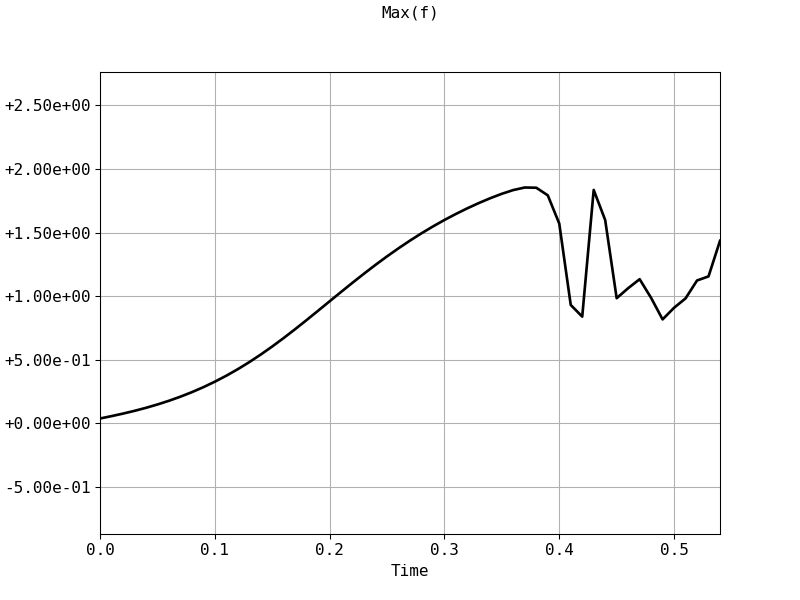}
\end{minipage}
&
\begin{minipage}{0.45\textwidth}
\includegraphics[width=\textwidth]{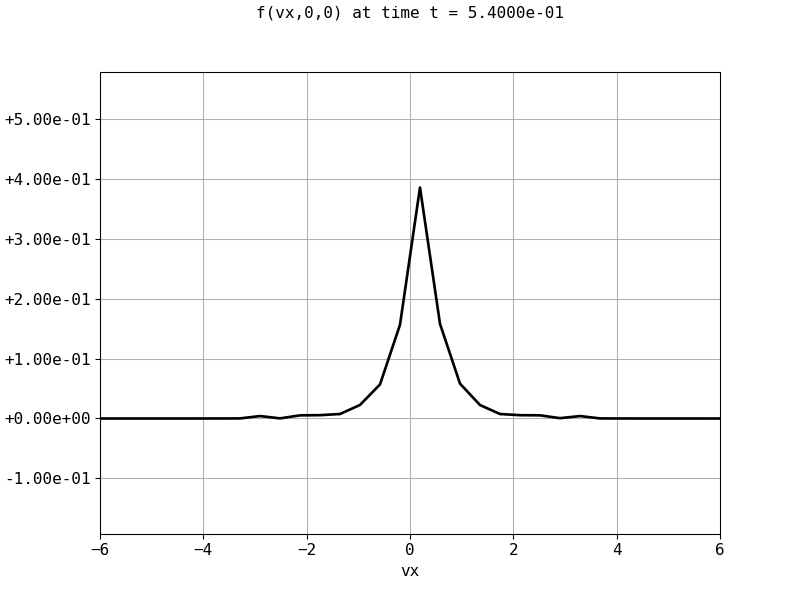}
\end{minipage}
\end{tabular}

\caption{\textbf{Trends to equilibrium.} Time evolution of $\max f_h(t,\cdot)$ (left) and $\{v_y=v_z=0\}$--blowup profile (right) for the 3D Bose-Einstein case with initial datum $f^{0,3}$ and $\hbar = r\hbar^{*}$ with $r \in{1, 1.05}$ (top to bottom}.
\label{bosons_3d_ball_indicator_r105}
\end{figure}


%
%

 All these observations are confirmed when we use the classical maxwellian $\mathcal{M}_{c}$ as initial datum (Figures omitted for redundancy).

\section{Conclusion and perspectives}
In this work, we illustrated the efficiency and accuracy of the Fast Spectral algorithm from \cite{FilbetHuJin:2012} along with the rescaling velocity method from \cite{FilbetRey:2012} for approximating the solutions to the Boltzmann-Nordheim equation. In particular, we showed that this numerical method is able to reproduce very accurately some of the main mathematical features of this equation, including the convergence towards singular steady states (namely the phenomena of Fermi-Dirac saturation and Bose-Einstein condensation). 
As a consequence, we believe that this method is perfectly suited to study some of the still unknown behavior of this model, in particular the rate of convergence towards the (singular) 3D Bose-Einstein condensates, that was conjectured in \cite{EscobedoVelazquez2015}.
Unfortunately, the accuracy needed for computing such singular solution in 3D is still unavailable in our approach.
 This motivates us for studying ways to make our simulation code faster in order to consider finer meshes and complete the classification of relaxation phenomena. This will be the main topic of the upcoming followup \cite{MoutonReyFast} of our work.
  
  \appendix
  
\section{Fast evaluation of the collision kernel} \label{eval_collision}

Let us briefly present how we practically compute the collision operator, by focusing on its most computationally expensive part $\mathcal{Q}_{1,q}^{R}(G_{h},G_{h},G_{h})$.

We remark that, for any $\mathbf{n} \in I_{\mathbf{N}}$,
\begin{displaymath}
\begin{split}
\widehat{\mathcal{Q}_{1,q}^{R}(G_{h},G_{h},G_{h})}_{\mathbf{n}} &= \sum_{\substack{\mathbf{k},\mathbf{l},\mathbf{m} \, \in \, I_{N} \\ \mathbf{k}+\mathbf{l}+\mathbf{m} \,=\, \mathbf{n}}} \beta^{R}(\mathbf{k}+\mathbf{m},\mathbf{l}+\mathbf{m}) \, \widehat{G}_{\mathbf{k}} \, \widehat{G}_{\mathbf{l}}\,\widehat{G}_{\mathbf{m}} \\
&= C\, \sum_{p\,=\,0}^{P-1} \sum_{\substack{\mathbf{k},\mathbf{l},\mathbf{m} \, \in \, I_{N} \\ \mathbf{k}+\mathbf{l}+\mathbf{m} \,=\, \mathbf{n}}} \alpha_{R,p}(\mathbf{k}+\mathbf{m})\,\alpha_{R,p}'(\mathbf{l}+\mathbf{m}) \, \widehat{G}_{\mathbf{k}}\,\widehat{G}_{\mathbf{l}} \, \widehat{G}_{\mathbf{m}} \\
&= C\, \sum_{p\,=\,0}^{P-1} \sum_{\mathbf{k},\mathbf{l}\,\in \, I_{N}} \alpha_{R,p}(\mathbf{n}-\mathbf{l})\,\alpha_{R,p}'(\mathbf{n}-\mathbf{k}) \, \widehat{G}_{\mathbf{k}}\,\widehat{G}_{\mathbf{l}} \, \widehat{G}_{\mathbf{n}-\mathbf{k}-\mathbf{l}} \\
&= C\, \sum_{p\,=\,0}^{P-1} \sum_{\mathbf{k},\mathbf{l}\,\in \, I_{N}} \alpha_{R,p}(\mathbf{l})\,\alpha_{R,p}'(\mathbf{k}) \, \widehat{G}_{\mathbf{n}-\mathbf{k}}\,\widehat{G}_{\mathbf{n}-\mathbf{l}} \, \widehat{G}_{\mathbf{k}+\mathbf{l}-\mathbf{n}} \, .
\end{split}
\end{displaymath}
Reminding that $\widehat{G}_{\mathbf{k}+\mathbf{l}-\mathbf{n}}$ can be expressed as
\begin{displaymath}
\widehat{G}_{\mathbf{k}+\mathbf{l}-\mathbf{n}} = \cfrac{1}{N}\,\sum_{\mathbf{j}\,\in\,I_{N}} G_{\mathbf{j}}\, e^{-2i\pi\,\frac{(\mathbf{k}-\mathbf{l}-\mathbf{n})\cdot\mathbf{j}}{N}}\, ,
\end{displaymath}
We can write
\begin{displaymath}
\begin{split}
\widehat{\mathcal{Q}_{1,q}^{R}(G_{h},G_{h},G_{h})}_{\mathbf{n}} = \cfrac{C}{N}\, \sum_{p\,=\,0}^{P-1} \sum_{\mathbf{j}\,\in\,I_{N}} G_{\mathbf{j}} \, e^{2i\pi\,\frac{\mathbf{n}\cdot\mathbf{j}}{N}} \, &\left[ \sum_{\mathbf{l}\,\in\,I_{N}} \alpha_{R,p}(\mathbf{l})\,\widehat{G}_{\mathbf{n}-\mathbf{l}} \, e^{-2i\pi\,\frac{\mathbf{l}\cdot\mathbf{j}}{N}} \right] \\
&\times \left[ \sum_{\mathbf{k}\,\in\,I_{N}} \alpha_{R,p}'(\mathbf{k})\,\widehat{G}_{\mathbf{n}-\mathbf{k}} \, e^{-2i\pi\,\frac{\mathbf{k}\cdot\mathbf{j}}{N}} \right] \, .
\end{split}
\end{displaymath}
For any $\mathbf{j} \in I_{N}$, we define $\widehat{\mathfrak{g}_{R,p,\mathbf{j}}}$ and $\widehat{\mathfrak{g}_{R,p,\mathbf{j}}'}$ on $I_{N}$ such that
\begin{displaymath}
\begin{split}
(\widehat{\mathfrak{g}_{R,p,\mathbf{j}}})_{\mathbf{k}} = \alpha_{R,p}(\mathbf{k})\, e^{-2i\pi\,\frac{\mathbf{k}\cdot\mathbf{j}}{N}} \, , \qquad
(\widehat{\mathfrak{g}_{R,p,\mathbf{j}}'})_{\mathbf{k}} = \alpha_{R,p}'(\mathbf{k})\, e^{-2i\pi\,\frac{\mathbf{k}\cdot\mathbf{j}}{N}} \, ,
\end{split}
\end{displaymath}
so we have
\begin{displaymath}
\begin{split}
\widehat{\mathcal{Q}_{1,q}^{R}(G_{h},G_{h},G_{h})}_{\mathbf{n}} &= \cfrac{C}{N}\, \sum_{p\,=\,0}^{P-1} \sum_{\mathbf{j}\,\in\,I_{N}} G_{\mathbf{j}} \,e^{2i\pi\,\frac{\mathbf{n}\cdot\mathbf{j}}{N}} \left[ \sum_{\mathbf{l}\,\in\,I_{N}} (\widehat{\mathfrak{g}_{R,p,\mathbf{j}}})_{\mathbf{l}} \,\widehat{G}_{\mathbf{n}-\mathbf{l}} \right] \, \left[ \sum_{\mathbf{k}\,\in\,I_{N}} (\widehat{\mathfrak{g}_{R,p,\mathbf{j}}'})_{\mathbf{k}} \,\widehat{G}_{\mathbf{n}-\mathbf{k}} \right] \\
&= \cfrac{C}{N}\, \sum_{p\,=\,0}^{P-1} \sum_{\mathbf{j}\,\in\,I_{N}} G_{\mathbf{j}} \,e^{2i\pi\,\frac{\mathbf{n}\cdot\mathbf{j}}{N}} \, ( \widehat{\mathfrak{g}_{R,p,\mathbf{j}}} * \widehat{G_{h}})_{\mathbf{n}} \, ( \widehat{\mathfrak{g}_{R,p,\mathbf{j}}'} * \widehat{G_{h}})_{\mathbf{n}} \\
&= \cfrac{C}{N}\, \sum_{p\,=\,0}^{P-1} \sum_{\mathbf{j}\,\in\,I_{N}} G_{\mathbf{j}} \,e^{2i\pi\,\frac{\mathbf{n}\cdot\mathbf{j}}{N}} \, \widehat{(\mathfrak{g}_{R,p,\mathbf{j}} \times G_{h})}_{\mathbf{n}} \, \widehat{(\mathfrak{g}_{R,p,\mathbf{j}}' \times G_{h})}_{\mathbf{n}} \, ,
\end{split}
\end{displaymath}
where $\mathfrak{g}_{R,p,\mathbf{j}}$ and $\mathfrak{g}_{R,p,\mathbf{j}}'$ are obtained with inverse Fourier transform of $\widehat{\mathfrak{g}_{R,p,\mathbf{j}}}$ and $\widehat{\mathfrak{g}_{R,p,\mathbf{j}}'}$ respectively. \\

  \bibliographystyle{acm}
  \bibliography{biblio}
\end{document}